\documentclass[preprint,12pt]{elsarticle}

\usepackage{lineno}

\usepackage[utf8]{inputenc}
\usepackage[ngerman,english]{babel}
\usepackage[T1]{fontenc}
\usepackage{marvosym} 

\usepackage{graphicx}
\usepackage{caption}
\usepackage{subcaption}
\usepackage{xcolor}
\usepackage{tabularx}

\setkeys{Gin}{draft=false}

\usepackage[final]{animate}

\usepackage{mathtools}
\usepackage{amssymb}
\usepackage{amsfonts}
\usepackage{amsbsy} 
\usepackage{amsopn} 
\usepackage{amstext} 
\usepackage{amsxtra} 
\usepackage{bbm} 

\usepackage{cancel} 
\usepackage{amscd} 
\usepackage[all]{xy} 
\usepackage{esint} 
\usepackage{empheq} 

\usepackage{mdframed}
\usepackage{eso-pic,picture}
\usepackage{lscape}
\usepackage{blindtext}
\usepackage{listings}
\usepackage{url}
\usepackage{textcomp} 
\usepackage{xspace} 
\usepackage{scalerel} 
\usepackage{stackengine} 
\usepackage{hhline}
\usepackage{lineno}
\usepackage{fancyhdr}
\usepackage{longtable}
\usepackage{wallpaper}
\usepackage{pdfpages}

\usepackage{siunitx} 
\usepackage{hyphenat} 
\usepackage{testhyphens} 

\input ushyphex
\begin{hyphenrules}{ngerman}
\end{hyphenrules}

\PassOptionsToPackage{svgnames,dvipsnames}{xcolor}
\usepackage{tcolorbox}
\tcbuselibrary{skins,breakable}

\mathcode`@="8000
{\catcode`\@=\active\gdef@{\mkern1mu}}

\newcolumntype{L}{>{$}l<{$}}
\newcolumntype{C}{>{$}c<{$}}
\newcolumntype{R}{>{$}r<{$}}

\usepackage{lmodern}
\usepackage{microtype}

\usepackage{mathrsfs} 

\usepackage[numbers]{natbib}

\usepackage[standard,amsmath,thmmarks]{ntheorem}

\theoremstyle{break}
\theorembodyfont{\normalfont}
\theoremsymbol{{\scriptsize\ensuremath{\diamondsuit}}}

\theoremstyle{nonumberplain}
\theoremseparator{:}
\theoremsymbol{\ensuremath{\square}}
\renewtheorem{proof}{Proof}

\addto\captionsenglish{}

\usepackage[mode=multiuser]{fixme} 
\fxusetheme{color}

\RequirePackage[l2tabu]{nag}

\usepackage{varioref} 
\usepackage[colorlinks=true,citecolor=blue,
urlcolor=blue,linkcolor=blue]{hyperref} 
\usepackage[capitalise]{cleveref} 

\let\ccref\cref
\renewcommand{\cref}[1]{\mbox{\ccref{#1}}}
\crefname{appendix}{}{}
\crefname{listing}{Algorithm}{Algorithms}
\crefname{equation}{}{}
\crefname{table}{Table~}{Tables~}

\graphicspath{{img/}}

\journal{}

 \newcommand{\D}{\mathrm{d}}

  \newcommand{\pd}[2]{\frac{\partial #1}{\partial #2}}

  \newcommand{\td}[2]{\frac{\mathrm d #1}{\mathrm d #2}}

\newcommand{\Q}{\mathbb{Q}}
\newcommand{\R}{\mathbb{R}}
\newcommand{\N}{\mathbb{N}}

\newcommand{\dx}{\mathrm{d}\mathbf{x}}
\newcommand{\ds}{\mathrm{d}\mathbf{s}}

\newcommand{\Ne}{\mathcal N_e}
\newcommand{\Neb}{\Ne^{\partial}}
\newcommand{\Nei}{\Ne^0}
\newcommand{\Be}{\mathcal B_e}

\newcommand{\beq}{\begin{equation}}
\newcommand{\eeq}{\end{equation}}

\newcommand{\ab}[2]{\mbox{#1.\,#2.,}~}
\newcommand{\eg}{\ab{e}{g}}
\newcommand{\etal}{\mbox{et al.\ }}

\newcommand{\ie}{, \ab{i}{e}}
\newcommand{\wrt}{\mbox{w.\,r.\,t.}~}

\newcommand{\IP}{\ensuremath{\mathbb{P}}}

\renewcommand{\l}{\ensuremath{\left(}}
\renewcommand{\r}{\ensuremath{\right)}}

\newcommand\red[1]{#1}
\newcommand\blue[1]{#1}

\begin{document}

\begin{frontmatter}



  \title{\blue{An element-based} convex limiting framework for continuous Galerkin methods with nonlinear stabilization}



\author{Dmitri Kuzmin\corref{cor1}}
\ead{kuzmin@math.uni-dortmund.de}

\cortext[cor1]{Corresponding author}

\author{Hennes Hajduk}
\ead{hennes.hajduk@math.tu-dortmund.de}

\author{Joshua Vedral}
\ead{joshua.vedral@math.tu-dortmund.de}

\address{Institute of Applied Mathematics (LS III), TU Dortmund University\\ Vogelpothsweg 87,
  D-44227 Dortmund, Germany}


\begin{abstract}
  We equip a high-order continuous Galerkin discretization of a general hyperbolic problem with a nonlinear stabilization term and introduce a new methodology for enforcing preservation of invariant domains. The amount of shock-capturing artificial viscosity is determined by a smoothness sensor that measures deviations from a weighted essentially nonoscillatory (WENO) reconstruction. Since this kind of dissipative stabilization does not guarantee that the nodal states of the finite element approximation stay in a convex admissible set, we adaptively constrain deviations of these states from intermediate cell averages. The representation of our scheme in terms of such cell averages makes it possible to apply convex limiting techniques originally designed for positivity-preserving discontinuous Galerkin (DG) methods. Adapting these techniques to the continuous Galerkin setting and using Bernstein polynomials as local basis functions, we prove the invariant domain preservation property under a time step restriction that can be significantly weakened by using a flux limiter for the auxiliary cell averages. The close relationship to DG-WENO schemes is exploited and discussed. All algorithmic steps can be implemented in a matrix-free and hardware-aware manner. The effectiveness of the new element-based limiting strategy is illustrated by numerical examples.
\end{abstract}



\end{frontmatter}


\section{Introduction}

Finite element methods based on continuous Galerkin (CG) approximations are rarely used for discretization of nonlinear hyperbolic problems. It is commonly believed that finite volume and discontinuous Galerkin (DG) methods are better suited for this purpose because of the ease with which they can be manipulated to satisfy physical and numerical admissibility conditions. In particular, the use of slope limiters makes it possible to enforce preservation of local and/or global bounds for scalar quantities of interest \cite{barth1989,kuzmin2010,zhang2011}.
Since slope limiting does not preserve global continuity, high-resolution finite element schemes of CG type are usually stabilized using nonlinear artificial viscosity \cite{guermond2018,guermond2011,hughes2010,johnson1990} or algebraic flux correction (AFC) tools \cite{kuzmin2023,kuzmin2012b}.

Many classical AFC schemes for low-order (linear or multilinear) continuous finite elements belong to the family of flux-corrected transport (FCT) methods \cite{kuzmin2010a,lohner1987}. In the scalar case, flux limiting can be performed using Zalesak's algorithm \cite{zalesak1979} or a localized limiter for individual edge or element contributions \cite{cotter2016,guermond2017,lohmann2017}. Both approaches can be generalized to systems in a manner that guarantees the validity of positivity constraints for quasi-concave scalar functions of conserved variables \cite{dobrev2018,guermond2018,lohmann2016}. Numerical schemes that provably satisfy such inequality constraints are called invariant domain preserving (IDP) \cite{guermond2016}
or, informally, positivity preserving \cite{zhang2011}.

The localized FCT method proposed by Guermond et al. \cite{guermond2018}
constrains forward Euler stages of an explicit strong stability preserving
Runge--Kutta (SSP-RK) method using \emph{convex limiting}. This terminology
reflects the fact that the nodal states of the flux-corrected CG
approximation represent convex combinations of admissible states.
The IDP property of the low-order (local
Lax--Friedrichs) method, and of the FCT extension, is guaranteed
under a CFL-like time step restriction. The \emph{monolithic}
convex limiting (MCL) procedure introduced in \cite{kuzmin2020} can
be combined with any time stepping method \cite{kuzmin2022a,moujaes2025,quezada-de-luna2022} and has been successfully extended to high-order
finite elements \cite{kuzmin2020f,kuzmin2020a}. 
Convex limiting of FCT or MCL type is also possible
for nodal DG methods using Bernstein or
{L}egendre--{G}auss--{L}obatto (LGL) finite elements of arbitrary order
\cite{hajduk2021,hajduk2025a,lin2024,pazner2021,rueda-ramirez2024}. However,
achieving optimal accuracy and performance
with such AFC schemes 
typically requires the use of compact-stencil representations,
subcell flux limiters, and local bounds depending on \emph{ad hoc}
smoothness indicators (as shown already in
\cite{hajduk2020,hajduk2020b,kuzmin2020a,lohmann2017}).

A powerful alternative framework for the design of positivity-preserving
high-order DG methods was developed by Zhang and Shu
\cite{zhang2010b,zhang2011,zhang2012}. In a typical implementation 
of their approach, which builds on the seminal paper \cite{perthame1996}
by Perthame and Shu,
spurious oscillations are suppressed by using a
weighted essentially nonoscillatory (WENO) reconstruction as a
substitute for an unacceptable DG approximation in a troubled
cell \cite{luo2007,qiu2005,shu2016,zhong2013,zhu2009,zhu2008,zhu2017}.
Positivity preservation can then be enforced using slope
or flux limiting. The Zhang--Shu slope limiter 
ensures positivity at the quadrature points corresponding
to a convex decomposition of the cell average into states that are
evolved using a one-dimensional IDP scheme \cite{kuzmin2023,zhang2011}.
Flux limiting makes it possible to constrain the
cell averages directly \cite{moe2017,xiong2016,xu2014}. In
particular, FCT and MCL algorithms can be used for this purpose 
\cite{guermond2019,kuzmin2021,kuzmin2023}.

\blue{The objective of this work is to show that the methodology
developed by Zhang and Shu can be adapted to continuous
finite element approximations.} The first step toward that end
has already been made in \cite{kuzmin2023a}, where we
designed nonlinear stabilization terms for CG discretizations
using shock detectors that depend on deviations from a Hermite
WENO reconstruction. In contrast to traditional WENO limiters,
our approach is suitable for baseline discretizations of CG type
and leads to analyzable nonlinear schemes \cite{vedral2025}. It
remains to design a limiter that ensures positivity preservation.
The convex limiting strategy that we propose below is
based on a representation of the CG approximation in terms of
constrained nodal states. An element-based slope limiter controls
deviations of these states from intermediate cell averages
that are positivity preserving under a restrictive CFL condition.
The optional use of a flux limiter enables us to guarantee
positivity preservation under the mild CFL condition of the
piecewise-constant DG-$\mathbb{P}_0$ approximation. Moreover,
all calculations can be performed in a matrix-free manner, i.e.,
without calculating the full element matrices of the high-order
finite element~space. 

In Section~\ref{sec:cgweno}, we discretize a generic hyperbolic
problem in space using Bernstein finite elements and
the CG-WENO method introduced in \cite{kuzmin2023a}.
As we show in Section~\ref{sec:low}, the stabilized baseline
scheme can be
decomposed into a low-order IDP approximation and a sum of
element contributions that
may require limiting. Representing the low-order part in terms of
intermediate cell averages, we formulate sufficient conditions
for explicit SSP-RK time discretizations to be IDP. 
Monolithic limiters for auxiliary cell averages and
auxiliary nodal states are presented in Section~\ref{sec:limiters}.
Finally, we perform numerical studies and draw conclusions 
in Sections \ref{sec:num} and \ref{sec:conclusions},
respectively.

\section{Stabilized CG discretization}
\label{sec:cgweno}

Let $u(\mathbf{x},t)$ denote a conserved quantity or a vector of conserved
variables depending on the space location $\mathbf{x}$ and time instant $t\ge 0$.
Restricting our attention to a bounded domain $\Omega\subset\R^d,\ d\in\{1,2,3\}$
and choosing a flux function $\mathbf{f}(u)$, we consider the
(possibly nonlinear) hyperbolic problem
\begin{subequations}\label{ibvp}
\begin{alignat}{3}
  \pd{u}{t}+\nabla\cdot\mathbf{f}(u)&=0 &&\quad\mbox{in}\ \Omega\times (0,T],
    \label{ibvp-pde}\\
    u(\cdot,0)&=u_0 &&\quad\mbox{in}\ \Omega,\label{ibvp-ic}\\
    \mathbf f(u)\cdot\mathbf n&=\mathcal F(u,\hat u;\mathbf n)
    &&\quad \mbox{on}\ \partial\Omega\times [0,T],
\end{alignat}
\end{subequations}
where $u_0$ is the initial data, $T>0$ is the final time, and $\hat u$ is
the data of a weakly imposed boundary condition. We denote by
$\mathcal F(u_L,u_R;\mathbf n)$ a boundary flux in the direction of the unit
outward normal  $\mathbf n$. The states $u_L$ and $u_R$ represent
the initial states of the corresponding Riemann problem \cite{kuzmin2023}.

An invariant domain of problem \eqref{ibvp} is a convex admissible
set $\mathcal G$ such that $u(\mathbf{x},t)\in\mathcal G$ for all
space-time locations $(\mathbf{x},t)$. In particular, this definition
implies that $u_0(\mathbf x)\in\mathcal G$ for all $\mathbf x\in\Omega$ 
and $\hat u(\mathbf{x},t)\in\mathcal G$ for all $(\mathbf x,t)\in\partial\Omega\times [0,T]$.

\subsection{Galerkin scheme using Bernstein finite elements}

To solve \eqref{ibvp} using a CG-type finite element method, we use a
conforming, possibly unstructured,
mesh $\mathcal T_h$ consisting of simplices or $d$-dimensional
boxes $K_e,\ e=1,\ldots,E_h$. For simplicity, we assume that $\bar\Omega
=\bigcup_{e=1}^{E_h}K_e$.
A~globally continuous approximation $u_h$ is
sought in a space $V_h$ that is spanned by nodal basis functions
$\varphi_1,\ldots,\varphi_{N_h}\in C(\bar\Omega)$ of polynomial degree $p\in \N$. The
corresponding nodal points are denoted by $\mathbf x_1,\ldots,\mathbf x_{N_h}$.
The basis functions $\varphi_i$ of the Bernstein polynomial basis
are nonnegative and form a partition of unity, i.e., $\sum_{i=1}^{N_h}\varphi_i\equiv 1$.
For a formal definition of $\varphi_i$ and a review of further
properties, we refer the reader to \cite{hajduk2021,kuzmin2020f,kuzmin2020a,lohmann2017}
and \cite[Sec. 6.1]{kuzmin2023}.

Adopting the Bernstein basis representation, we store the global indices
of nodes belonging to a macrocell $K_e$ in the integer
set $\Ne$ such that $$u_h|_{K_e}=\sum_{j\in\Ne}u_j\varphi_j,\qquad e=1,\ldots,E_h.$$
The indices of macrocells containing a nodal point
$\mathbf x_i$ are stored in the integer set $\mathcal E_i$. The
boundary $\partial K_e$ of $K_e$ consists of facets $S_{ee'}$. We define $\Be$ as the set of cell indices such that $S_{ee'}=\partial K_e\cap \partial K_{e'}$
for $e'\in\{1,\ldots,E_h\}\backslash\{e\}$. A~unique index $e'>E_h$ is
associated with each boundary facet $S_{ee'}\subset\partial\Omega$. The
set $\Be^\partial=\{e'\in\Be\,:\, e'>E_h\}$ contains the indices of such facets.
\red{The index notation that we use is summarized in Table~\ref{table:sets}.
\begin{table}[ht!]
\centering
\scriptsize\red{
\begin{tabular}{ll}
  $e\in\{1,\ldots,E_h\}$ & index of a mesh element $K_e$ belonging to $\mathcal T_h$ \\
  $e'\in\Be$ & index of a facet $\begin{cases}
  \mbox{
    $S_{ee'}=\partial K_e\cap \partial K_{e'}$  if $e'\le E_h$} \\
  \mbox{
    $S_{ee'}\subset\partial K_e\cap\partial\Omega$ if $e'>E_h$}
  \end{cases}$
  \\
  $e'\in\Be^\partial$ & index of a boundary facet $S_{ee'}$ with $e'> E_h$\\
  $i\in\{1,\ldots,N_h\}$ & index of a nodal point $\mathbf x_i$ or
  basis function $\varphi_i$ \\
  $i\in \Ne$ & index of a nodal point $\mathbf x_i\in K_e$  \\
  $i\in\Neb$ & index of a boundary node $\mathbf x_i\in\partial K_e$\\
  $i\in\Nei$ & index of an internal node $\mathbf x_i\in K_e\backslash\partial K_e$\\
  $e\in\mathcal E_i$ & index of a  mesh element $K_e$ containing $\mathbf x_i$\\
  $j\in\mathcal J_i$ & index of a point $\mathbf x_j\in K_e$ for some $e\in
  \mathcal E_i$
\end{tabular}}
\caption{\red{Index notation used in descriptions of numerical schemes.}}\label{table:sets}
\end{table}}

The standard CG approximation $u_h\in V_h$ to a weak solution of \eqref{ibvp}
satisfies 
\begin{align*}
\sum_{e=1}^{E_h}\int_{K_e}v_h\Big[\pd{u_h}{t}&+\nabla\cdot\mathbf{f}(u_h)\Big]\dx\\
&+\sum_{e=1}^{E_h}\int_{\partial K_e\cap\partial\Omega}v_h[\mathcal F(u_h,\hat u_h;\mathbf n)
  -\mathbf f(u_h)\cdot\mathbf n]\ds=0
\end{align*}
for any test function $v_h\in V_h$. Introducing the notation 
\begin{align}
  (v_h,u_h)&=\sum_{e=1}^{E_h}\int_{K_e}v_hu_h\dx,\\
  a(v_h,u_h)&=\sum_{e=1}^{E_h}\int_{K_e}v_h\nabla\cdot\mathbf{f}(u_h)\dx,\\
  b(v_h,u_h,\hat u_h)&=
  \sum_{e=1}^{E_h}\int_{\partial K_e\cap\partial\Omega}
  v_h[\mathcal F(u_h,\hat u_h;\mathbf n)-\mathbf f(u_h)\cdot\mathbf n]\ds,
\end{align}
we write the semi-discrete problem in the form
\beq\label{weak:galerkin}
\td{}{t}(v_h,u_h)+a(v_h,u_h)+b(v_h,u_h,\hat u_h)=0\qquad \forall v_h\in V_h.
\eeq
The use of $v_h\in\{\varphi_1,\ldots,\varphi_{N_h}\}$ in
\eqref{weak:galerkin} yields a system of semi-discrete nonlinear
equations for the nodal states $u_j(t)$ that define $u_h=\sum_{j=1}^{N_h}u_j
\varphi_j$. Note that, in contrast to Lagrange finite elements,
the identity $u_j(t)=u_h(\mathbf x_j,t)$ generally holds only if 
the nodal point $\mathbf x_j$ is a vertex of the mesh $\mathcal T_h$.

\subsection{Dissipative WENO stabilization}

To achieve optimal convergence and avoid spurious oscillations
within the global bounds of IDP  constraints, we add a nonlinear
stabilization term
$$s_h(v_h,u_h)=\sum_{e=1}^{E_h}s_h^e(v_h,u_h)$$
on the left-hand side of \eqref{weak:galerkin}. The
CG-WENO methods presented in \cite{kuzmin2023a,vedral2025}
blend the local bilinear forms of high- and low-order
stabilization operators using a smoothness sensor
$\gamma_e\in[0,1]$. The element contributions
\begin{align}
s_h^e(v_h,u_h)&=\nu_e\int_{K_e}
\nabla v_h\cdot(\nabla u_h-\gamma_e\mathbf{g}_h)\dx
\label{stab:WENO}
\end{align}
are defined using the consistent $L^2$ projection $\mathbf{g}_h$
of $\nabla u_h$ into the CG space. The 
viscosity parameter $\nu_e=\frac{\lambda_e h_e}{2p}$ depends on the
local mesh size $h_e$, polynomial degree $p$, and a local bound
$\lambda_e$ for the maximum speed. Details regarding the calculation of $\mathbf{g}_h$
and $\lambda_e$ can be found in \cite{kuzmin2020f,kuzmin2023a,lohmann2017,vedral2025}.

In formula \eqref{stab:WENO}, the smoothness sensor $\gamma_e$ acts as
a slope limiter for the projected gradient $\mathbf{g}_h$. In the case
$\gamma_e=0$, the local bilinear form \eqref{stab:WENO} introduces
low-order stabilization of local Lax--Friedrichs type. The setting
$\gamma_e=1$ corresponds to the linear high-order stabilization
employed in \cite{kuzmin2020f,lohmann2017}. The symmetry and
positive semi-definiteness of the corresponding bilinear form
can be shown as in \cite[Lem. 4.1]{olshanskii2025}. The
WENO shock detector
\begin{align}
  \gamma_e = 1 - \min\bigg(1,\frac{\|u_h^e-u_h^{e,*}\|_e}{\|u_h^e\|_e}\bigg)^q
  \label{WENO:alpha}
\end{align}
proposed in \cite{kuzmin2023a} measures deviations of $u_h^e=u_h|_{K_e}$ from a
WENO reconstruction $u_h^{e,*}$. The sensitivity of $\gamma_e$ to these deviations 
can be varied by adjusting the parameter $q\ge 1$. \red{As $q$ is increased, the resolution of steep gradients improves, but violations of maximum principles may occur / become more pronounced (see Section~\ref{sec:scalar2D}).} The Sobolev semi-norm $\|\cdot\|_e$ is defined
by
  \begin{align}
  \|v\|_{e}=\left(\sum_{1\leq|\mathbf{k}|\leq p}
  h_e^{2|\mathbf{k}|-d}\int_{K_e}|D^\mathbf{k}v|^2\dx
    \right)^{1/2}\qquad
  \forall v\in H^p(K_e),
  \label{WENO:norm}
\end{align}
where $\mathbf{k}=(k_1,\ldots,k_d)$ is the
  multiindex of the partial derivative
  $D^\mathbf{k}v$.
  
 Following the design of Hermite WENO limiters for DG methods
 \cite{luo2007,zhu2009,zhu2017},
 the reconstruction $u_h^{e,*}$ is defined as a convex
 combination of candidate polynomials $u_{h,l}^e,\ l=0,\ldots,n_e$ with
 nonlinear weights \cite{jiang1996}
\[
\omega_l^e=\frac{\tilde{w}_l^e}{\sum_{k=0}^{n_e}\tilde{w}_k^e}, \qquad
\tilde{w}_l^e = \frac{w_l^{e,\rm lin}}{\|u_{h,l}^e\|_e^r+\epsilon}
\]
that depend on the choice of the linear weights $w_l^{e,\rm lin}$.
The small parameter $\epsilon >0$
is used to avoid division by zero. \red{
The exponent $r\ge 1$ determines the sensitivity of the nonlinear weights to the smoothness indicators. Larger values of the parameter $r$ increase the difference between the weights of smooth and nonsmooth candidates but can degrade the accuracy in smooth regions. Liu et al.~\cite{liu1994} suggest setting $r$ equal to the polynomial degree of the WENO reconstruction. Jiang and Shu \cite{jiang1996} found that $r=2$ achieves the best balance between robustness and accuracy. This is currently the default choice in classical WENO schemes. In our implementation, the} candidate
polynomials $u_{h,l}^e$ are constructed using Hermite interpolation
from cells $K_{e'}$ sharing a facet $S_{ee'}$ with $K_e$.
For details, we refer the reader to \cite{kuzmin2023a,vedral2025}.

\red{The design of the local stabilization term
 \eqref{stab:WENO} is motivated by the fact that deviations of $u_h^e$ from
 $u_h^{e,*}$ are large in regions where $u_h$ is not smooth. In these regions, the value of the WENO sensor $\gamma_e$ is small, and \eqref{stab:WENO} introduces large amounts of low-order numerical viscosity.
 In smooth regions, the value of $\gamma_e$ is close to unity. The high-order stabilization corresponding to $\gamma_e=1$ ensures optimal convergence to smooth exact solutions \cite{kuzmin2020f,lohmann2017}.}

The use of the test function $v_h=\varphi_i$ in the stabilized WENO version
\beq\label{weak:stab}
\td{}{t}(v_h,u_h)+a(v_h,u_h)+
b(v_h,u_h,\hat u_h)+s_h(v_h,u_h)=0\qquad \forall v_h\in V_h
\eeq
of the spatial semi-discretization \eqref{weak:galerkin} yields the evolution equation
\beq\label{weak:stab2}
(\varphi_i,\dot u_h)+a(\varphi_i,u_h)+
  b(\varphi_i,u_h,\hat u_h)+s_h(\varphi_i,u_h)=0,
\eeq
where $\dot u_h=\sum_{j=1}^{N_h}\dot u_j\varphi_j$ denotes the time derivative of $\dot u_h$.
  Discretization in time can be performed, e.g., using an explicit SSP-RK method.

\section{Low-order IDP discretization}
\label{sec:low}

In many cases, the results produced by the fully discrete version of the 
CG-WENO scheme \eqref{weak:stab} are excellent \cite{kuzmin2023a,vedral2025}.
However, there is no guarantee that the nodal states $u_j$ of the Bernstein
finite element approximation are positivity preserving / IDP. To show that the
IDP property can be enforced using monolithic convex limiting, we first
approximate $\mathbf f(u_h)$ by
\beq\label{gfe}
\mathbf f_h=\sum_{j=1}^{N_h}\mathbf f(u_j)\varphi_j
\eeq
and define the low-order counterpart
\beq\label{weak:low1}
m_i\td{u_i}{t}+a_i(u_h)+s_i(u_h)+b_i(u_h,\hat u_h)=0
\eeq
of the semi-discrete equation \eqref{weak:stab2} using (cf. \cite[Ch. 3 and 6]{kuzmin2023}, \cite{kuzmin2025a})
\begin{align}
  m_i&=\sum_{e\in\mathcal E_i}m_i^e,\qquad m_i^e=\int_{K_e}\varphi_i\dx,\\
  a_i(u_h)&=\sum_{e\in\mathcal E_i}\frac{m_i^e}{|K_e|}\int_{K_e}\nabla\cdot\mathbf{f}_h\dx
  =\sum_{e\in\mathcal E_i}\frac{m_i^e}{|K_e|}\sum_{e'\in\Be}\int_{S_{ee'}}\mathbf{f}_h\cdot\mathbf n\ds,\\
  s_i(u_h)&=\sum_{e\in\mathcal E_i}s_i^e(u_h),\qquad
  s_i^e(u_h)=\frac{m_i^e}{\Delta t_e}(u_i-u^e),\label{rusdiss}\\
b_i(u_h,\hat u_h)&=
\sum_{e\in\mathcal E_i}\sum_{e'\in\Be^{\partial}}\sigma_{i,ee'}
  [\mathcal F(u_i,\hat u_{i,ee'};\mathbf n)-\mathbf f(u_i)\cdot \mathbf{n}],
  \label{bidef}\\
  \sigma_{i,ee'}&=\int_{S_{ee'}}\varphi_i\ds,\qquad
\hat u_{i,ee'}
=
\frac{1}{\sigma_{i,ee'}}\int_{S_{ee'}}\varphi_i\hat u_h\ds.
\label{sigmadef}
\end{align}
The stabilization term $s_i(u_h)$ introduces low-order
\emph{Rusanov dissipation} (cf. \cite{abgrall2006},
\cite[Ch. 4]{kuzmin2023}), the levels of which are
inversely proportional to a `fake' time step~$\Delta t_e$.
An IDP upper bound for $\Delta t_e$ is derived below.
The nonlinear term $a_i(u_h)$ of the evolution equation \eqref{weak:low1}
approximates $a(\varphi_i,u_h)$ using 
\eqref{gfe} and inexact quadrature
for the integral $\int_{K_e}\varphi_i\nabla\cdot\mathbf f_h\dx$.
The approximation of $(\varphi_i,u_h)$ by $m_iu_i$ corresponds to mass lumping, which 
is equivalent to using inexact quadrature for $L^2$ scalar products.
For Bernstein finite elements
$m_i^e=\frac{|K_e|}{|\Ne|}$, where $|\Ne|$ is the cardinality
of the index set $\Ne$ and $|K_e|$ is the $d$-dimensional volume of $K_e$.
Similar mass lumping is performed in the boundary term~\eqref{bidef}, in
which we use
the local Lax--Friedrichs (LLF) flux
\beq\label{fluxLLF}
\mathcal F(u_L,u_R;\mathbf n)=\frac{\mathbf f(u_L)+\mathbf f(u_R)}2\cdot\mathbf n-
\frac{\lambda_{LR}}2(u_R-u_L)
\eeq
depending on the maximum speed $\lambda_{LR}$ of the corresponding Riemann problem. 

\subsection{Intermediate cell averages}

Let us first analyze the low-order scheme \eqref{weak:low1} under the simplifying assumption of periodic boundary conditions. Since $b_i(u_h,\hat u_h)=0$ in the
periodic case, the ordinary differential equation \eqref{weak:low1} can
be written as
\beq\label{weak:low2}
m_i\td{u_i}{t}=\sum_{e\in\mathcal E_i}m_i^e\frac{\bar u^e-u_i}{\Delta t_e},
  \eeq
  where $\Delta t_e$ is the \red{pseudo} time step of the Rusanov dissipation
  \eqref{rusdiss} and
\beq\label{uebar}
\bar u^e=u^e-\frac{\Delta t_e}{|K_e|}
\sum_{e'\in\Be}\int_{S_{ee'}}\mathbf f_h\cdot\mathbf n\ds
\eeq
is an intermediate cell average depending on the old cell average
\beq\label{ueold}
u^e=\frac{1}{|K_e|}\int_{K_e}u_h\dx=\frac{1}{|K_e|}\sum_{i\in\Ne}m_i^eu_i.
\eeq

If \eqref{weak:low2} is discretized in time using an explicit SSP-RK method,
then each intermediate stage is a forward Euler update of the form
\beq\label{fe-update}
m_iu_i^{\rm SSP}=\sum_{e\in\mathcal E_i}m_i^e\left[\left(1-\frac{\Delta t}{\Delta t_e}\right)u_i
  +\left(\frac{\Delta t}{\Delta t_e}\right)\bar u^e\right],
\eeq
where $\Delta t>0$ is a global time step such that
\beq\label{cflglob}
\Delta t\le \min_{1\le e\le E_h}\Delta t_e.
\eeq
Under this condition, \eqref{fe-update} yields a convex
combination $u_i^{\rm SSP}$ of the states $u_i$ and $\bar u^e,\ e\in\mathcal E_i$.
If these states belong to a convex admissible set $\mathcal G$, then so
does  $u_i^{\rm SSP}$ and, therefore, the SSP-RK update is IDP w.r.t.
$\mathcal G$.

\red{
\subsection{Analysis in the one-dimensional case}
  
For an element $K_e$ of a 1D mesh with $\Ne=\{i,j\}$
and local spacing
$h_e=|K_e|$,
the intermediate state $\bar u^e$ defined by \eqref{uebar} reduces to
  \beq\label{ubar1D}
  \bar u_{ij}=\frac{u_i+u_j}2-(j-i)
  \frac{\Delta t_e}{h_e}[f(u_j)-f(u_i)]
  =\bar u_{ji}.
  \eeq 
  Let $\lambda_{ij}^{\max}$ denote the maximum speed of the 1D Riemann
  problem with the initial states $u_i$ and $u_j$. If the
  pseudo time step $\Delta t_e$ satisfies 
 \beq\label{cfl1D}
\frac{\lambda_{ij}^{\max}\Delta t_e}{h_e}\le\frac12,
  \eeq
  then the \emph{bar state} $\bar u_{ij}$ defined by \eqref{ubar1D} corresponds to a spatially averaged exact solution to the Riemann problem \cite{guermond2016,harten1983a}. Thus, $u_i,u_j\in\mathcal G$ implies $\bar u_{ij}\in
  \mathcal G$ for every invariant domain $\mathcal G$ of the hyperbolic
  problem.

The 1D version of
\eqref{fe-update} can be rewritten in terms of $\bar u_{ij}$ as follows:
\beq\label{festepbar}
u_i^{\rm SSP}=\Big(1-\frac{\Delta t}{m_i}\sum_{j\in\mathcal J_i\backslash\{i\}}
2d_{ij}\Big)u_i+\frac{\Delta t}{m_i}\sum_{j\in\mathcal J_i\backslash\{i\}}
2d_{ij}\bar u_{ij}.
\eeq
Here $\mathcal J_i= \mathcal N_{e'}\cup \mathcal N_{e}$ for
$\mathcal N_{e'}=\{i-1,i\}$ and $\mathcal N_e=\{i,i+1\}$.
The auxiliary coefficients $d_{i,i-1}=m_i^{e'}/(2\Delta t_{e'})$ and $d_{i,i+1}=m_i^{e}/(2\Delta t_{e})$ are
positive.

\begin{lemma}\label{lem:IDP1D}
  Let $\mathcal G$ be a convex invariant domain. Suppose that 
  \beq\label{idpcrit1d}
  u_i,u_j\in\mathcal G
  \quad\Rightarrow\quad \bar u_{ij}\in\mathcal G
\eeq
and the time step $\Delta t$ satisfies the CFL-like condition
$$\frac{\Delta t}{m_i}\sum_{j\in\mathcal J_i\backslash\{i\}}
2d_{ij}\le 1.$$
Then the forward Euler step \eqref{festepbar} is IDP in the sense that
$$
u_j\in\mathcal G\ \forall j\in  \mathcal J_i
\quad\Rightarrow\quad u_i^{\rm SSP}\in\mathcal G.
$$
  \end{lemma}
\begin{proof}
  Under the assumptions of the lemma, $u_i^{\rm SSP}$ is a convex
  combination of $u_i\in\mathcal G$ and $\bar u_{ij}\in\mathcal G,\ j\in
  \mathcal J_i\backslash\{i\}$.
  It follows that $u_i^{\rm SSP}\in\mathcal G$.
\end{proof}  

The assumption \eqref{idpcrit1d} of the above lemma is met for $\bar u_{ij}$
of the form \eqref{ubar1D} if the local CFL condition $\eqref{cfl1D}$
holds for the pseudo time step $\Delta t_e$.

\subsection{Analysis in the multidimensional case}

Guermond and Popov \cite{guermond2016} found that a representation in
the form \eqref{festepbar} is also possible for a multidimensional finite
element counterpart of the local Lax--Friedrichs method. This low-order
IDP scheme uses the coefficients $d_{ij}$ of a graph Laplacian operator for
algebraic stabilization purposes. High-order finite element approximations
can also be written in the form \eqref{festepbar} with $\bar u_{ij}$ replaced by
$\bar u_{ij}^H=\bar u_{ij}+f_{ij}/(2d_{ij})$. The raw antidiffusive fluxes $f_{ij}$
can then be limited in a way that ensures the IDP property.
The monolithic convex limiting strategy introduced in \cite{kuzmin2020} 
produces the best flux approximation
$f_{ij}^*\approx f_{ij}$ such that $\bar u_{ij}^*=\bar u_{ij}+f_{ij}^*/(2d_{ij})
\in\mathcal G$ whenever $u_j\in\mathcal G\ \forall j\in\mathcal J_i$.

In the aforementioned extensions of \eqref{festepbar}, the bar states $\bar u_{ij}$
and their flux-corrected counterparts $\bar u_{ij}^*$ are associated with pairs of neighbor nodes
or cells. Multidimensional intermediate states $\bar u^e$ of our low-order method \eqref{fe-update}
are associated with elements and cannot be interpreted as
averaged solutions of one-dimensional Riemann problems. Indeed, formula
\eqref{uebar} resembles an explicit update for a cell average of a DG
approximation. Such updates are genuinely multidimensional and more
natural in the finite element setting. The representation of the
SSP-RK stage \eqref{fe-update} in terms of the intermediate cell
averages $\bar u^e$ reveals structural similarity to
positivity-preserving DG schemes \cite{zhang2010b,zhang2011,zhang2012}
and makes it possible to handle CG discretizations similarly.

In contrast to the edge-based generalization \cite{guermond2016} of \eqref{festepbar}, the element-based
version \eqref{fe-update} averages over elements with indices
$e\in\mathcal E_i$ rather than nodes with indices $j\in\mathcal J_i\backslash\{i\}$.
The difference between the edge-based and element-based convex decompositions
on a triangular mesh is illustrated in Fig.~\ref{decomp}.}

\begin{figure}[h!]\centering
\begin{subfigure}[b]{0.4\textwidth}
\caption{\red{edge-based}}
\includegraphics[width=\textwidth]{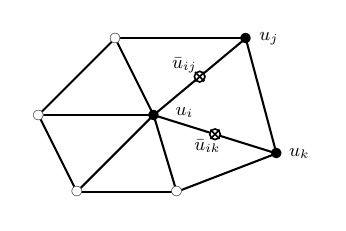}
\end{subfigure}
\begin{subfigure}[b]{0.4\textwidth}
\caption{\red{element-based}}
\includegraphics[width=\textwidth]{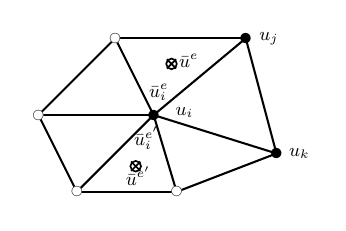}
\end{subfigure}
\caption{\red{Intermediate states of convex decompositions on a triangular mesh.}}
\label{decomp}
\end{figure}

It turns out that the IDP property of $\bar u^e$ is guaranteed if
$\Delta t_e$ is sufficiently small. We state and prove this result in the following lemma.
\begin{lemma}\label{lemma1}
  {\it
  Let $\bar u^e$ be the intermediate cell average defined by \eqref{uebar}. Then
  \beq\label{idpcrit}
u_i,u_j\in\mathcal G\quad \forall j\in\Ne
\quad\Rightarrow\quad \bar u^e\in\mathcal G
\eeq
under a subcell CFL condition of the form
  \beq\label{cflsubcell}
\Delta t_e\le \Delta t_e^{\max}
\eeq
with a computable upper bound $\Delta t_e^{\max}$ that
depends on the local maximum speed and the type of the
finite element approximation on $K_e$.}
\end{lemma}

\red{
\begin{proof}
  Adapting the proof technique employed in \cite[Thm. 5.15]{kuzmin2023}, we will express
  $\bar u^e$ as a convex combination of states that are evolved using the one-dimensional
  LLF scheme and possess the IDP property by Lemma \ref{lem:IDP1D}.

    To begin, we split the index set $\Ne$ into $\Neb=\{i\in\Ne\,:\,\mathbf x_i\in\partial K_e\}$ and
    $\Nei=\Ne\backslash\Neb$. In the
     case $\Nei\ne\emptyset$, we introduce the auxiliary state
\beq
    u_0^e=\frac{1}{m_0^e} \sum_{i\in\Nei}m_i^eu_i,\qquad
    m_0^e=\sum_{i\in\Nei}m_i^e
    \eeq
and represent the old average $u^e$ defined by \eqref{ueold}
as follows:
\beq\label{ueconnvex}
u^e=\frac{m_0^e}{|K_e|}u_0^e+\sum_{i\in \Neb}\frac{m_i^e}{|K_e|}u_i.
\eeq
Recalling the definition \eqref{gfe} of the flux approximation $\mathbf f_h$, we find that
$$
\sum_{e\in\Be}\int_{S_{ee'}}\mathbf f_h\cdot\mathbf n\ds
=\sum_{i\in\Neb}\mathbf f(u_i)\cdot\mathbf{c}_{i,e},
\qquad \mathbf c_{i,e}=\int_{\partial K_e}\varphi_i\mathbf n\ds.
$$
The intermediate cell average \eqref{uebar} admits the convex decomposition
\beq\label{uebarconnvex}
\bar u^e=\frac{m_0^e}{|K_e|}\bar u_0^e+\sum_{i\in \Neb}\frac{m_i^e}{|K_e|}\bar u_i^e,
\eeq
where
\begin{align}
\bar u_i^e&=u_i
-\frac{\Delta t_e}{m_i^e}|\mathbf c_{i,e}|(\mathcal F(u_i,u_i;\mathbf n_{i,e})
-\mathcal F(u_0^e,u_i;\mathbf n_{i,e})),\quad i\in\Neb, \label{uebdr1}\\
\bar u_0^e&=
u_0^e-\frac{\Delta t_e}{m_0^e}
 \sum_{i\in\Ne^\partial}|\mathbf c_{i,e}|
 (\mathcal F(u_0^e,u_i;\mathbf n_{i,e})-\mathcal F(u_0^e,u_0^e;\mathbf n_{i,e}))
 \label{uebdr2}
\end{align}
are intermediate states 
defined using the LLF flux \eqref{fluxLLF} with the normal direction
$\mathbf n_{i,e}=
\frac{\mathbf c_{i,e}}{|\mathbf c_{i,e}|}$. Note that
$\mathcal F(u_i,u_i;\mathbf n)=\mathbf f(u_i)\cdot\mathbf n$ and
\begin{align*}
\mathcal F(u_i,u_i;\mathbf n_{i,e})
-\mathcal F(u_0^e,u_i;\mathbf n_{i,e})&=
\frac{\mathbf f(u_i)-\mathbf f(u_0^e)}2\cdot\mathbf n_{i,e}
+\lambda_{0i}^e\frac{u_i-u_0^e}2,
\\
\mathcal F(u_0^e,u_i;\mathbf n_{i,e})-\mathcal F(u_0^e,u_0^e;\mathbf n_{i,e})
&=
\frac{\mathbf f(u_i)-\mathbf f(u_0^e)}2\cdot\mathbf n_{i,e}
-\lambda_{0i}^e\frac{u_i-u_0^e}2.
\end{align*}
Introducing the one-dimensional IDP bar states
(cf. \cite{guermond2016,hoff1979})
$$
\bar u_{0i}^e=\frac{u_i+u_0^e}2-
\frac{\mathbf f(u_i)-\mathbf f(u_0^e)}{2\lambda_{0i}^e}
\cdot\mathbf n_{i,e}
$$
such that
$$
\frac{\mathbf f(u_i)-\mathbf f(u_0^e)}{2}\cdot\mathbf n_{i,e} 
=\lambda_{0i}^e\left(\frac{u_i+u_0^e}2-\bar u_{0i}^e\right),
$$
we express relations \eqref{uebdr1} and \eqref{uebdr2} as follows (cf. \cite{guermond2016}):
\begin{align}
  \bar u_i^e&=\blue{\Big(1-\frac{\Delta t_e}{m_i^e}
  |\mathbf c_{i,e}|\lambda_{0i}^e\Big)u_i+\frac{\Delta t_e}{m_i^e}
  |\mathbf c_{i,e}|\lambda_{0i}^e\bar u_{0i}^e,\quad i\in\Neb,}\label{uebdr1bar}\\
\bar u_0^e&=\blue{\Big(1-
\frac{\Delta t_e}{m_0^e}
\sum_{i\in\Ne^\partial}|\mathbf c_{i,e}|\lambda_{0i}^e\Big)u_0^e
+\frac{\Delta t_e}{m_0^e}
\sum_{i\in\Ne^\partial}|\mathbf c_{i,e}|\lambda_{0i}^e
\bar u_{0i}^e.} \label{uebdr2bar}
\end{align}
\blue{Note that these updates are of the form \eqref{festepbar} and Lemma \ref{lem:IDP1D}
is applicable.} Let
\beq\label{dtmaxsubcell}
\Delta t_e^{\max}=\frac{1}{s}\min\left\{
\min_{i\in\Neb}\frac{m_i^e}{|\mathbf c_{i,e}|\lambda_{0i}^e},
\frac{m_0^e}{\sum_{i\in\Neb}|\mathbf c_{i,e}|\lambda_{0i}^e}
\right\},
\eeq
where $s\ge 1$. If the upper bound $\Delta t_e^{\max}$ of
condition \eqref{cflsubcell} is given by \eqref{dtmaxsubcell} with $s=1$,
then \blue{the states
$\bar u_i^e,\ i\in\Neb$ and $\bar u_0^e$ are convex combinations of the
IDP states that appear on the right-hand sides of
\eqref{uebdr1bar}
and \eqref{uebdr2bar}, respectively. This proves the claimed IDP property of
the convex combination \eqref{uebarconnvex}.} 

If $\Nei=\emptyset$, which is the case for
$\mathbb P_1/\mathbb Q_1$ and triangular $\mathbb P_2$ elements, we
use $u_0^e=u^e$ and the convex decompositions (cf. \cite[Thm. 5.15]{kuzmin2023})
$$
u^e=\frac{1}{2}\left[u_0^e+\sum_{i\in \Ne}\frac{m_i^e}{|K_e|}u_i\right],\qquad
\bar u^e=\frac{1}{2}\left[\bar u_0^e+\sum_{i\in \Ne}\frac{m_i^e}{|K_e|}\bar u_i^e\right].
$$
The states $\bar u_i^e$ and $\bar u_0^e$ are given by \eqref{uebdr1bar}
and \eqref{uebdr2bar} with $\Delta t_e$ replaced by $2\Delta t_e$. It follows that the
IDP property of $\bar u^e$ is guaranteed for $\Delta t_e\le\Delta t_e^{\max}$,
where $\Delta t_e^{\max}$ is the subcell CFL bound given by \eqref{dtmaxsubcell}
with $s=2$.
\end{proof}  
}

The proof of the lemma can be extended to the case of
non-periodic boundary conditions. If the boundary
term $b_i(u_h,\hat u_h)$ defined by \eqref{bidef} is
included on the right-hand side of \eqref{weak:low2},
a generalization of \eqref{uebdr1} and \eqref{uebdr2}
needs to be
analyzed. Let $\mathcal B_{i,e}$ be the subset of $\Be$
such that 
$\varphi_i$ is not identically zero on 
$S_{ee'}$ with $e'\in\mathcal B_{i,e}$. Introducing
the additional notation
$$
\mathbf c_{i,ee'}=\int_{S_{ee'}}\varphi_i\mathbf n\ds,
\quad\mathbf n_{i,ee'}=\frac{\mathbf c_{i,ee'}}{|\mathbf c_{i,ee'}|},
\qquad u_{i,ee'}=\begin{cases}
u_{i} & \mbox{if}\ e'\le E_h,\\
\hat u_{i,ee'} & \mbox{if}\ e'>E_h
\end{cases}
$$
and assuming that $\Ne^0\ne\emptyset$ for simplicity,
we arrive at \eqref{uebarconnvex} with
\begin{align*}
\bar u_i^e&=u_i
-\frac{\Delta t_e}{m_i^e}\sum_{e'\in\mathcal B_{i,e}}
|\mathbf c_{i,ee'}|(\mathcal F(u_i,u_{i,ee'};\mathbf n_{i,ee'})
-\mathcal F(u_0^e,u_i;\mathbf n_{i,ee'})),\quad i\in\Neb,\\
\bar u_0^e&=
u_0^e-\frac{\Delta t_e}{m_0^e}
 \sum_{i\in\Ne^\partial}\sum_{e'\in\mathcal B_{i,e}}|\mathbf c_{i,ee'}|
 (\mathcal F(u_0^e,u_i;\mathbf n_{i,ee'})-\mathcal F(u_0^e,u_0^e;\mathbf n_{i,ee'})).
\end{align*}
The CFL bounds for these LLF updates can again be derived
as in \cite{guermond2016,kuzmin2023} using representations in terms
of $\bar u_{0i}^e$ and $\bar u_{i,ee'}^e$ with $e'>E_h$. The latter bar state
is defined as in \eqref{ubar1D} but using the Riemann data 
$\hat u_{i,ee'}$ in place of $u_j$. We leave
the calculation of $\Delta t_e^{\max}$ as an exercise to the
reader. The flux-limited version of our method will be IDP under
a simpler and weaker CFL condition.
\medskip

An important practical implication of Lemma~\ref{lemma1} is the
following theorem.
\begin{theorem}\label{theorem1}
  {\it Suppose that the global time step $\Delta t$ satisfies
    condition \eqref{cflglob} and all local time steps $\Delta t_e$
    satisfy \eqref{cflsubcell} with $\Delta t_e^{\max}$
    defined in the proof of Lemma~\ref{lemma1}. Then the fully discrete
    scheme \eqref{fe-update} is IDP for $i=1,\ldots,N_h$.}
  \end{theorem}

\begin{proof}
  The claim is true because 
  \eqref{fe-update} yields a convex combination $u_i^{\rm SSP}$
  of the IDP states $u_i$ and $\bar u^e,\ e\in\mathcal E_i$
  under the assumptions of the theorem.
\end{proof}  

\red{In view of \eqref{rusdiss}, the levels of Rusanov
dissipation built into \eqref{weak:low1}
are proportional to $1/\Delta t_e$.} 
The naive choice 
$\Delta t_e\coloneqq \Delta t\coloneqq
\min_{1\le e'\le E_h}\Delta t_{e'}^{\max}$ leads to an extremely
diffusive approximation of global Lax--Friedrichs type. 
An appropriate choice of time steps satisfying 
\eqref{cflglob} and \eqref{cflsubcell} is given by
\beq
\Delta t_e=\Delta t_e^{\max},\qquad
\Delta t=\min_{1\le e'\le E_h}\Delta t_{e'}.
\eeq
We call $\Delta t_e$ a \emph{\red{pseudo}} time step because, similarly to
graph viscosity coefficients $d_{ij}$ of edge-based LLF
schemes \cite{guermond2016,kuzmin2010a,luo1994}, the
value of $\Delta t_e^{\max}$ is determined at the level
of spatial semi-discretization \eqref{weak:low2} by
the local maximum speed and the distance between the
nodes of the Bernstein finite element.

As the degree $p$ of the polynomial approximation is increased
without changing the size of the macrocell $K_e$,
the IDP bound of Lemma~\ref{lemma1} becomes more
restrictive. To make $\Delta t_e^{\max}$ independent
of $p$, we redefine $\bar u^e$ as follows:
\beq\label{uebar2}
\bar u^e=u^e-\frac{\Delta t_e}{|K_e|}
\sum_{e'\in\Be}|S_{ee'}|\mathcal F(u^e,u^{e'};\mathbf n_{ee'}).
\eeq
That is, we approximate $\mathbf f_h|_{S_{ee'}}$ by the
LLF flux $\mathcal F(u^e,u^{e'};\mathbf n_{ee'})$,
where
$$\mathbf n_{ee'}
=\frac{\mathbf c_{ee'}}{|\mathbf c_{ee'}|},\qquad
\mathbf c_{ee'}=\sum_{i\in\Ne}\mathbf c_{i,ee'}=
\int_{S_{ee'}}\mathbf n\ds.
$$
\red{Importantly, the adjustment of numerical fluxes in the
formula for $\bar u^e$ does not affect the global
conservation property of \eqref{weak:low2}
and \eqref{fe-update}.} At the same time, the
assertions of Lemma~\ref{lemma1} and~
Theorem \ref{theorem1} become valid for
\beq\label{dtmaxcell}
\Delta t_e^{\max}=\frac{|K_e|}{\sum_{e'\in\Be}|S_{ee'}|\lambda_{ee'}}.
\eeq
The intermediate cell averages $\bar u^e$
defined by \eqref{uebar2} with $\Delta t_e\le \Delta t_e^{\max}$
are IDP for $\Delta t_e^{\max}$ given by \eqref{dtmaxcell}.
We omit the formal proof 
because \eqref{uebar2} has the same
structure as the finite volume and DG schemes considered in
\cite{guermond2019,kuzmin2021,kuzmin2023}. Therefore, the
IDP property of $\bar u^e$ is guaranteed under the same
macrocell CFL condition, which imposes the simple bound
\eqref{dtmaxcell} on the time step.

\begin{remark}\label{rem:timestep}
Instead of approximating \eqref{uebar} by \eqref{uebar2},
the time scale $\Delta t_e$ of 
formula \eqref{uebar} can be initialized using
\eqref{dtmaxcell}. If this choice produces an
unacceptable state $\bar u^e\notin\mathcal G$, the value
of $\Delta t_e$ should be decreased as much as necessary
to obtain $\bar u^e\in\mathcal G$. We use this strategy
in the 2D examples of Section~\ref{sec:num}. The IDP property
is guaranteed for $\Delta t_e$ satisfying
\eqref{cflsubcell} with $\Delta t_e^{\max}$ defined in
Lemma~\ref{lemma1}.  Note that, in view of \eqref{cflglob},
the value of the global time step $\Delta t$ is determined
by the smallest local time step. Hence, there is no guarantee
that $\Delta t$ can be increased without
changing the numerical fluxes as in \eqref{uebar2}.
\end{remark}

\section{Monolithic convex limiting}
\label{sec:limiters}

We are now ready to present a constrained high-order extension of the
low-order IDP schemes analyzed in Section~\ref{sec:low}. Adopting the
general design principles of monolithic convex limiting 
\cite{kuzmin2020,kuzmin2020f,kuzmin2020a} and adapting them
to our element-based setting, we consider semi-discrete
schemes of the form
\beq\label{mcl}
m_i\td{u_i}{t}=\sum_{e\in\mathcal E_i}m_i^e\frac{\bar u_i^e-u_i}{\Delta t_e}.
\eeq
The low-order approximation \eqref{weak:low2} uses the
intermediate states $\bar u_i^e=\bar u^e$. In the case of
non-periodic boundary conditions, which we discussed in
Section~\ref{sec:low}, the average LLF flux across
a boundary facet $S_{ee'}\subset\partial\Omega$ is given by 
$$
f_{ee'}^\partial=\frac{1}{|S_{ee'}|}\sum_{j\in\Ne^\partial}|\mathbf c_{j,ee'}|
\mathcal F(u_j,\hat u_{j,ee'};\mathbf n_{j,ee'}),
\qquad e'\in\Be^{\partial}.
$$

Let us blend the averaged high-order flux
\beq\label{hoflux}
f_{ee'}^H=\begin{cases}
\frac{1}{|S_{ee'}|}
\int_{S_{ee'}}\mathbf f_h\cdot\mathbf n\ds
& \mbox{if}\ e'\le E_h,\\
f_{ee'}^\partial
& \mbox{if}\ e'>E_h
\end{cases}
\eeq
and its low-order counterpart
\beq
f_{ee'}^L=\begin{cases}
\mathcal F(u^e,u^{e'};\mathbf n_{ee'})
& \mbox{if}\ e'\le E_h,\\
\mathcal F(u^e,\hat u^{e'};\mathbf n_{ee'})
& \mbox{if}\ e'>E_h
\end{cases}
\eeq
using a correction factor $\alpha_{ee'}=\alpha_{e'e}$ to be defined below.
The resulting flux-limited version of formula \eqref{uebar} reads
(cf.~\cite{kuzmin2021,kuzmin2025a})
\beq\label{uelim}
\bar u^e=u^e-\frac{\Delta t_e}{|K_e|}
\sum_{e'\in\Be}|S_{ee'}|(\alpha_{ee'}f_{ee'}^H+(1-\alpha_{ee'})f_{ee'}^L).
\eeq
If no facet of the cell $K_e$ belongs to a non-periodic boundary, the
choices $\alpha_{ee'}=1$ and $\alpha_{ee'}=0$ correspond to 
\eqref{uebar} and \eqref{uebar2}, respectively.

The high-order scheme 
\eqref{weak:stab2} can also be written in the form
\eqref{mcl} with
$\bar u_i^e=\bar u^{e}+f_i^e/m_i^e$, where $\bar u^{e}$ is
defined by \eqref{uelim} with $\alpha_{ee'}=1\ \forall e'\in\Be$.
The \emph{antidiffusive element contributions} $f_i^e$ that recover
\eqref{weak:stab2}
are given by
\begin{align}\nonumber
f_i^e=m_i^e(u_i-u^e)
&+\Delta t_e\Big[
  \int_{K_e}\nabla \varphi_i\cdot(\mathbf f(u_h)-\mathbf f_h)\dx\\
  &\phantom{-\Delta t_e\Big[}
  -\int_{K_e}\left(\varphi_i-\frac{m_i^e}{|K_e|}\right)
    \nabla\cdot\mathbf f_h\dx-g_i^e\label{aec}\\
  &\phantom{-\Delta t_e\Big[}
  -s_h^e(\varphi_i,u_h)
  -\int_{K_e}\varphi_i(\dot u_h-\dot u_i)\dx\Big],\nonumber
\end{align}
where 
\begin{align}
  g_i^e&=\frac{m_i^e}{|K_e|}\left(
\int_{\partial K_e\cap \partial\Omega}\mathbf f_h\cdot\mathbf n\ds
 - \sum_{j\in\Neb}
 \sum_{e'\in\Be^{\partial}}|\mathbf c_{j,ee'}|
 \mathcal F(u_j,\hat u_{j,ee'};\mathbf n_{j,ee'})
 \right) \nonumber\\&+
  \int_{\partial K_e\cap\partial\Omega}
  \varphi_i[\mathcal F(u_h,\hat u_h;\mathbf n)
    -\mathbf f_h\cdot\mathbf n]\ds.\label{aecb}
\end{align}
The contribution of $g_i^e$ vanishes in the periodic case and
for indices $i$ such that
$\mathbf x_i\notin\partial\Omega$. If $K_e$ has
no nodes on $\partial\Omega$, then $g_i^e=0\ \forall i\in\Ne$, and
the components of the
element vector $(f_i^e)_{i\in\Ne}$ satisfy the zero sum condition
\beq\label{zerosum}
\sum_{i\in\Ne} f_i^e=0
\eeq
\red{because
  $$m_i^e=\int_{K_e}\varphi_i\dx,\quad\sum_{i\in\Ne}\varphi_i\equiv 1,\quad
  \sum_{i\in\Ne}\nabla\varphi_i\equiv\mathbf 0,\qquad
  \sum_{i\in\Ne}\dot u_i\varphi_i=\dot u_h|_{K_e},
  $$
  $$
  \sum_{i\in\Ne}\frac{m_i^e}{|K_e|}=1,\qquad
  \sum_{i\in\Ne} m_i^e u_i
  =\int_{K_e}u_h\dx=|K_e|u^e=
  \Big(\sum_{i\in\Ne} m_i^e\Big)u^e.$$
  In the following theorem, we formally prove that our high-order
  CG-WENO discretization \eqref{weak:stab2}
  can be recovered by adding the antidiffusive element
  contributions $f_i^e$ to the low-order scheme with deactivated flux limiter.
}

\begin{theorem}
  {\it The semi-discrete scheme 
    \eqref{mcl} is equivalent to the CG-WENO scheme \eqref{weak:stab2}
  for $\bar u_i^e=\bar u^{e}+f_i^e/m_i^e$ with
     $f_i^e$ defined by \eqref{aec} and} 
$$\bar u^e=u^e-\frac{\Delta t_e}{|K_e|}
\sum_{e'\in\Be}|S_{ee'}|f_{ee'}^H.$$
\end{theorem}
\begin{proof} Recalling the definition \eqref{hoflux} of the
  high-order flux $f_{ee'}^H$ and using the assumptions of the theorem,
  we find that
  \eqref{mcl} can be written as
  \begin{align*}
    m_i\td{u_i}{t}&=\sum_{e\in\mathcal E_i}m_i^e\frac{\bar u_i^e-u_i}{\Delta t_e}
    =\sum_{e\in\mathcal E_i}\left(m_i^e\frac{\bar u^e-u_i}{\Delta t_e}
      +\frac{f_i^e}{\Delta t_e}\right)\\
&=\sum_{e\in\mathcal E_i}\left(m_i^e\frac{u^e-u_i}{\Delta t_e}
      -\frac{m_i^e}{|K_e|}\sum_{e'\in\Be}|S_{ee'}|f_{ee'}^H
      +\frac{f_i^e}{\Delta t_e}\right) \\
      &= \sum_{e\in\mathcal E_i}\left(m_i^e\frac{u^e-u_i}{\Delta t_e}
      -\frac{m_i^e}{|K_e|}\int_{K_e}\nabla\cdot\mathbf f_h\dx
      +\frac{f_i^e}{\Delta t_e}+g^e
      \right),
  \end{align*}
  where $\int_{K_e}\nabla\cdot\mathbf f_h\dx
  =\int_{\partial K_e}\mathbf f_h\cdot\mathbf n\ds$ and
 \begin{align*}
   g^e&=\frac{m_i^e}{|K_e|}\left(\int_{\partial K_e\cap \partial\Omega}
   \mathbf f_h\cdot\mathbf n\ds-
  \sum_{j\in\Neb}\sum_{e'\in\Be^\partial}|\mathbf c_{j,ee'}|
  \mathcal F(u_j,\hat u_{j,ee'};\mathbf n_{j,ee'})\right)\\&
  =g_i^e-\int_{\partial K_e\cap\partial\Omega}
  \varphi_i[\mathcal F(u_h,\hat u_h;\mathbf n)
    -\mathbf f_h\cdot\mathbf n]\ds.
\end{align*}
Using \eqref{aec} and \eqref{aecb}, we deduce that
   \begin{align*}  
     m_i\td{u_i}{t}&=\sum_{e\in\mathcal E_i}\Big[
  \int_{K_e}\nabla \varphi_i\cdot(\mathbf f(u_h)-\mathbf f_h)\dx
  -\int_{K_e}\varphi_i
    \nabla\cdot\mathbf f_h\dx\\
    &\phantom{\sum_{e\in\mathcal E_i}\Big[}
    -\int_{\partial K_e\cap\partial\Omega}
  \varphi_i[\mathcal F(u_h,\hat u_h;\mathbf n)
    -\mathbf f_h\cdot\mathbf n]\ds\\
  &\phantom{\sum_{e\in\mathcal E_i}\Big[}
    - s_h^e(\varphi_i,u_h)-\int_{K_e}\varphi_i(\dot u_h-\dot u_i)\dx\Big]\\
    &=-\sum_{e\in\mathcal E_i}\Big[
    \int_{K_e}\varphi_i\nabla\cdot\mathbf f(u_h)\dx \\
        &\phantom{\sum_{e\in\mathcal E_i}\Big[}
    +\int_{\partial K_e\cap\partial\Omega}
  \varphi_i[\mathcal F(u_h,\hat u_h;\mathbf n)
    -\mathbf f(u_h)\cdot\mathbf n]\ds\\
    &\phantom{\sum_{e\in\mathcal E_i}\Big[}
    + s_h^e(\varphi_i,u_h)+\int_{K_e}\varphi_i(\dot u_h-\dot u_i)\dx\Big]\\
  & =
  -[a(\varphi_i,u_h)+b(\varphi_i,u_h,\hat u_h)+s_h(\varphi_i,u_h)
  +(\varphi_i,\dot u_h-\dot u_i)].
   \end{align*}
This proves that the scheme under investigation is equivalent to  
   \eqref{weak:stab2}.
\end{proof}

In the process of convex limiting, we first ensure that
\eqref{uelim} is IDP by choosing $\Delta t_e$ small enough
to satisfy the requirements of Lemma~\ref{lemma1} or using
$\Delta t_e=\Delta t_e^{\max}$ defined by \eqref{dtmaxcell}
and tuning the correction factors $\alpha_{ee'}\in[0,1]$.
Next, we calculate element-based
correction factors $\beta_e\in[0,1]$ such that the IDP property
of $\bar u^e$ is preserved by the intermediate states
\beq\label{slope}
\bar u_i^{e}=\bar u^{e}+\beta_ef_i^e/m_i^e,\qquad i\in\Ne.
\eeq
The application of $\beta_e$ to the element contributions
$f_i^e,\ i\in\mathcal N_e$ can be interpreted as
slope limiting of Barth--Jespersen type \cite{barth1989} in the scalar case and of Zhang--Shu type \cite{zhang2010b,zhang2011,zhang2012} in the case of
a hyperbolic system.

The fully discrete version of the constrained scheme \eqref{mcl}
is given by 
\beq\label{eq:SSPupdate}
m_iu_i^{\rm SSP}=\sum_{e\in\mathcal E_i}m_i^e\left[\left(1-\frac{\Delta t}{\Delta t_e}\right)u_i
  +\left(\frac{\Delta t}{\Delta t_e}\right)\bar u_i^e\right].
\eeq
Its IDP property is guaranteed under the global CFL
condition \eqref{cflglob} with \red{pseudo} time steps
$\Delta t_e=\Delta t_e^{\max}$ defined by \eqref{dtmaxcell} if
the flux limiter is activated (otherwise at least for
$\Delta t_e$ satisfying condition \eqref{cflsubcell}
of Lemma~\ref{lemma1}).

It remains to choose the algorithms for calculating
$\alpha_{ee'}$ and $\beta_e$. The choices that we make in this work
are based on our experience with convex
limiting of FCT and MCL type (see \cite[Ch. 4 and 6]{kuzmin2023},
 \cite{dobrev2018,hajduk2021,kuzmin2020,kuzmin2021,lohmann2017}).

\subsection{Flux limiting}
\label{sec:fluxlim}

Since our formula \eqref{uelim}  for the intermediate cell average $\bar u^e$ has the structure of a fully discrete finite volume scheme, flux limiting can be performed using the localized FCT algorithms presented in \cite{guermond2019,kuzmin2021,kuzmin2025a}. In general, the correction factor $\alpha_{ee'}\in[0,1]$ must satisfy the symmetry condition
\beq
\alpha_{ee'}=\alpha_{e'e}\qquad \forall e'\in\Be\backslash\Be^\partial
\eeq
and ensure that $\bar u^e\in\mathcal G$
for an invariant domain $\mathcal G$ that is preserved by
\beq\label{uelowFCT}
\bar u^{e,L}=u^e-\frac{\Delta t_e}{|K_e|}
\sum_{e'\in\Be}|S_{ee'}|f_{ee'}^L
\eeq
with the \red{pseudo} time step
$\Delta t_e=\Delta t_e^{\max}$, where $\Delta t_e^{\max}$
is given by \eqref{dtmaxcell}. 

The numerical fluxes that define the value of $\bar u^e$
in \eqref{uelim} can be written as
\beq\label{flim}
f_{ee'}=\alpha_{ee'} f_{ee'}^H+(1-\alpha_{ee'})f_{ee'}^L
=f_{ee'}^L-\bar f_{ee'}^A,
\eeq
where $\bar f_{ee'}^A=\alpha_{ee'}f_{ee'}^A$ is a limited counterpart
of the antidiffusive flux
$$f_{ee'}^A=f_{ee'}^L-f_{ee'}^H.
$$
For a general hyperbolic problem of the form \eqref{ibvp}, the representation of
$$
\bar u^e=\bar u^{e,L}+\frac{\Delta t_e}{|K_e|}
\sum_{e'\in\Be}|S_{ee'}|\bar f_{ee'}^A
=\sum_{e'\in\Be}
\frac{|S_{ee'}|}{|\partial K_e|}u_{ee'}
$$
as a convex combination of the flux-corrected intermediate states
$$u_{ee'}=\bar u^{e,L}+
\frac{|\partial K_e|\Delta t_e\bar f_{ee'}^A}{|K_e|}
=\bar u^{e,L}+\alpha_{ee'}
\frac{|\partial K_e|\Delta t_ef_{ee'}^A}{|K_e|}
$$ 
reveals that $\bar u^e\in\mathcal G$ if
$u_{ee'}\in\mathcal G\ \forall e'\in\Be$. The low-order approximation
$\bar u^{e,L}$ is IDP for our choice of $\Delta t_e$. If
$u_{ee'}\notin\mathcal G$ or $u_{e'e}\notin\mathcal G$,
then the high-order target flux $f_{ee'}^A=-f_{e'e}^A$ is unacceptable
and requires limiting.

In the scalar case, local or global bounds
$u^{e,\min}\in\mathcal G$ and $u^{e,\max}\in\mathcal G$
are imposed on the cell average $\bar u^e$ defined by \eqref{uelim}. The
validity of the discrete maximum principle
$\bar u^e\in[u^{e,\min},u^{e,\max}]$
is
guaranteed for 
\beq\label{fAlim}
\bar f_{ee'}^A
=\begin{cases}
\min\{f_{ee'}^A,f_{ee'}^{\max}\} & \mbox{if}\ f_{ee'}^A\ge0,\\
\max\{f_{ee'}^A,f_{ee'}^{\min}\} & \mbox{if}\ f_{ee'}^A< 0,
\end{cases}
\eeq
where
\cite{kuzmin2021,kuzmin2020b,kuzmin2025a},\cite[Sec. 5.3.1.1]{kuzmin2023}
\begin{align*}
  f_{ee'}^{\max}&=\begin{cases}
   \min\left\{
    \frac{|K_e|}{|\partial K_e|}\frac{u^{e,\max}-u^{e,L}}{\Delta t_e},
    \frac{|K_{e'}|}{|\partial K_{e'}|}\frac{u^{e',L}-u^{e',\min}}{\Delta t_{e'}}\right\}
    & \mbox{if}\ e\le E_h,\\
\frac{|K_e|}{|\partial K_e|}\frac{u^{e,\max}-u^{e,L}}{\Delta t_e} &
\mbox{if}\ e> E_h,
  \end{cases} \\
  f_{ee'}^{\min}&=\begin{cases}
  \max\left\{
    \frac{|K_e|}{|\partial K_e|}\frac{u^{e,\min}-u^{e,L}}{\Delta t_e},
    \frac{|K_{e'}|}{|\partial K_{e'}|}\frac{u^{e',L}-u^{e',\max}}{\Delta t_{e'}}\right\}
    & \mbox{if}\ e\le E_h,\\
\frac{|K_e|}{|\partial K_e|}\frac{u^{e,\min}-u^{e,L}}{\Delta t_e} &
\mbox{if}\ e> E_h.
      \end{cases}
\end{align*}
Note that the unnecessary and
ill-conditioned calculation of a correction
factor $\alpha_{ee'}\in[0,1]$ for 
$\bar f_{ee'}^A=\alpha_{ee'}f_{ee'}^A$ is avoided in \eqref{fAlim}.

Guermond et al.~\cite{guermond2018,guermond2019} present a
general line search algorithm for finding
$\alpha_{ee'}$ that ensures positivity preservation
for quasi-concave functions of the conserved variables.
An explicit formula for such $\alpha_{ee'}$ can be
derived if the IDP constraints are linear or can be replaced
with linear sufficient conditions, as shown in
\cite{kuzmin2020,kuzmin2023,moujaes2025} for the
Euler equations of gas dynamics. A~very elegant pressure
fix was derived by Abgrall et al. \cite{abgrall-arxiv,wissocq2025}
using the general framework of geometric quasi-linearization
\cite{wu2023}. \red{A similar formula for $\alpha_{ee'}$ is
  derived in Appendix A using the linearization proposed
  in \cite{kuzmin2020,kuzmin2020c}.}

\begin{remark}\label{remark:mcl}
In contrast to the FCT-type convex limiting algorithms presented
in \cite{guermond2018,guermond2019}, we constrain the intermediate
cell averages \eqref{uelim} of the
semi-discrete scheme~\eqref{mcl} instead of
 auxiliary states $u_{ee'}^{\rm FCT}$
that depend on the \red{(global) real} 
time step $\Delta t$ rather than a \red{(local) pseudo} time step
$\Delta t_e$. Therefore, our limiting strategy is monolithic
in the sense that $\alpha_{ee'}$ is independent
of $\Delta t$.
\end{remark}

\subsection{Slope limiting}
\label{sec:slopelim}

The intermediate nodal
states $\bar u_i^e$ of formula \eqref{slope} can be
constrained using one of the element-based convex limiting
algorithms reviewed in \cite[Ch. 4 and 6]{kuzmin2023}.
The localized FCT schemes developed in
\cite{cotter2016,dobrev2018,lohmann2017}
define 
\beq\label{slimit1}
\beta_e=\min_{i\in\Ne}\beta_{i,e}
\eeq
using nodal correction factors $\beta_{i,e}\in[0,1]$ such that
the IDP property of $\bar u^e$ implies that of $\bar u_i^e
=\bar u^e+\beta_ef_i^e/m_i^e$ for any $\beta_e\in[0,\beta_{i,e}]$.
Importantly, the multiplication by $\beta_e$ preserves the
zero-sum property \eqref{zerosum} of the unlimited element
contributions $f_i^e$ in the absence of boundary terms.

To preserve global or local bounds $u_i^{\min}\in\mathcal G$
and $u_i^{\max}\in\mathcal G$ in the scalar case, we
calculate $\beta_{i,e}$ using the FCT formula
\cite{dobrev2018,hajduk2020,kuzmin2025a,lohmann2017}
\beq\label{slimit2}
\beta_{i,e}=
\begin{cases}
\min\left\{1,\frac{m_i^e(u_i^{\max}-\bar u^e)}{f_i^e}\right\}
&\mbox{if}\ \ f_i^e>0,\\
\min\left\{1,\frac{m_i^e(u_i^{\min}-\bar u^e)}{f_i^e}\right\}
&\mbox{if}\ \ f_i^e<0,\\
1 & \mbox{otherwise}.
\end{cases}
\eeq
As noticed in \cite[Sec. 4.3]{lohmann2017}, this definition
of $\beta_{i,e}$ has the same structure as the vertex-based
version \cite{kuzmin2010} of the Barth--Jespersen slope limiter \cite{barth1989}.

In extensions to systems, the IDP property of
$\bar u_i^e$ can be enforced as in the work of Zhang and Shu
\cite{zhang2010b,zhang2011,zhang2012} or using explicit
formulas derived from linear sufficient conditions. Examples
of element-based algorithms using
closed-form IDP limiters for the compressible Euler equations can be found
in \cite{dobrev2018,kuzmin2020c} and \cite[Example~4.16]{kuzmin2023}.
\red{The simple pressure limiter from
\cite{kuzmin2023,kuzmin2020c} is presented in Appendix A.}
The positivity fix employed in \cite{abgrall-arxiv,wissocq2025}
can also be adapted to the element-based format. The 
framework of geometric quasi-linearization \cite{wu2023} can
be used to derive $\beta_{i,e}$ for general systems.

\begin{remark}
  The monolithic structure of the semi-discrete scheme \eqref{mcl}
  using the flux-corrected cell averages \eqref{uelim} is preserved
  by our slope limiting procedure. Similarly to the
   MCL schemes reviewed in \cite{kuzmin2023}, constraints
  are imposed at the level of spatial semi-discretization,
  which has a well-defined residual even in the steady state
  limit. This property distinguishes our approach from
  predictor-corrector FCT algorithms and classical slope limiters.
\end{remark}

\subsection{Summary and discussion}

\red{The algorithm that we use to construct and limit a high-order
  CG-WENO discretization in each SSP-RK stage 
  involves the following steps:
  \begin{enumerate}
\item In a loop over elements $K_e$,
     \begin{itemize}
\item Perform WENO reconstruction and calculate $\gamma_e$
       given by \eqref{WENO:alpha}.
\item Calculate and store the contributions
       $s_h^e(\varphi_i,u_h),\ i\in\Ne$ of \eqref{stab:WENO}.
\item Add all contributions of $K_e$ to the right-hand side or residual.
 \end{itemize}   
   \item Solve the linear system of equations
     \eqref{weak:stab2} for the nodal states of $\dot u_h$.
\item In a loop over facets $S_{ee'}$, calculate $f_{ee'}=f_{ee'}^H$
  if flux limiting is deactivated
  and $f_{ee'}$ given by \eqref{flim} otherwise (see Section ~\ref{sec:fluxlim}).
  
\item In a loop over elements $K_e$,
  \begin{itemize}
  \item Calculate the intermediate cell average
$\bar u^e=u^e-\frac{\Delta t_e}{|K_e|}
    \sum_{e'\in\Be}|S_{ee'}|f_{ee'}$ using a pseudo time step
    $\Delta t_e\le\Delta t_e^{\max}$ with $\Delta t_e^{\max}$
    given by \eqref{dtmaxsubcell}
    if the flux limiter is deactivated  and by
    \eqref{dtmaxcell} otherwise.
  \item Calculate the element contributions $f_i^e,\ i\in\Ne$ 
    given by \eqref{aec}.
  \item Use a slope limiter (see Section \ref{sec:slopelim}) to calculate
    $\beta_e$ given by \eqref{slimit1}.
  \item Calculate the intermediate nodal states $\bar u_i^e,\ i\in \Ne$ given by \eqref{slope}.
  \end{itemize}      
\item Perform the SSP update \eqref{eq:SSPupdate} using a time step
  $\Delta t$ that satisfies \eqref{cflglob}.
\end{enumerate}}

We have intentionally refrained from using matrix notation
for our finite element schemes because all algorithmic steps
can be performed without calculating global or local matrices.
Even the system of equations \eqref{weak:stab2} for the nodal time
derivatives $\dot u_j$ can be solved using a matrix-free
iterative method~\cite{abgrall2017b}. The residuals of fully
discrete schemes can be assembled from element vectors
without generating full element matrices.
\blue{
The number of maximum speeds $\lambda_{0,i}^e$ that are needed to calculate $\Delta t_e^ {\max}$ defined by \eqref{dtmaxsubcell}
equals the small number of boundary nodes $i\in\mathcal N_e^{\partial}$ rather than the large number of all node pairs ($i,j$) with indices $i\in\mathcal N_e$ and $j\in \mathcal N_e\backslash\{i\}$. In the flux-limited version, $\Delta t_e^ {\max}$ is determined by \eqref{dtmaxcell} rather than \eqref{dtmaxsubcell}, and the number of speeds $\lambda_{ee'}$ equals the number of cells with indices $e'\in\mathcal B_e$.}

\blue{
\begin{remark}
  The virtual finite element method developed
  by Abgrall et al.~\cite{abgrall2026} uses monolithic
  convex limiting in a ``PAMPA'' algorithm
  \cite{abgrall2025} that also represents a semi-discrete
  scheme in terms of nodal and average states. The employed flux limiter 
  constrains the cell averages at the time level $t^{n+1}$ using the
  MCL framework introduced in \cite{kuzmin2020,kuzmin2021} and the
  pressure fix presented in \cite{abgrall-arxiv,wissocq2025}.
  The raw antidiffusive element contributions $f_i^e$ of the slope limiting
  step correspond to a residual distribution method. Our element-based
  MCL formulation uses a different convex decomposition of a different
  baseline scheme (CG-WENO). Slope limiting is performed w.r.t. different
  states (intermediate cell averages rather than low-order nodal states).
  While our approach is designed for
  arbitrary-order finite element approximations on simplex or
  tensor-product meshes, the PAMPA scheme proposed in \cite{abgrall2026}
  is suitable for general polygonal meshes and up to third-order accurate.
  The triangular mesh version  \cite{abgrall2025} uses $\mathbb P_2$
  elements enriched by cubic bubble functions.
\end{remark}
}

\section{Numerical results}
\label{sec:num}

Let us now apply the proposed methodology to hyperbolic test problems.
If the exact entropy solution of the problem at hand is unknown, we compare our results with reference solutions obtained on very fine meshes using
a finite volume scheme with LLF fluxes.
In this numerical study, we use uniform structured meshes and initialize the
Bernstein degrees of freedom
by $u_j(0)\coloneqq u_0(\mathbf x_j)$ if the second-order
accuracy of this approximation is sufficient for our purposes.
Unless stated otherwise, we impose periodic boundary conditions
and set the steepening parameter in \eqref{WENO:alpha} equal to $q=1$.
Discretization in time is performed using the third-order explicit SSP Runge--Kutta
method with adaptive global time steps corresponding to
$\omega_{\rm CFL}=0.5$, where $\omega_{\rm CFL}=
\Delta t/\Delta t^{\max}$ is a constant
scaling factor for the sharp bound 
$\Delta t^{\max}$ of the CFL condition that guarantees the IDP property.

\subsection{Scalar problems in 1D}

In our numerical studies for scalar one-dimensional problems of the form \eqref{ibvp}, the invariant domain to be preserved is $\mathcal G=[u^{\min},u^{\max}]$ with $u^{\min}=\min_{x\in \bar\Omega}u_0(x)$ and
$ u^{\max}=\max_{x\in \bar\Omega}u_0(x)$.

\subsubsection{Linear advection}\label{sec:1dadv}

The simplest useful test problem  is the 1D linear transport equation corresponding
to \eqref{ibvp} with $f(u)=u$ and $\Omega=(0,1)$. That is, the initial data $u_0$
is advected with constant velocity $v=1$. The periodic boundary conditions
$u(0,t)=u(1,t)$ are imposed at the end points of $\Omega$. In this 1D example, the
exact solution $u(\cdot,t)$  coincides with $u_0$ for all time instants $t\in\N_0$.

To begin, we advect the $L^2(\Omega)$ projection $u_h(\cdot,0)\in V_h$ of the smooth initial
profile $u_0(x)=\mathrm e^{-100(x-0.5)^2}$ 
to study the convergence behavior of the CG-WENO
target scheme \eqref{weak:stab2} and of its limited counterpart \eqref{mcl}.
The $L^2(\Omega)$ errors at the final time $t=1$ and the corresponding experimental
orders of convergence (EOC) for $\mathbb P_1$ and $\mathbb P_2$ elements
are listed in \cref{tab:eoc-advection}.

\begin{table}[ht!]
\centering
\scriptsize
\begin{tabular}{c||cc|cc||cc|cc}
&\multicolumn{4}{c||}{$\mathbb P_1$}&\multicolumn{4}{c}{$\mathbb P_2$} \\
$1/h$ & WENO & EOC & WENO-L & EOC & WENO & EOC & WENO-L & EOC \\
\hline
32  & 8.84E-02 &      & 8.95E-02       && 6.51E-04 &      & 4.29E-03 \\
64  & 3.12E-02 & 1.50 & 3.13E-02 & 1.52 & 4.81E-05 & 3.76 & 8.24E-04 & 2.38 \\
128 & 3.60E-03 & 3.11 & 3.61E-03 & 3.12 & 5.97E-06 & 3.01 & 1.59E-04 & 2.38 \\
256 & 2.74E-04 & 3.72 & 2.76E-04 & 3.71 & 7.63E-07 & 2.97 & 3.04E-05 & 2.38 \\
512 & 2.33E-05 & 3.55 & 2.85E-05 & 3.28 & 9.64E-08 & 2.98 & 5.89E-06 & 2.37
\end{tabular}
\caption{$L^2(\Omega)$ error norms and convergence rates for
  1D advection of the smooth profile $u_0(x)=\mathrm e^{-100(x-0.5)^2}$
  with velocity $v=1$ up to $t=1$.}\label{tab:eoc-advection}
\end{table}
While the optimal convergence rate $p+1$ is attained or even exceeded by the target scheme for Bernstein finite elements of degree $p\in\{1,2\}$, it turns out that even slope limiting \wrt global bounds is too restrictive in the $\mathbb P_2$ case.
The issue here is subtle and can be explained as follows:
While the \textit{pointwise values} of the exact solution remain
within the range $\mathcal G=[0,1]$, this
is not necessarily the case for the \textit{Bernstein coefficients}.
This peculiarity of Bernstein-basis AFC schemes is the reason 
why constraining the degrees of freedom to stay in $\mathcal G$ leads to peak clipping effects and a loss of accuracy in the neighborhood of smooth global extrema (cf. \cite{hajduk2020b}). To achieve optimal convergence, we envision the use of collocated LGL spectral elements (cf.~\cite{pazner2021,rueda-ramirez2024}) in future implementations of the proposed methodology. 

\begin{remark}
  We observed similar convergence behavior in numerical experiments for
  the nonlinear Burgers test from \cite[Sec.~7.2]{kuzmin2023a}. The results
  (not shown here) confirm that the target scheme converges
  at the optimal rate $p+1$. The limited $\mathbb{P}_1$ approximation exhibits second-order
  convergence, but the imposition of global bounds on the Bernstein coefficients of the
  $\mathbb{P}_2$ version
  prevented us from achieving full third-order accuracy.
\end{remark}

In the next linear advection test, we advect the initial profile~\cite{hajduk2021}
\begin{align}\label{eq:adv-init}
u_0(x) = \begin{cases}
1 & \text{if } 0.2 \le x \le 0.4, \\
\exp(10)\exp(\frac{1}{0.5-x})\exp( \frac{1}{x-0.9}) & \text{if } 0.5 < x < 0.9, \\
0 & \text{otherwise},
\end{cases}
\end{align}
which is comprised of a discontinuous step function and a bump with
$C^{\infty}$-regularity. The final time is again $t=1$.

Using the fixed number $N_h=128$ of nodal points for Bernstein finite elements with polynomial degrees $p\in\{1,2,4,8,16\}$, we perform the linear advection test for the IDP methods presented in Sections~\ref{sec:low} and~\ref{sec:limiters}. In algorithms that use the high-order cell average \eqref{uebar}, the local bounds $\Delta t_e$ of the CFL condition \eqref{cflglob} may not exceed the threshold $\Delta t_e^{\max}$ of Lemma~\ref{lemma1}. The use of the low-order cell average \eqref{uebar2} ensures the IDP property for $\Delta t_e^{\max}$ defined by the less restrictive
CFL bound \eqref{dtmaxcell} of the finite volume LLF scheme. The flux-limited version \eqref{uelim} is IDP under the same CFL condition, which reduces to $\frac{v\Delta t}{h} \le \frac12$ with $v=1$
for this particular test problem.

The results shown in \cref{fig:adv-a,fig:adv-b} confirm that flux limiting can drastically mitigate the time step restrictions. Interestingly enough, the flux-corrected approximations even seem to be less diffusive than numerical solutions obtained without flux limiting but using smaller time steps. This behavior indicates that smearing effects tend to become less pronounced if fewer time steps are needed to reach the final time. It can also be observed that the low-order results become less accurate as the polynomial degree is increased while the mesh is coarsened to match the fixed value of~$N_h$. In view of the findings reported in~\cite{hajduk2020,kuzmin2020a,lohmann2017}, this behavior is caused by the fact that the inherent Rusanov dissipation \eqref{rusdiss} lacks sparsity.

\red{In our experience, the accuracy of the low-order solutions can be greatly improved by using the largest \red{pseudo} time step $\Delta t_e$ that ensures the IDP property of $\bar u^e$ for the given data. However, the results presented in \cref{fig:adv-a} show that the accuracy of slope-limited high-order approximations is not significantly affected by the levels of numerical dissipation in the low-order method corresponding to $\beta_e=0$.} As expected, the step function is slightly better captured with $\mathbb P_1$ elements, while the $\mathbb P_2$ version preserves the peak of the smooth bump with higher precision, despite the fact that the corresponding mesh is one level coarser than that for the $\mathbb P_1$ simulation. 

While it is remarkable that flux limiting ensures the validity of IDP constraints under the large-cell CFL condition, the requirement of stability for the unconstrained target (CG-WENO) scheme may dictate the use of smaller time steps in practice. If the target is unstable, violations of global bounds can be prevented by flux/slope limiters, but spurious oscillations may occur within these bounds, leading to unbounded growth of total variation or approximations of poor quality (see, \eg \cite[left panel of Fig.~3]{hajduk2025a}). The CFL bound of the high-order scheme using Bernstein finite elements does depend on the polynomial degree and may be just marginally larger than the time step that meets the requirements of Lemma~\ref{lemma1}. Moreover, the efficiency gain resulting from the use of slightly larger $\Delta t$ is diminished by the overhead cost associated with flux limiting. For spectral element approximations of LGL type, the CFL condition of the high-order discretization is far less severe than the IDP time step restriction \eqref{cflsubcell} of Lemma~\ref{lemma1}. Hence, the possibility of optional flux limiting may lead to substantial speedups in the LGL version, which represents a promising alternative to Bernstein polynomials.


\begin{figure}[ht!]
\begin{subfigure}[b]{0.32\textwidth}
\caption{No flux limiting}\label{fig:adv-a}
\includegraphics[width=\textwidth]{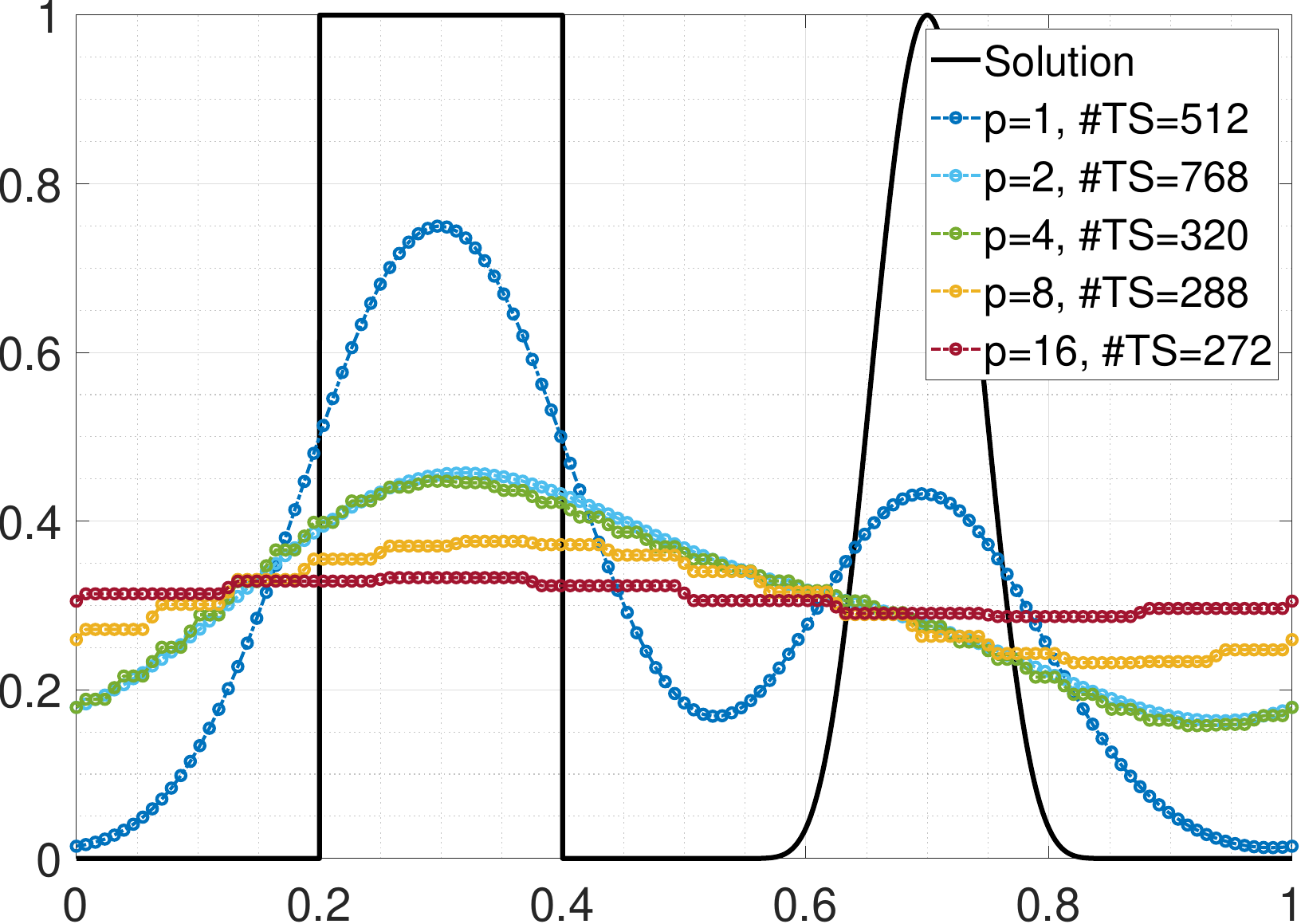}
\end{subfigure}
\begin{subfigure}[b]{0.32\textwidth}
\caption{With flux limiting}\label{fig:adv-b}
\includegraphics[width=\textwidth]{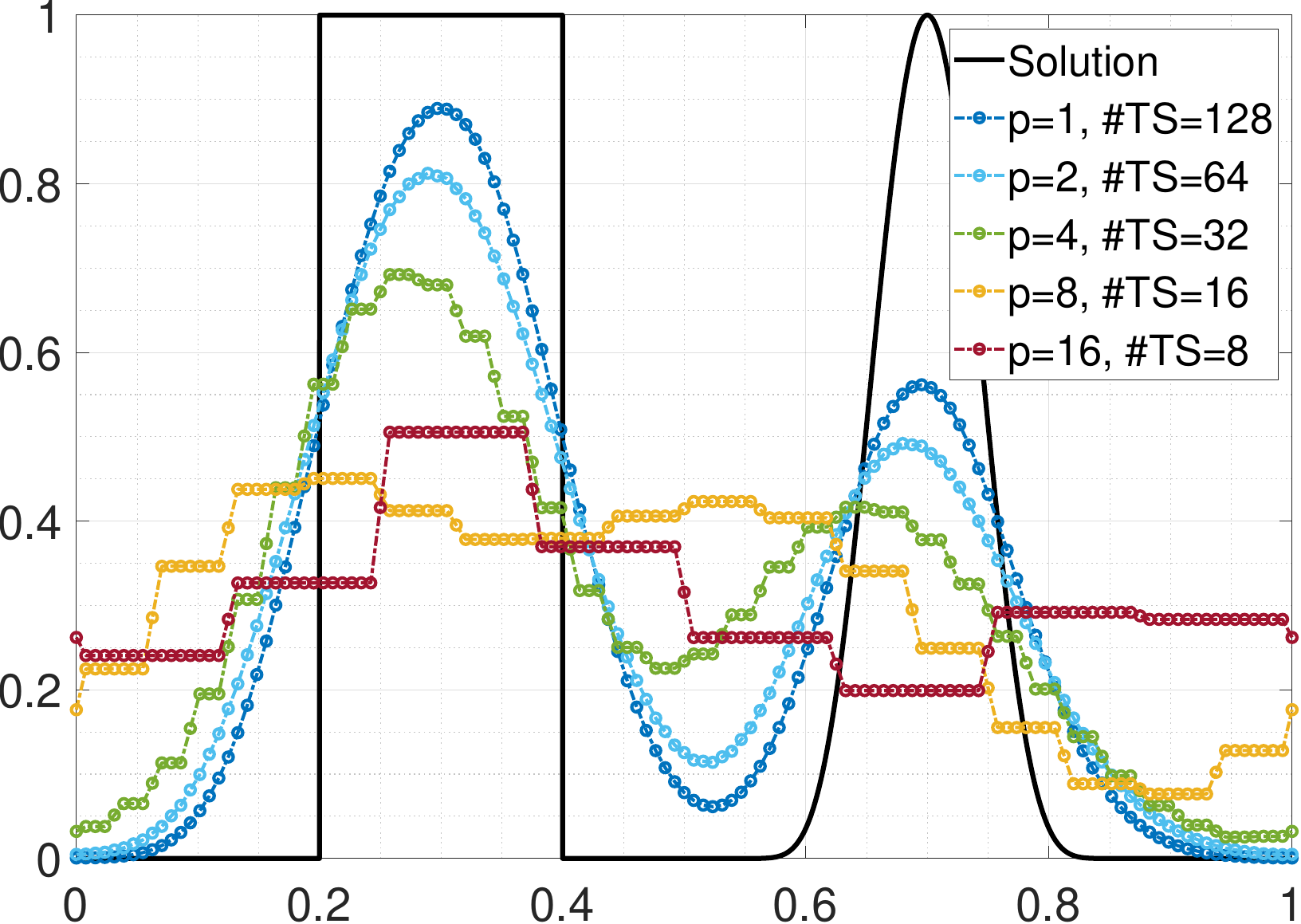}
\end{subfigure}
\begin{subfigure}[b]{0.32\textwidth}
\caption{Slope limiting only}
\includegraphics[width=\textwidth]{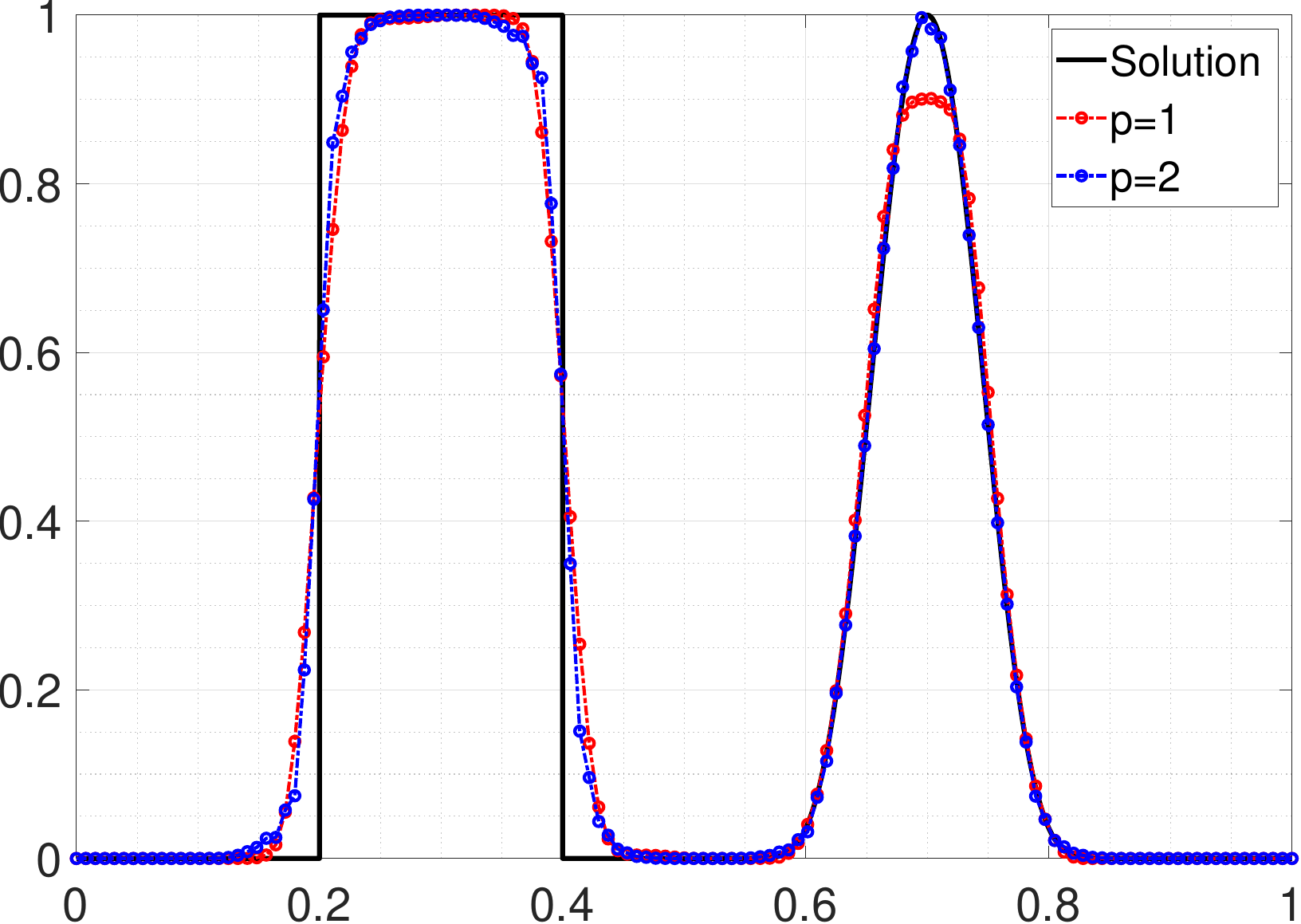}
\end{subfigure}
\caption{Low-order (a,b) and slope-limited (c) solutions to the 1D linear advection problem with initial condition \eqref{eq:adv-init} at $t=1$. All approximations use $N_h=128$ nodes. The number of employed time steps
  is denoted by \#TS.}
\end{figure}

\subsubsection{Nonlinear problem with a nonconvex flux}\label{sec:1dkpp}

Next, we consider the first of the two Riemann problems proposed in~\cite[Sec.~3.1]{kurganov2007} for the 1D nonlinear conservation 
law with the $C^1(\R)$ flux function
\begin{align}\label{eq:1dkpp-flux}
f(u) = \begin{cases}
\frac14 u(1-u) & \mbox{if } u\le\frac12, \\
\frac12u(u-1) + \frac3{16} & \mbox{otherwise.}
\end{cases}
\end{align}
The domain $\Omega=(0,1)$ has an inlet at $x=0$ and an outlet
at $x=1$. The inflow boundary condition $u(0,t)=0$
is prescribed in this test. The initial condition is given by $u_0(x)=\chi_{\{x\ge1/4\}}$, where $\chi_X$ is the characteristic function of a set $X\subset\Omega$. The unique entropy solution consists of a shock propagating to the right and a rarefaction wave that connects the pre-shock value with the right state of the Riemann problem. A closed-form expression for the vanishing viscosity solution can be found in \cite[Suppl. materials, Sec.~A.4.2.2]{kuzmin2023}.
Many numerical methods fail to predict the pre-shock value correctly
and yield profiles with piecewise-constant areas around the shock.

As in the previous example, we first use the low-order methods to solve this problem up to the final time $t=1$ using a total of $N_h=129$ nodes for all combinations of polynomial degree $p$ and mesh resolution $h$. The low-order results obtained without and with flux limiting are displayed in \cref{fig:1dkpp-a,fig:1dkpp-b}. The slope-limited recovery of auxiliary nodal states $\bar u_i^e$ from the unlimited high-order version of $\bar u^e$ yields the approximations shown in \cref{fig:1dkpp-c} for $p\in\{1,2\}$. Since the exact solution is piecewise linear in this test, the $\mathbb P_1$ solution is more accurate than the $\mathbb P_2$ approximation with the same number of nodes. However, the accuracy of the latter is still satisfactory.

Figure \ref{fig:1dkpp-d} shows the evolution of the total
 entropy $\eta_\Omega(u)=\frac12\int_\Omega u^2\D x$ for the slope-limited $\mathbb P_1$ and $\mathbb P_2$ approximations. We observe a monotone decrease of  $\eta_\Omega(u_h)$, which is dominated by the outflow through the boundary at $x=1$. Remarkably, we found no need for entropy stabilization in this test.


\begin{figure}[ht!]
\begin{subfigure}[b]{0.24\textwidth}
\caption{No flux limiting}\label{fig:1dkpp-a}
\includegraphics[width=\textwidth]{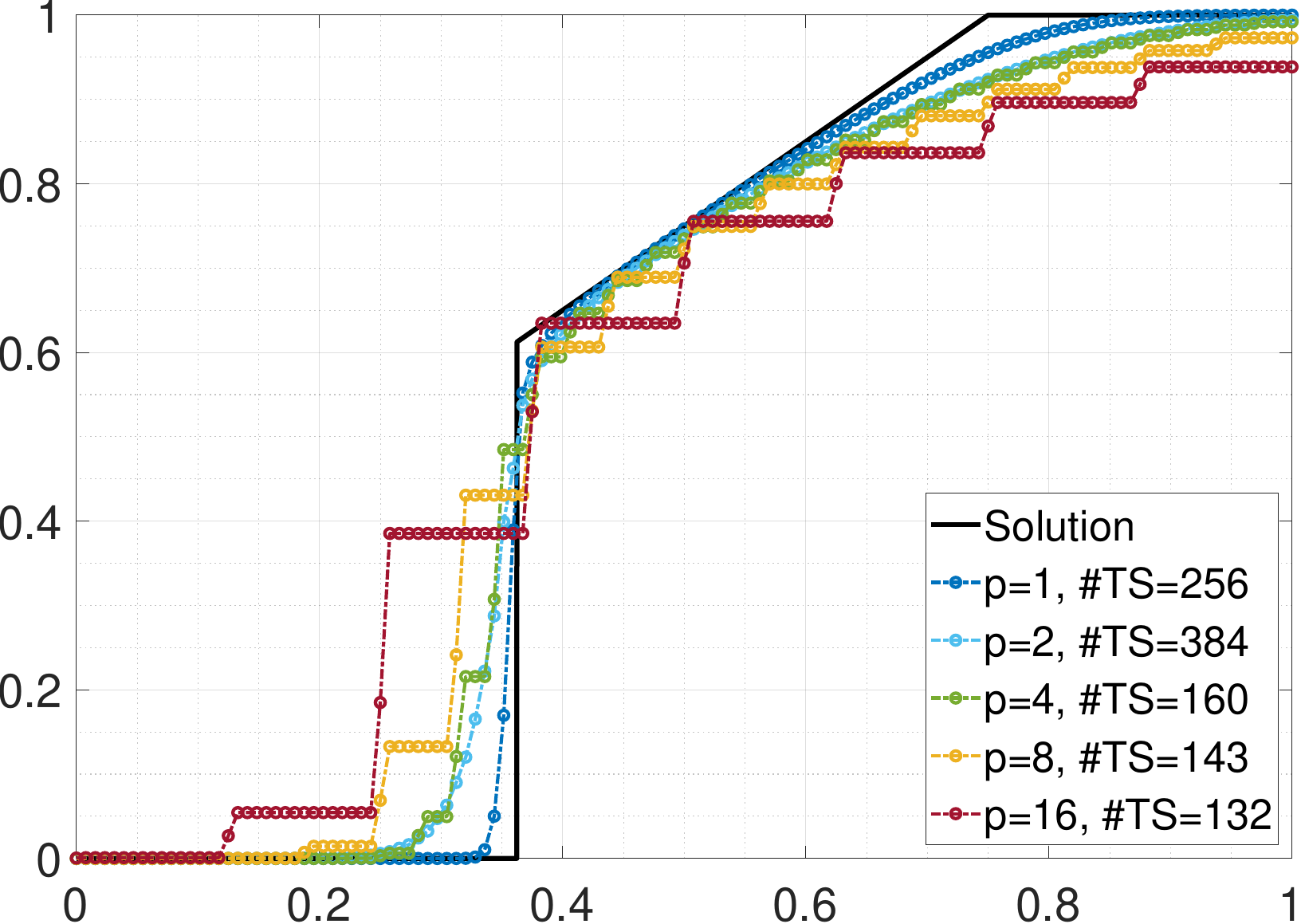}
\end{subfigure}
\begin{subfigure}[b]{0.24\textwidth}
\caption{With flux limiting}\label{fig:1dkpp-b}
\includegraphics[width=\textwidth]{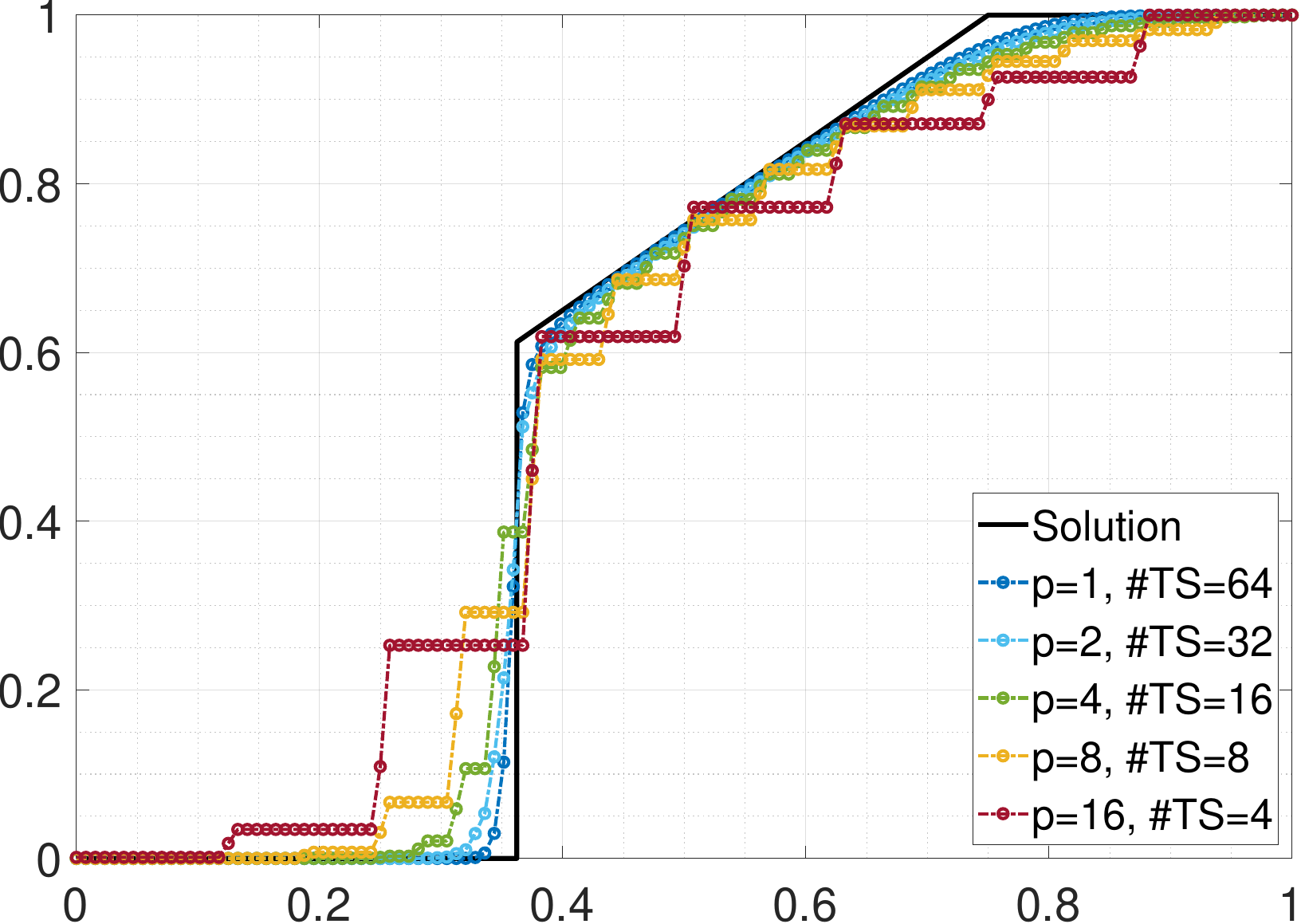}
\end{subfigure}
\begin{subfigure}[b]{0.24\textwidth}
\caption{Slope limiting only}\label{fig:1dkpp-c}
\includegraphics[width=\textwidth]{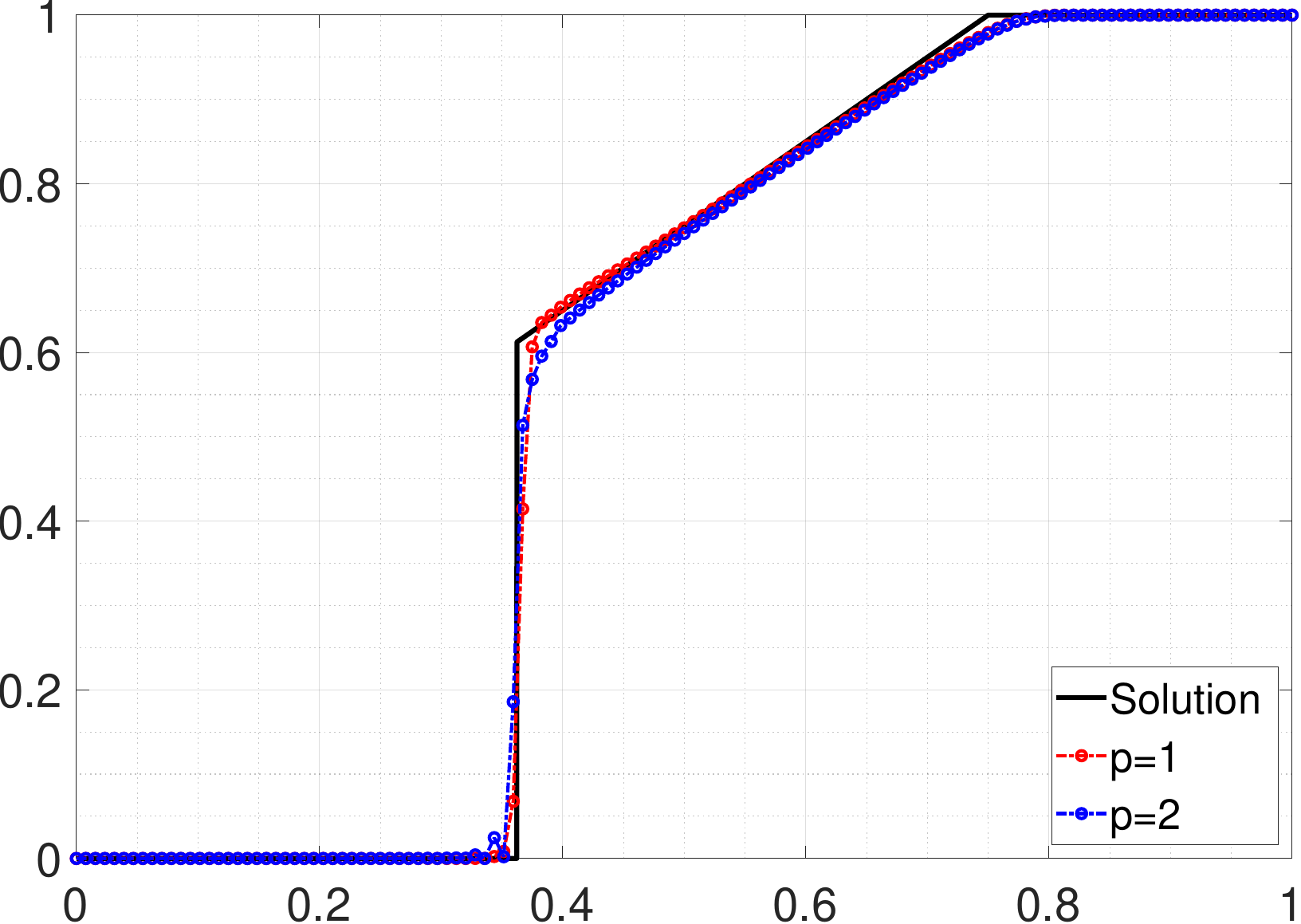}
\end{subfigure}
\begin{subfigure}[b]{0.24\textwidth}
\caption{$\eta_\Omega(u_h)$ vs. $t$}\label{fig:1dkpp-d}
\includegraphics[width=\textwidth]{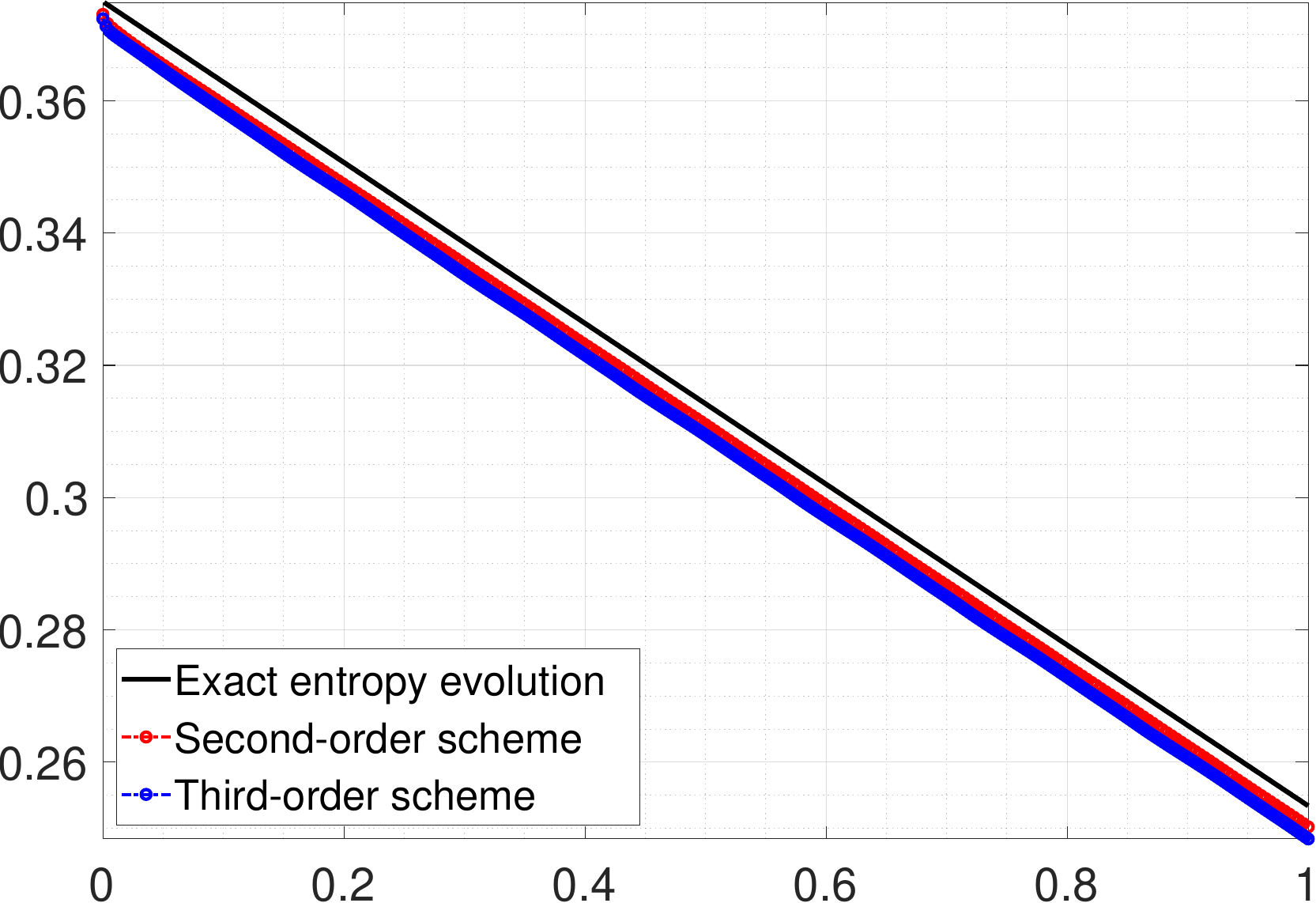}
\end{subfigure}
\caption{Results for the nonlinear 1D conservation law with the flux function \eqref{eq:1dkpp-flux} and $u_0(x)=\chi_{x\ge1/4}$. Low-order (a,b) and slope-limited (c) solutions at $t=1$. All approximations use $N_h=129$ nodes. The number of employed time steps is denoted by \#TS. The entropy evolution history is plotted in panel (d).}
\end{figure}

\subsection{Euler equations in 1D}

Let us now consider the 1D compressible Euler equations, a nonlinear
hyperbolic system of conservation laws for
the density $\rho$, momentum $\rho v$, and total energy $\rho E$
of an ideal gas. The pressure $p$ is given by the equation of state
$p = (\gamma-1)(\rho E - \tfrac12 \rho|v|^2)$, where $\gamma=1.4$
is the adiabatic constant. An invariant domain $\mathcal G$ of the
Euler system consists
of all states $u=(\rho,\rho v,\rho E)^\top$ such that $\rho> 0$
and $p(u)> 0$. The linear positivity constraint for $\rho$ can be
enforced as in the scalar case. To ensure positivity preservation
for $p(u)$, we adapt the 
limiter proposed by Abgrall et al.
 \cite{abgrall-arxiv,wissocq2025}
to our purposes.

\subsubsection{Modified Sod's shock tube}

In the first test for the Euler equations, we consider a
1D Riemann problem corresponding to a modification
\cite[Sec.~6.4]{toro2009} of Sod's shock tube problem
\cite{sod1978}. The modified version is more challenging
because a sonic point occurs within the rarefaction wave
region. As a consequence, numerical methods that lack
entropy stability tend to produce entropy shocks.

The constant initial states of the modified Sod problem are given by
\begin{align*}
(\rho_0,v_0,p_0)(x)=\begin{cases}
(1,0.75,1) &\text{if } x < 0.25,\\
(0.125,0,0.1)&\text{otherwise.}
\end{cases}
\end{align*}
The left boundary of the domain $\Omega=(0,1)$ is a supersonic
inlet, at which the values of the primitive variables are
determined by the left initial state of the Riemann problem.
The boundary data $\hat u$ for $x=1$ may
be set equal to the right initial state, since the final
time $t=0.2$ is too short for the waves to reach the right
boundary. The exact entropy solution features a
rarefaction, a contact, and a shock wave. Its derivation
can be found in ~\cite[Ch.~4]{toro2009}.

We solve this test problem with the target scheme and its slope-limited counterpart using $N_h=129$ degrees of freedom per variable and $p\in\{1,2\}$.
The results displayed in \cref{fig:modsod-a,fig:modsod-b} are in good agreement with the vanishing viscosity solution.
Similarly to the test problem considered in \cref{sec:1dkpp}, the use of
$\IP_1$ elements and fine grids is the best option in this example.
While higher-order approximations on coarser grids perform well for
rarefaction waves, they are typically more diffusive at discontinuities.

\begin{figure}[ht!]
\centering
\begin{subfigure}[b]{0.32\textwidth}
\caption{Target, no limiting}\label{fig:modsod-a}
\includegraphics[width=\textwidth]{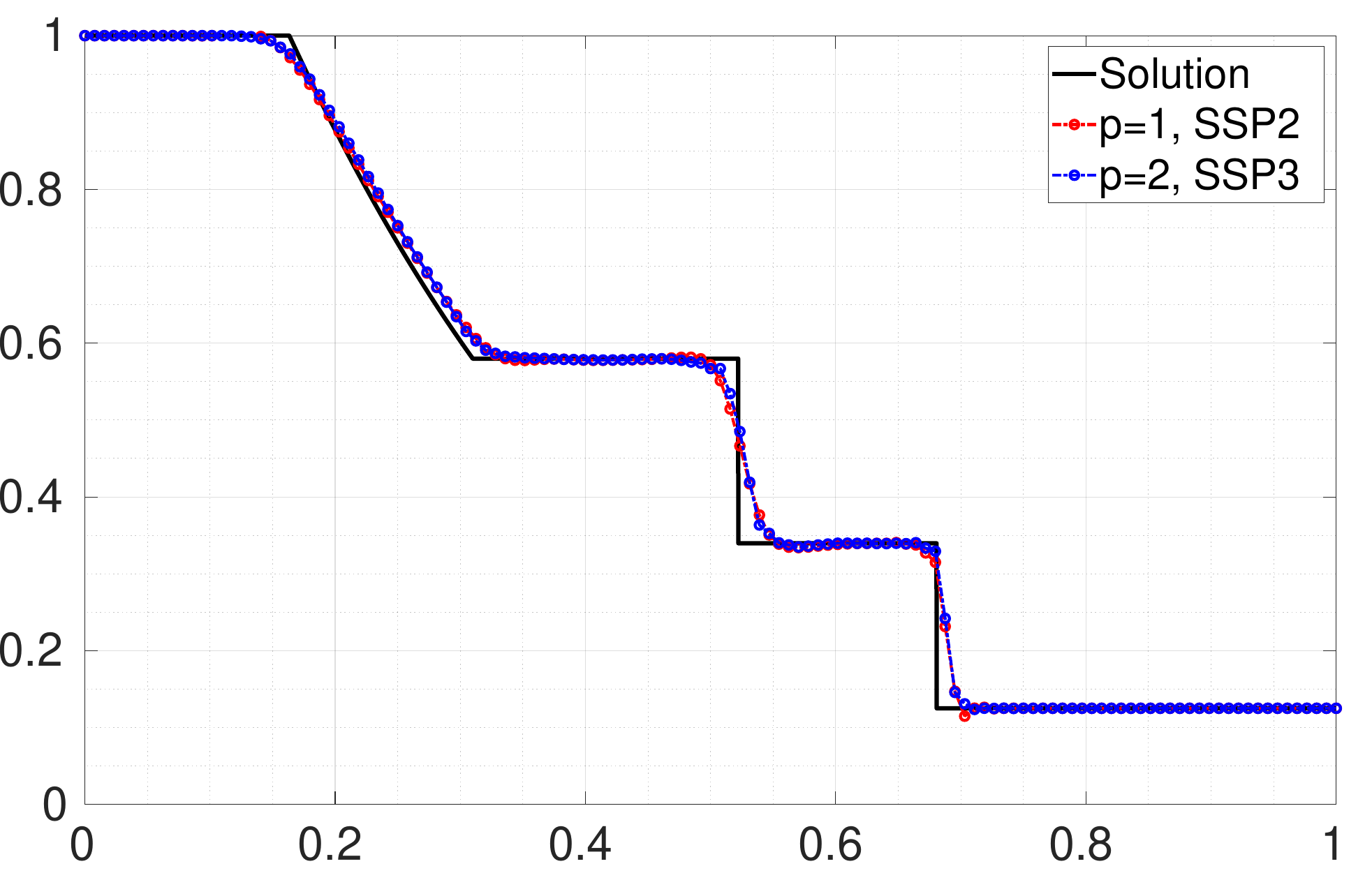}
\end{subfigure}
\begin{subfigure}[b]{0.32\textwidth}
\caption{Slope-limited target}\label{fig:modsod-b}
\includegraphics[width=\textwidth]{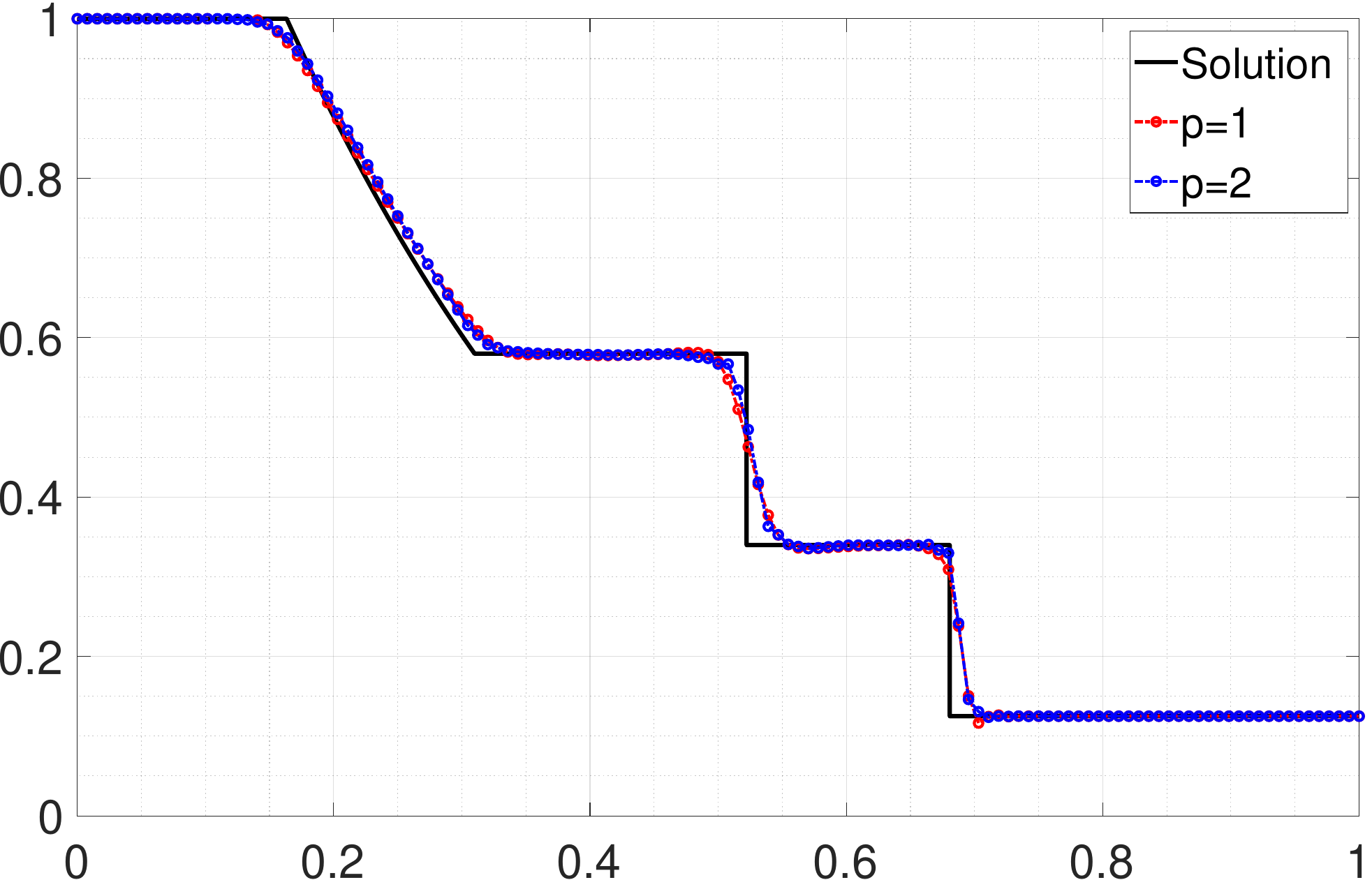}
\end{subfigure}
\begin{subfigure}[b]{0.32\textwidth}
\caption{ $\eta_\Omega (u_h)$ vs. $t$}\label{fig:modsod-c}
\includegraphics[width=\textwidth]{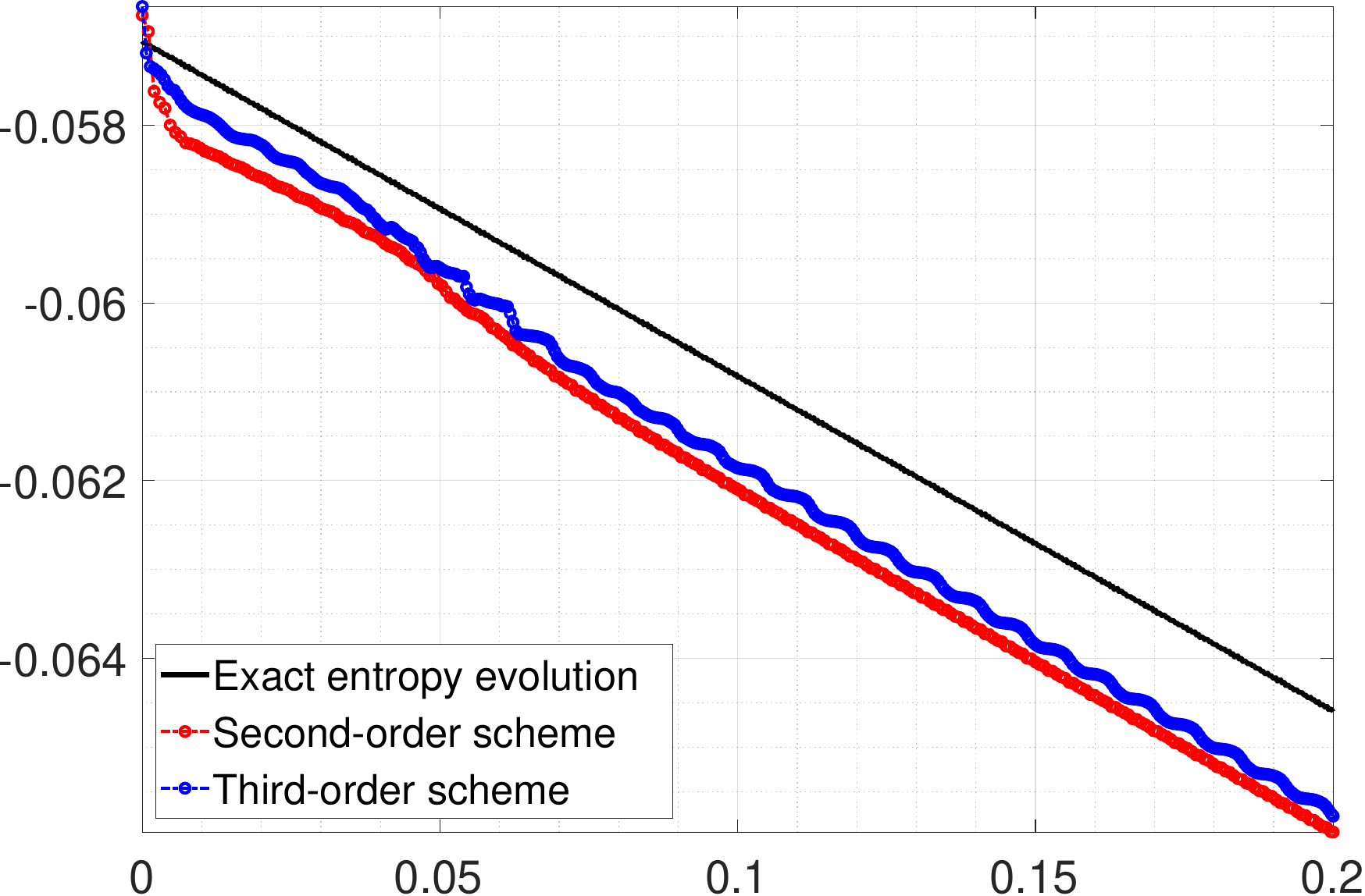}
\end{subfigure}
\caption{Modified shock tube of Sod, density at $t=0.2$ (a,b) and entropy evolution (c) computed with $N_h=129$ unknowns per variable.}
\end{figure}

In \cref{fig:modsod-c}, we plot the evolution of
$\eta_\Omega(u)=\frac{1}{1-\gamma}
\int_\Omega\rho s(u)\D x$, where $s(u) = \log\l p\rho^{-\gamma} \r$
is the specific physical entropy. In contrast to the test case
considered in \cref{sec:1dkpp}, the entropy flux $q(u)=\rho v s/(1-\gamma)$
vanishes at the boundaries, as long as the states $u(0,t)$ and
$u(1,t)$ remain equal to the corresponding initial states of
the Riemann problem. On the time interval for which this is the
case, the total entropy $\eta_\Omega(u)$ of the exact vanishing
viscosity solution $u$ is subject to change only due to entropy
losses at the shock and/or contact. Numerically, we observe a
general trend of entropy dissipation but $\eta_\Omega(u_h)$ does
not decrease monotonically. A probable cause of this behavior
is the lack of local entropy stability, which can be cured by
adding entropy viscosity or using limiter-based entropy fixes
\cite{kuzmin2020e}. However, the small wiggles in the evolution
of $\eta_\Omega(u_h)$ are no major concern, since no entropy
shocks are observed and $\eta_\Omega(u_h)\le\eta_\Omega (u)$.
As demonstrated already in \cite{kuzmin2023a} in the context
of scalar problems, the WENO stabilization term of our method
dissipates enough entropy in most cases. Although
additional entropy stabilization may still be appropriate
for some scalar problems (such as the challenging but artificial KPP
test in Section~\ref{sec:kpp}), we have never observed
entropy-violating behavior of IDP finite element schemes for
the Euler equations.

\subsubsection{Woodward--Colella blast wave problem}\label{sec:blast}

The 1D blast wave problem~\cite{woodward1984} is a challenging test for the Euler equations due to strong jumps in the piecewise-constant initial pressure
\begin{align*}
p_0(x)=\begin{cases}
1000 & \text{if } 0<x<0.1,\\
0.01 & \text{if } 0.1<x<0.9,\\
 100 & \text{if } 0.9<x<1.
\end{cases}
\end{align*}
The initial density is $\rho_0\equiv 1$,  and the fluid is at rest\ie $v_0\equiv 0$.
The domain $\Omega=(0,1)$ has reflecting boundaries at both end points.
For short time intervals, the exact solution can be obtained as for any Riemann problem of the Euler equations~\cite[Ch.~4]{toro2009}.
However, by the final time $t=0.038$, interactions with the boundaries and collisions of waves generated by the two Riemann problems result in the unavailability of an exact solution.

We solve this test problem numerically using $N_h=1001$ degrees of freedom per variable and $p\in\{1,2\}$. The density profile calculated with the target scheme is shown in \cref{fig:wc-a}. Since no slope limiting is performed in this experiment, some negative pressures arise during the simulation run. The code does not crash only because we calculate the speed of sound using the formula $c=\sqrt{\max\{0,\gamma p/\rho\}}$. This  practice is not to be recommended for non-IDP schemes, since it gives rise to consistency errors and disguises unacceptable pressure behavior. Due to round-off errors, a numerical implementation of an IDP scheme might also produce negative pressures of the order of machine precision. In this case, setting $c\coloneqq0$ is an acceptable way to avoid complex-valued solutions. However, such situations did not occur in simulations in which we used the density limiter and the pressure limiter proposed by Abgrall \etal\cite{abgrall-arxiv}. The slope-limited density profiles shown in \cref{fig:wc-b} are nonoscillatory, and no negative pressures were generated at any time.

Interestingly enough, the $\IP_2$ version captures the right density peak slightly better than the $\IP_1$ version despite the fact that the initial condition is piecewise constant. This observation further motivates the use of high-order methods for applications in which complex flow patterns are expected to arise.

\begin{figure}[ht!]
\centering
\begin{subfigure}[b]{0.49\textwidth}
\caption{Target, no limiting}\label{fig:wc-a}
\includegraphics[width=\textwidth]{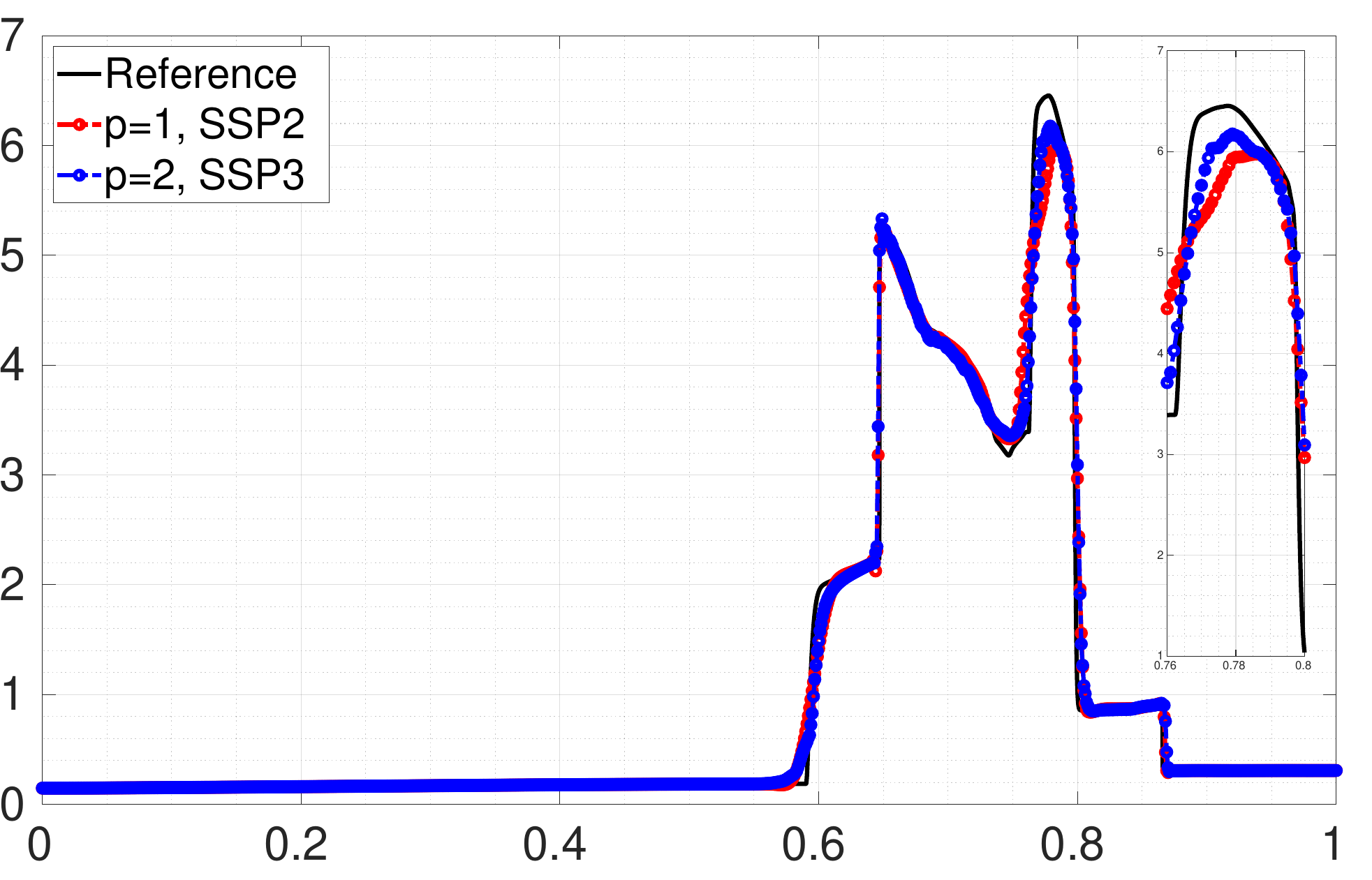}
\end{subfigure}
\begin{subfigure}[b]{0.49\textwidth}
\caption{Slope-limited target}\label{fig:wc-b}
\includegraphics[width=\textwidth]{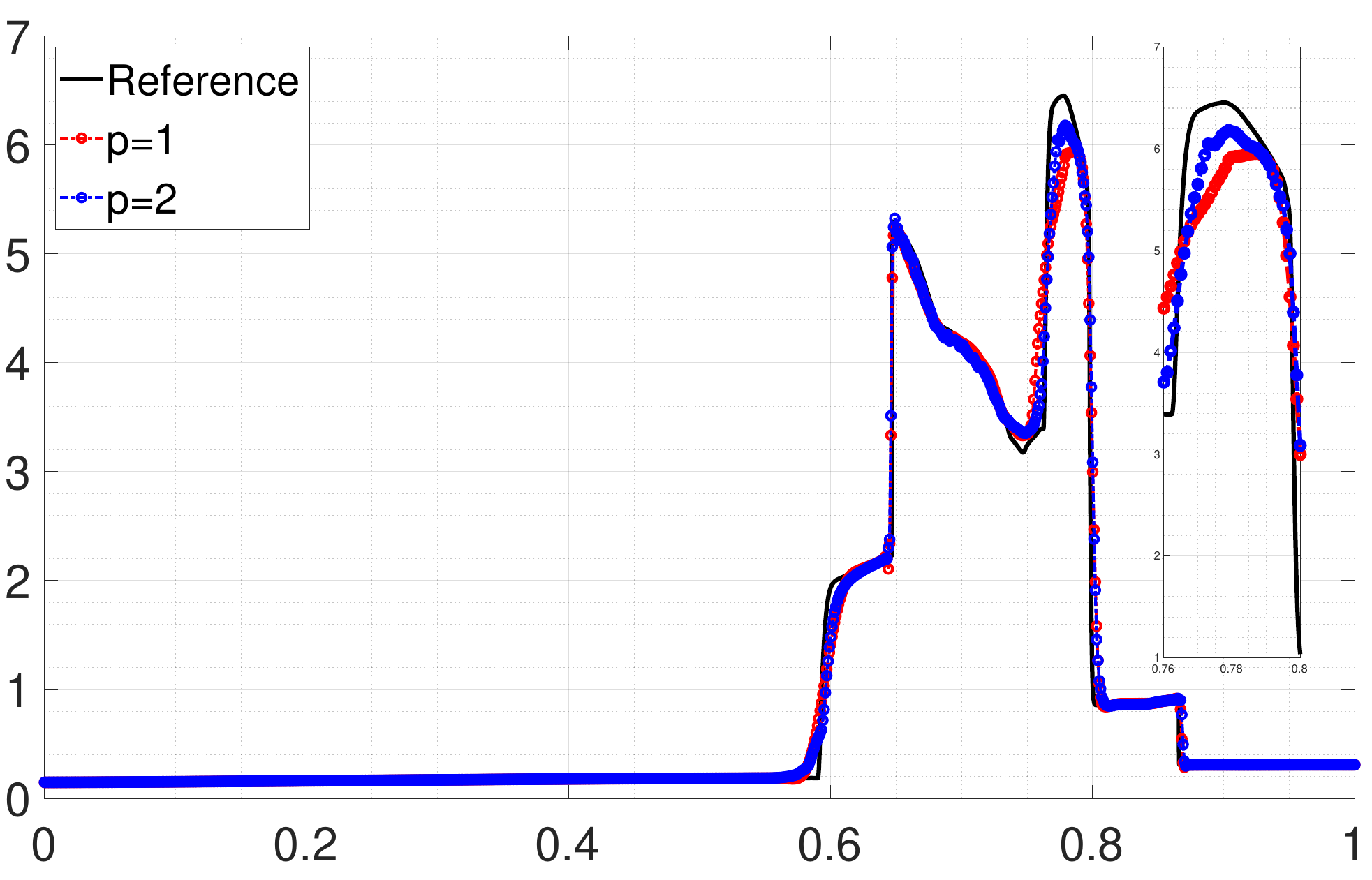}
\end{subfigure}
\caption{Density profiles for the 1D blast wave problem at $t=0.038$
  computed with $N_h=1001$.}
\end{figure}

\subsubsection{Shu--Osher test / sine-shock interaction}

Next, we perform simulations for the popular
Shu--Osher test problem~\cite{shu1989}, which uses the initial condition
\begin{align*}
(\rho_0,v_0,p_0)(x) = \begin{cases}
(3.857143, 2.629369, 10.33333) &\text{if } x < -4, \\
(1+0.2\sin(5x), 0, 1) & \text{otherwise}
\end{cases}
\end{align*}
for the 1D Euler equations to be solved in the
domain $\Omega =(-5,5)$. Up to the final time
$t=1.8$ that is typically used for this test,
both boundary conditions can be defined using a
Riemann solver with the external state $\hat u=u$.
Reference solutions exhibit a right-propagating shock wave with a post-shock region in which the density is highly oscillatory but smooth.

We test the methods under investigation using $p\in\{1,2\}$ and
$N_h=501$ degrees of freedom per variable. Since the
exact solution can be captured well by the
target scheme, no violations of IDP constraints occur
during the simulation run without the slope limiter.
Therefore, the density and pressure fixes are not
activated in practice, and the results coincide with
those obtained without limiting (see \cref{fig:shuosher}).

It is quite encouraging to see that the $\IP_2$ version captures the post-shock oscillatory region much better than the $\IP_1$ approximation on a mesh with twice as many cells.
Furthermore, as is the case for the 1D blast wave problem studied in \cref{sec:blast}, limiting \wrt global bounds does not lead to severe peak clipping effects.
However, one can easily think of situations in which this issue is likely to occur, \eg in simulations of near-vacuum states with quadratic ($\IP_2$) or higher order Bernstein elements. As long as the quantities to be limited remain sufficiently far away from the global bounds of IDP constraints, optimal convergence can be achieved with high-order Bernstein elements. In the context of AFC schemes for the advection equation, numerical evidence for 
the validity of this claim was provided in \cite[Sec.~7.1]{hajduk2020b}.

\begin{figure}[ht!]
\centering
\begin{subfigure}[b]{0.49\textwidth}
\caption{Target, no limiting}
\includegraphics[width=\textwidth]{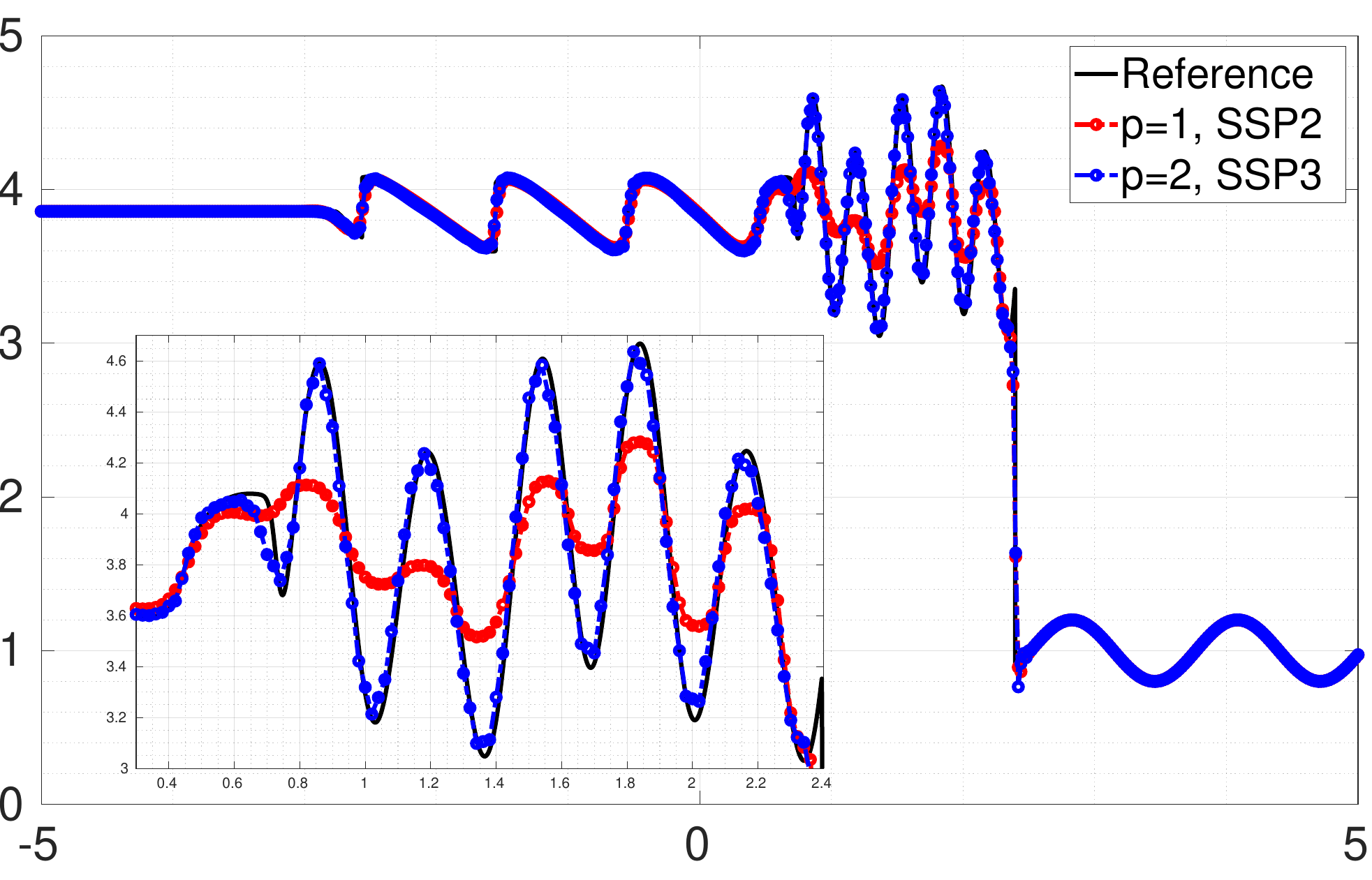}
\end{subfigure}
\begin{subfigure}[b]{0.49\textwidth}
\caption{Slope-limited target}
\includegraphics[width=\textwidth]{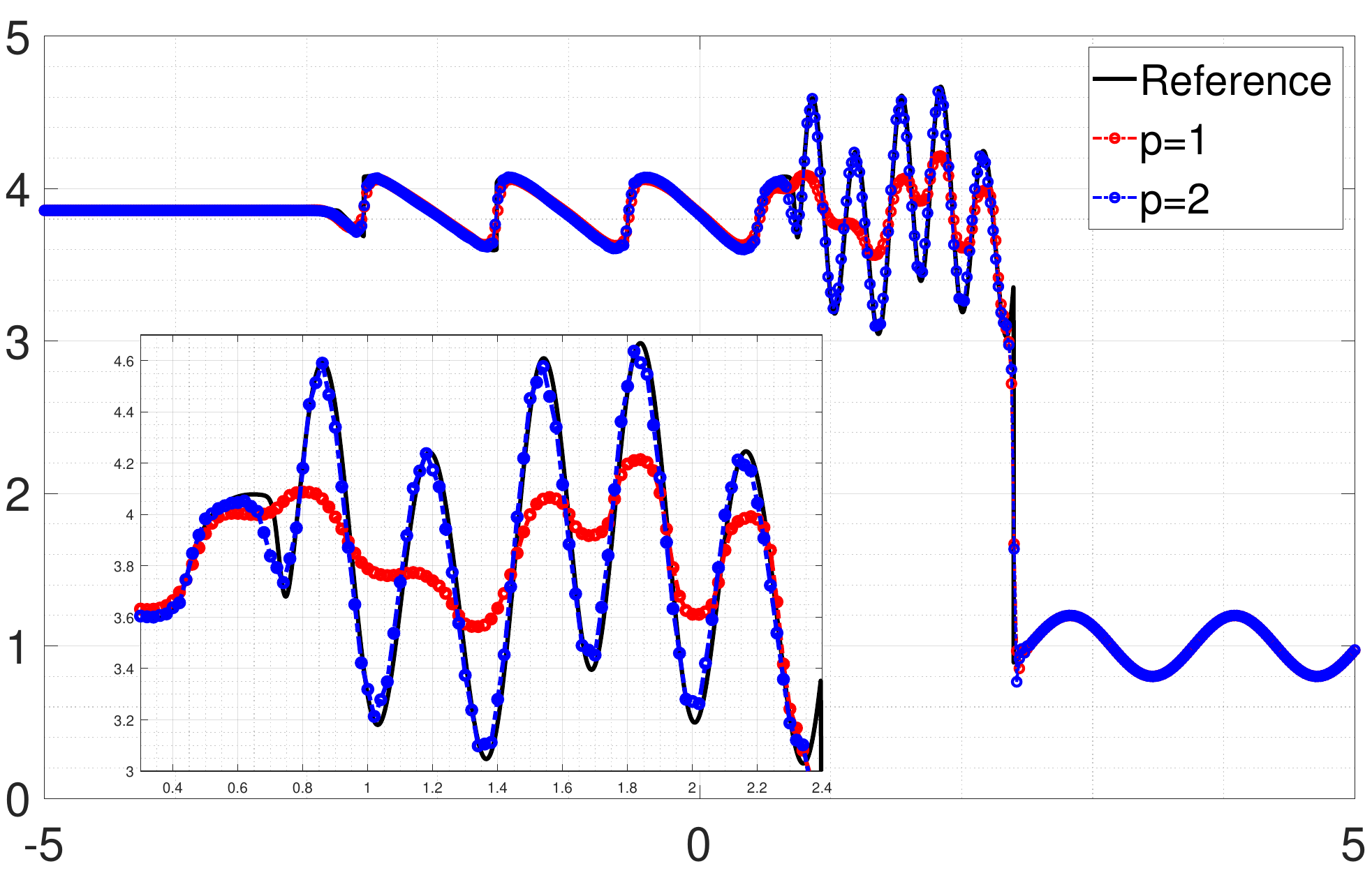}
\end{subfigure}
\caption{Density profiles for the Shu--Osher test at $t=1.8$ computed with $N_h=501$.}\label{fig:shuosher}
\end{figure}

\subsection{Scalar problems in 2D}
\label{sec:scalar2D}

To evaluate the numerical behavior of the limited CG-WENO scheme and its
components in two space dimensions, we consider two additional scalar
test problems, one of which is linear and the other one is nonlinear.
In both cases, we restrict our studies to $\Q_1$ finite
element approximations, for which the Bernstein and Lagrange basis functions
coincide. Computations are performed on uniform meshes with
spacing $h=\frac{1}{128}$. We use the \red{pseudo} time step
$\Delta t_e=\Delta t_e^{\max}$ defined by \eqref{dtmaxcell} and
the global time step $\Delta t=10^{-3}$, which satisfies
the CFL condition \eqref{cflglob} for both test problems.
Although our choice of $\Delta t_e$ exceeds the threshold
of Lemma~\ref{lemma1}, the high-order intermediate cell
averages stay in bounds. Therefore, we
deactivate the flux limiter for~\eqref{uelim}
by setting $\alpha_{ee'}\coloneqq 1\ \forall
e\in\Be$, as proposed in Remark~\ref{rem:timestep}.

Since second-order accuracy is the best we can hope to
achieve with the $\Q_1$ version, we replace the element
contributions \eqref{aec} by
\begin{align}\nonumber
  f_i^e&=\gamma_e\Big(
  m_i^e(u_i-u^e)\\
  &\phantom{\gamma_i^e\Big(}-\Delta t_e\left[
  \int_{K_e}\left(\varphi_i-\frac{m_i^e}{|K_e|}\right)
  \nabla\cdot\mathbf f_h\dx+
  \int_{K_e}\varphi_i(\dot u_h^L-\dot u_i^L)\dx
  \right]\Big),  \label{aecQ1}
\end{align}
where $\dot u_h^L$ denotes a low-order approximation to
the time derivative $\dot u_h$.
This choice is adopted because it corresponds to the second-order
target employed in \cite{kuzmin2020}. A decrease in the value of
$\gamma_e\in[0,1]$
 on the right-hand side of \eqref{aecQ1}
increases the levels of low-order Rusanov dissipation
\eqref{rusdiss}, while decreasing the levels of
high-order stabilization via the terms depending
on $\Delta t_e$. 

In our presentation of numerical results for two-dimensional
hyperbolic problems, the
methods under investigation are labeled as follows:
\begin{itemize}
\item LO: low-order scheme corresponding to $\gamma_e=\beta_e=0$;
\item HO: high-order scheme corresponding to $\gamma_e=\beta_e=1$;
\item WENO: constrains \eqref{aecQ1} using
  $\gamma_e$ defined by \eqref{WENO:alpha}, sets $\beta_e=1$;
\item WENO-L: constrains \eqref{aecQ1} using
  $\gamma_e$ and $\beta_e$ defined by \eqref{WENO:alpha}\\
  and \eqref{slimit1},\eqref{slimit2}
  with global bounds, respectively.
\end{itemize}
In the WENO version, we use the linear weights $w_l^{e,\rm lin}=0.2$ for
$l=1,\ldots,n_e$ and set $w_0^{e,\rm lin}=1-\sum_{l=1}^{n_e}w_l^{e,\rm lin}$.
To study the dependence of the results on the parameter $q\ge 1$
of formula \eqref{WENO:alpha}, we run simulations for $q\in\{3,5,10\}$.
The range of approximate solutions is reported in the figures
to quantify the levels
of numerical diffusion and the magnitude of undershoots/overshoots.

\subsubsection{Solid body rotation}

The solid body rotation problem introduced by LeVeque \cite{leveque1996}
is widely used to assess the ability of numerical methods to capture
smooth and discontinuous shapes of an exact solution to the linear
advection equation
$$
  \pd{u}{t}+
  \nabla\cdot(\mathbf{v}u)=0\quad\mbox{in}\  
\Omega=(0,1)^2
$$  
with the velocity field ${\bf v}(x,y)=(0.5-y,x-0.5)^\top$. The initial
and boundary conditions for this test are defined in
\cite{kuzmin2020,leveque1996}. We stop computations at the
final time $t=2\pi$, at which the exact solution coincides
with the initial data. The invariant domain to be preserved
is the interval $\mathcal G=[0,1]$.

\begin{figure}[h!]
  \small
  
  \begin{minipage}[t]{0.5\textwidth}
\centering (a)  LO, $u_h\in[0.0,0.39]$
\vskip0.2cm

\includegraphics[width=0.95\textwidth,trim=0 0 0 0,clip]{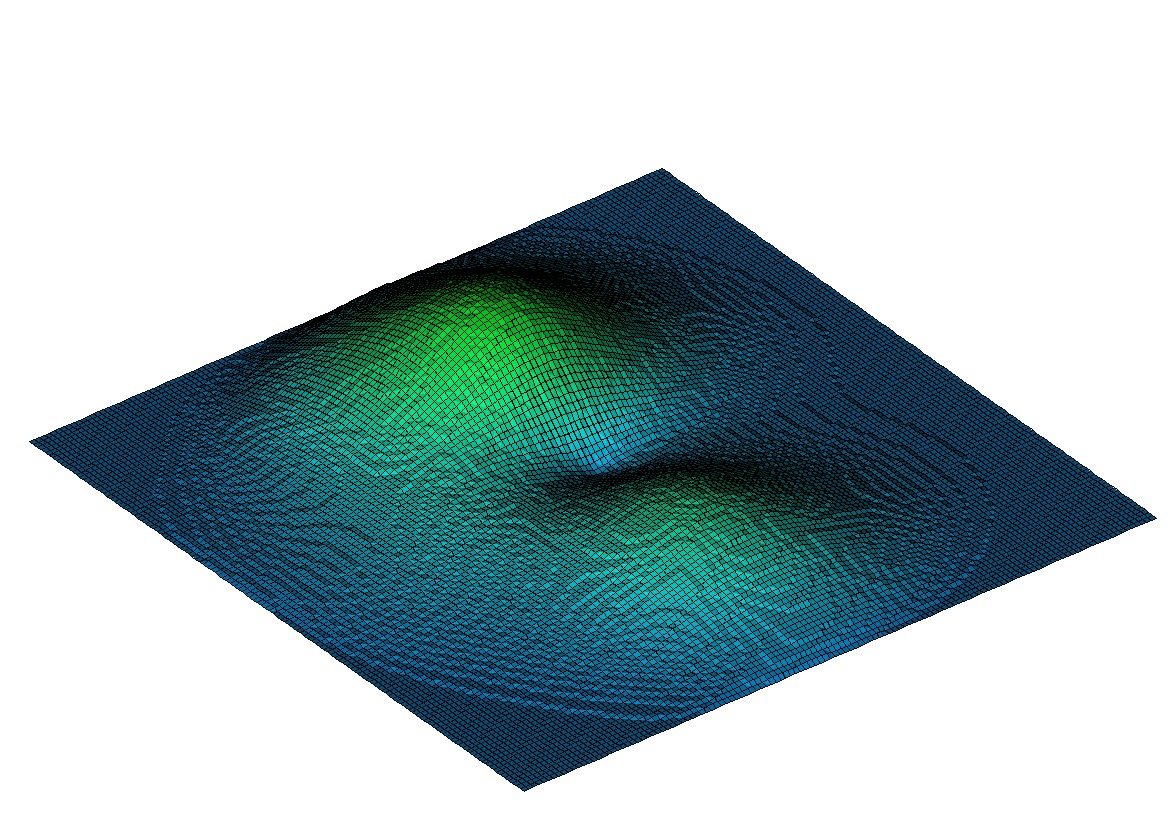}

\end{minipage}%
\begin{minipage}[t]{0.5\textwidth}

  \centering (b) HO, $u_h\in[-0.06,1.12]$
  \vskip0.2cm
  
\includegraphics[width=0.95\textwidth,trim=0 0 0 0,clip]{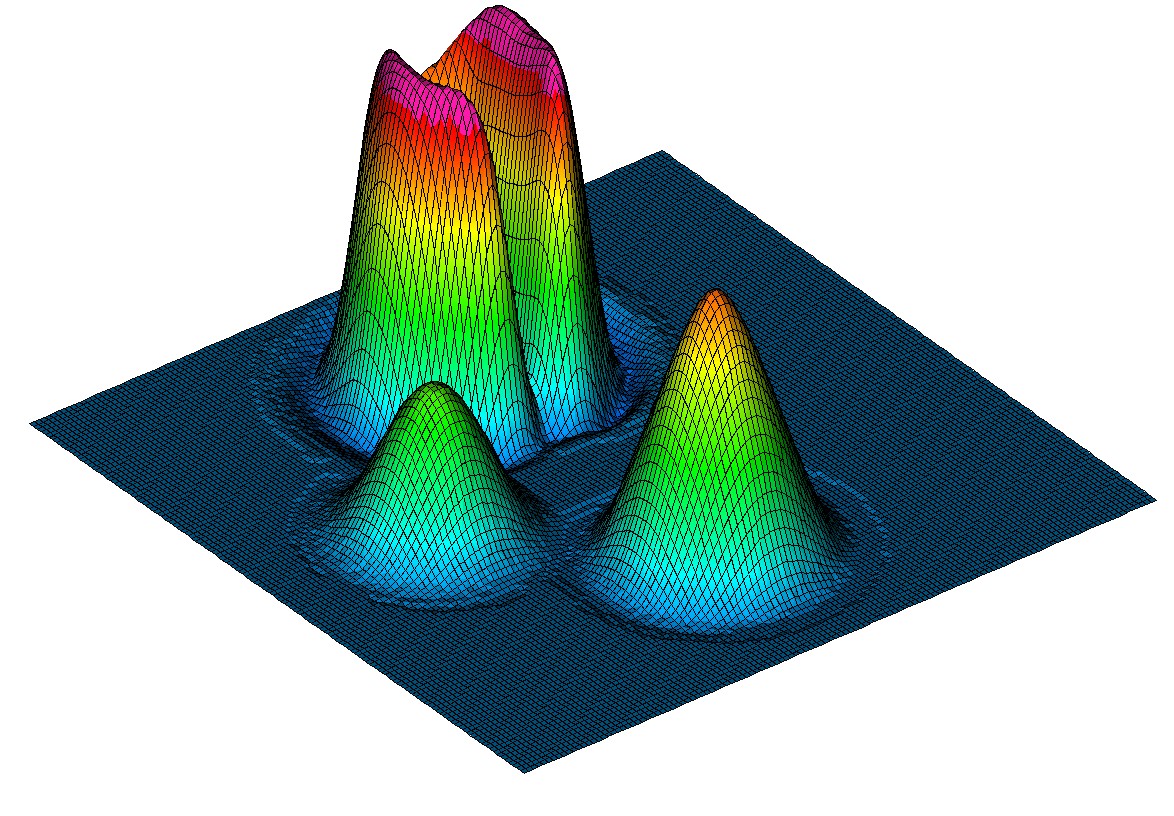}

\end{minipage}

\vskip0.35cm

\begin{minipage}{0.5\textwidth}
\centering (c) WENO, $q=3$, \vskip0.1cm

$u_h\in[-6.26\mbox{e-}10,0.96672]$\vskip0.1cm

\includegraphics[width=0.95\textwidth,trim=0 0 0 0,clip]{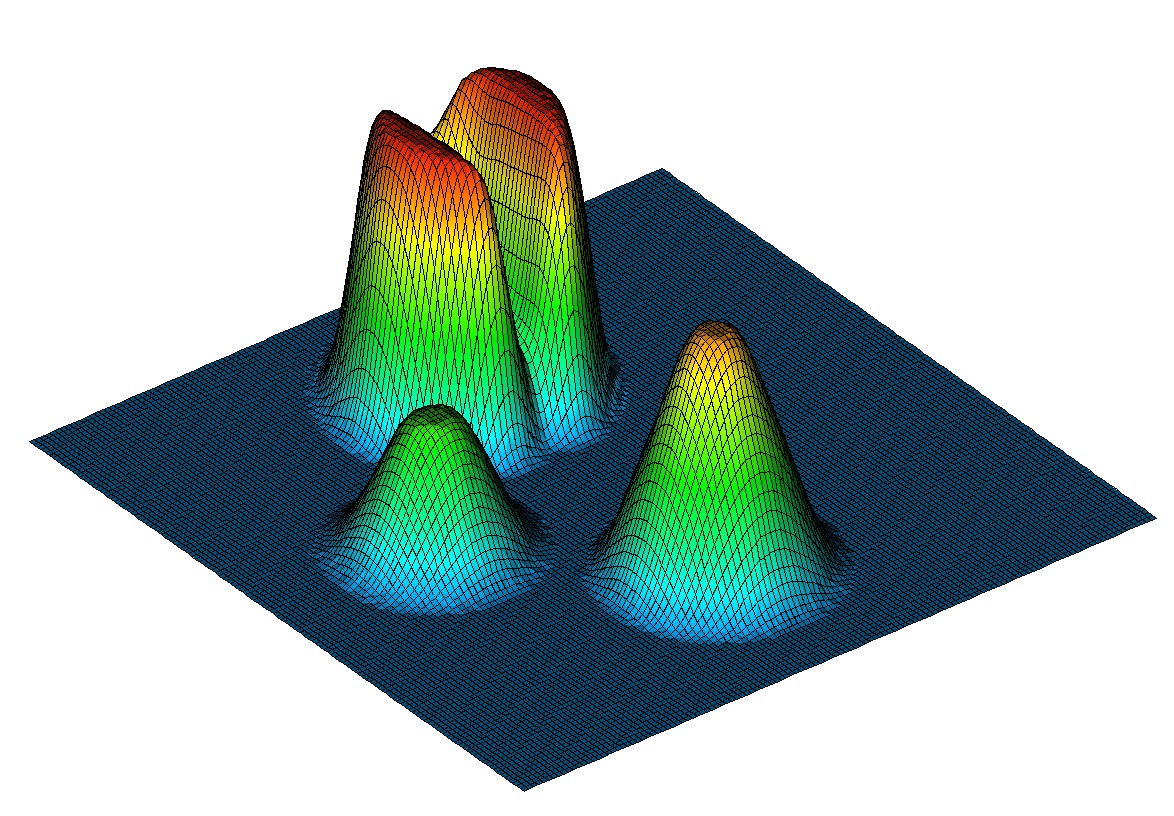}

\end{minipage}%
\begin{minipage}{0.5\textwidth}

\centering (d) WENO, $q=5$,\vskip0.1cm

$u_h\in[-8.80\mbox{e-}5,0.99615]$ \vskip0.1cm

\includegraphics[width=0.95\textwidth,trim=0 0 0 0,clip]{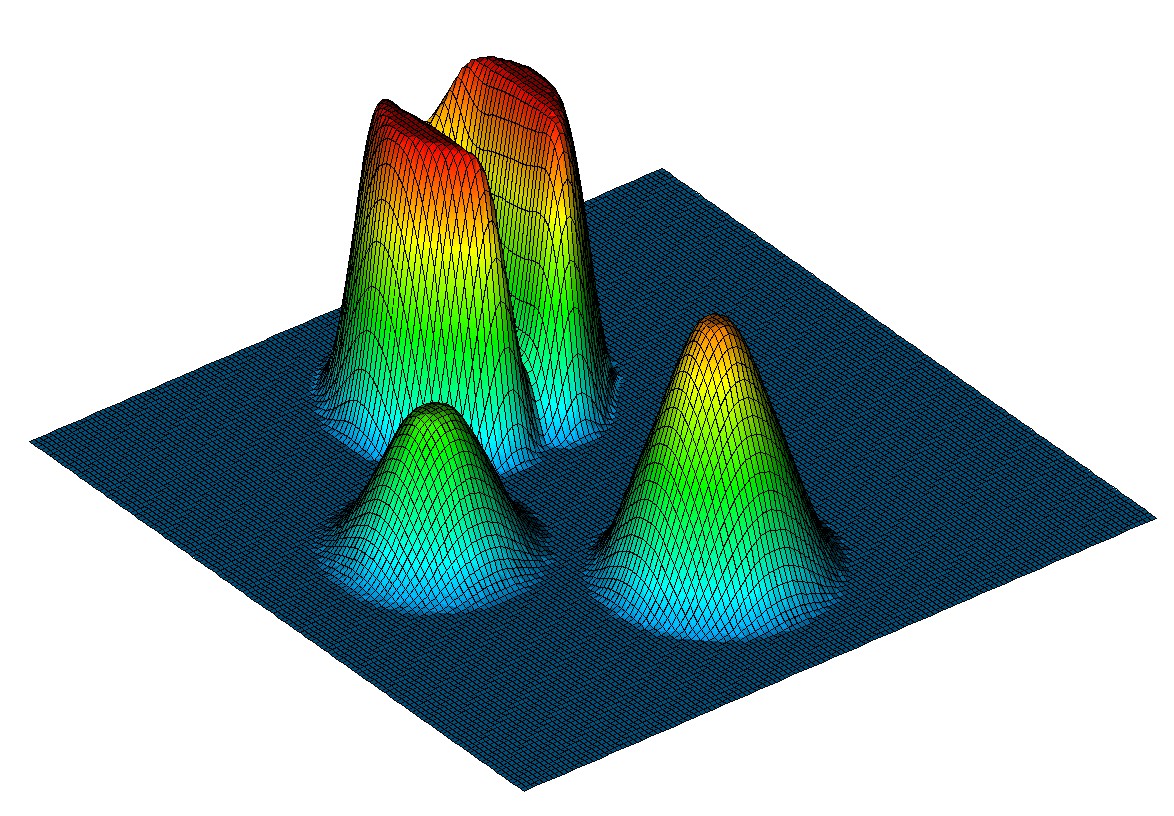}

\end{minipage}

\vskip0.35cm

\begin{minipage}{0.5\textwidth}
\centering (e) WENO, $q=10$, \vskip0.1cm

$u_h\in[ -7.03\mbox{e-}3,1.03601]$ \vskip0.1cm

\includegraphics[width=0.95\textwidth,trim=0 0 0 0,clip]{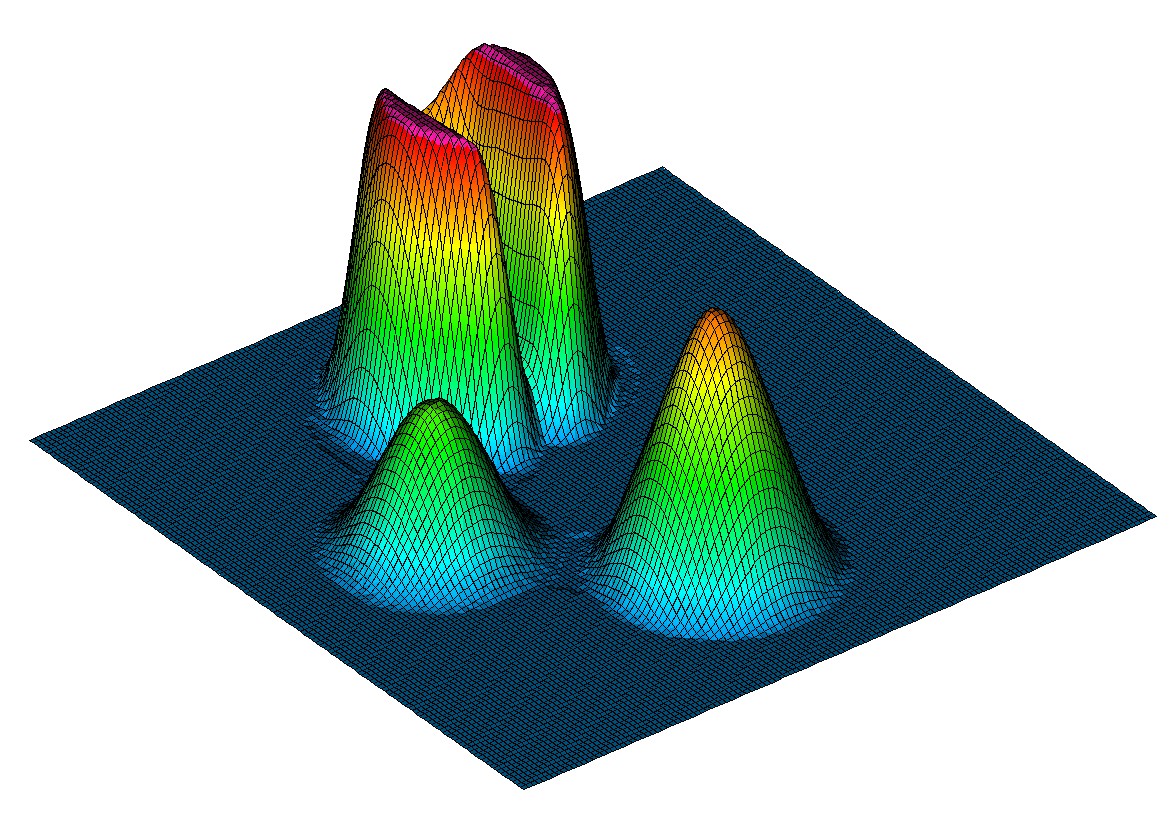}

\end{minipage}%
\begin{minipage}{0.5\textwidth}

\centering (f) WENO-L, $q=10$,\vskip0.1cm

$u_h\in[0.0,0.99989]$ \vskip0.1cm

\includegraphics[width=0.95\textwidth,trim=0 0 0 0,clip]{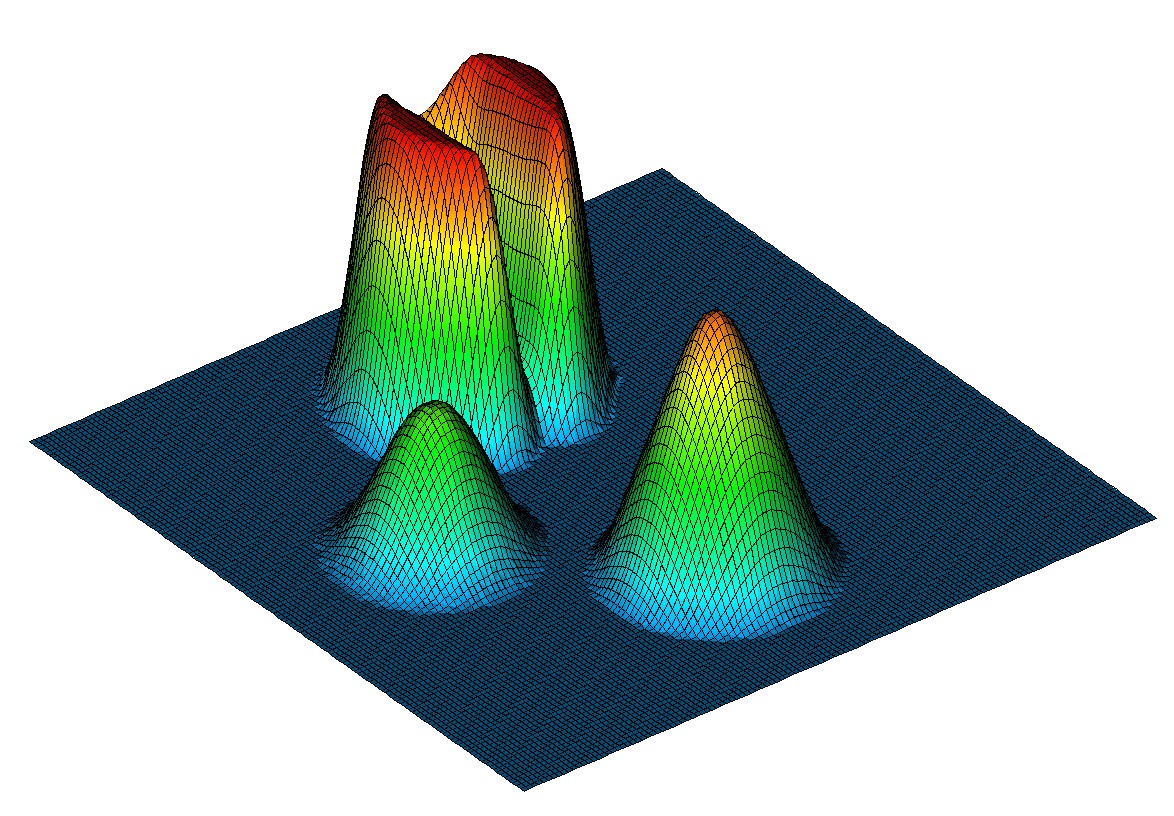}

\end{minipage}

\caption{Numerical solutions to the solid body rotation problem \cite{leveque1996} at $t=2\pi$ calculated using a uniform mesh, $\mathbb Q_1$ elements, $h=\frac{1}{128}$ and $\Delta t=10^{-3}$. }

\label{sbr-pics}

\end{figure}

The numerical results presented in Fig.~\ref{sbr-pics} show
that the low-order (LO) scheme is IDP but too diffusive, while
the high-order (HO) target produces undershoots and overshoots
that are clearly visible even in the eyeball norm. The WENO
results are remarkably accurate and almost IDP for $q=3$. However,
violations of the global bounds $u^{\min}=0$ and $u^{\max}=1$
become more pronounced as the value of $q$ is increased to
$5$ and $10$. The WENO-L solution is both nonoscillatory
and IDP. \red{Hence, our slope limiting strategy eliminates the
risk of failure due to `greedy'
choices of WENO parameters.}

\subsubsection{KPP problem}
\label{sec:kpp}

Extending our experiments to nonlinear scalar equations
of the form \eqref{ibvp-pde} in 2D, we consider the KPP
problem \cite{kurganov2007}, in which $\mathbf{f}(u)=(\sin(u),\cos(u))$
and $\Omega=(-2,2)\times(-2.5,1.5)$. The initial condition is given by
$$
u_0(x,y)=\begin{cases}
\frac{7\pi}{2} & \mbox{if}\quad \sqrt{x^2+y^2}\le 1,\\
\frac{\pi}{4} & \mbox{otherwise}.
\end{cases}
$$
At the final time $t=1$ the exact entropy solution of the KPP
problem exhibits \red{a spiral-shaped shock structure, which is
captured correctly by WENO schemes only for the conservative
choice $q=1$ of the steepening parameter (not shown here).
To avoid convergence to spurious weak solutions for larger
values of $q$, we replace the sensor $\gamma_e$ of
formula \eqref{aecQ1} by $\min\{\gamma_e,\xi_e\}$,} where
$\xi_e$ is an entropy correction factor. Using the nodal
sensors $R_i$ of entropy viscosity stabilization employed
in \cite{guermond2018} and \cite[Eq.~51]{kuzmin2020a}, we set
$\xi_e=\min_{i\in\Ne}R_i$.

Figure~\ref{kpp-pics} shows the entropy-satisfying low-order
solution, the entropy-violating high-order solution, and
the WENO results for the KPP problem. Similarly to the
linear solid body rotation test, the global maximum
and minimum of the initial data are preserved for
$q=3$. The solutions obtained with $q=5$ and $q=10$
without limiting exhibit undershoots and overshoots.
\red{The use of slope limiting in the WENO-L version ensures
the IDP property without having any negative impact
on the accuracy of the results.}

\begin{figure}[h!]
  \small
  
\begin{minipage}[t]{0.5\textwidth}
\centering (a)  LO, $u_h\in[0.785,10.885]$
\vskip0.2cm

\includegraphics[width=0.95\textwidth,trim=0 0 0 0,clip]{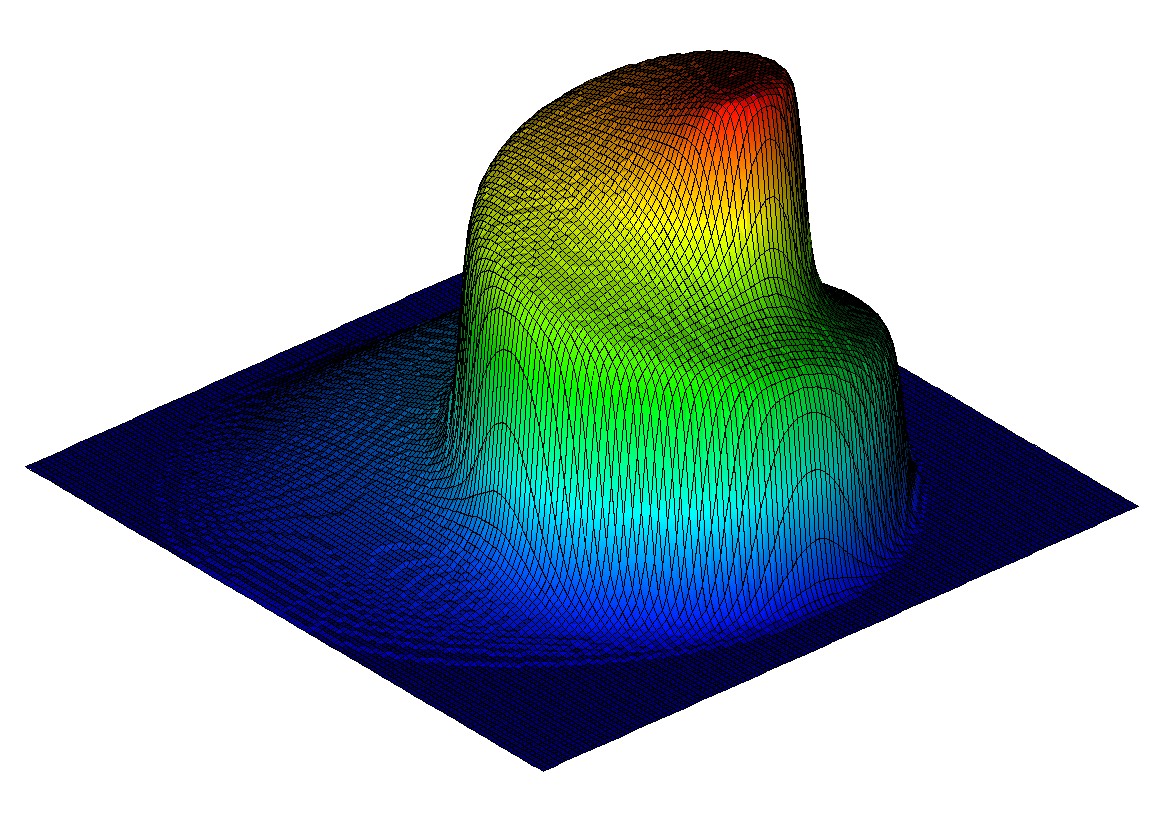}

\end{minipage}%
\begin{minipage}[t]{0.5\textwidth}

  \centering (b) HO, $u_h\in[-2.054,14.547]$
  \vskip0.2cm
  
\includegraphics[width=0.95\textwidth,trim=0 0 0 0,clip]{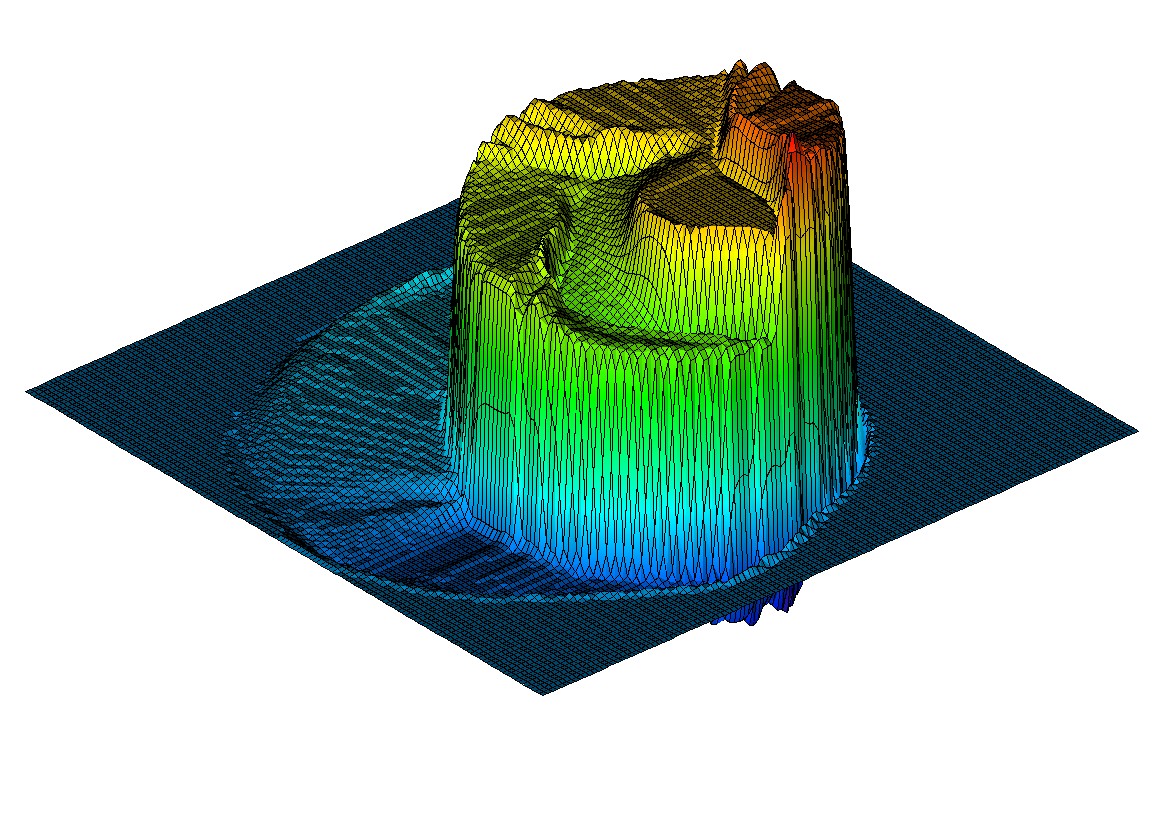}

\end{minipage}

\vskip0.35cm

\begin{minipage}{0.5\textwidth}
\centering (c) WENO, $q=3$, \vskip0.1cm

$u_h\in[0.785,10.997]$\vskip0.1cm

\includegraphics[width=0.95\textwidth,trim=0 0 0 0,clip]{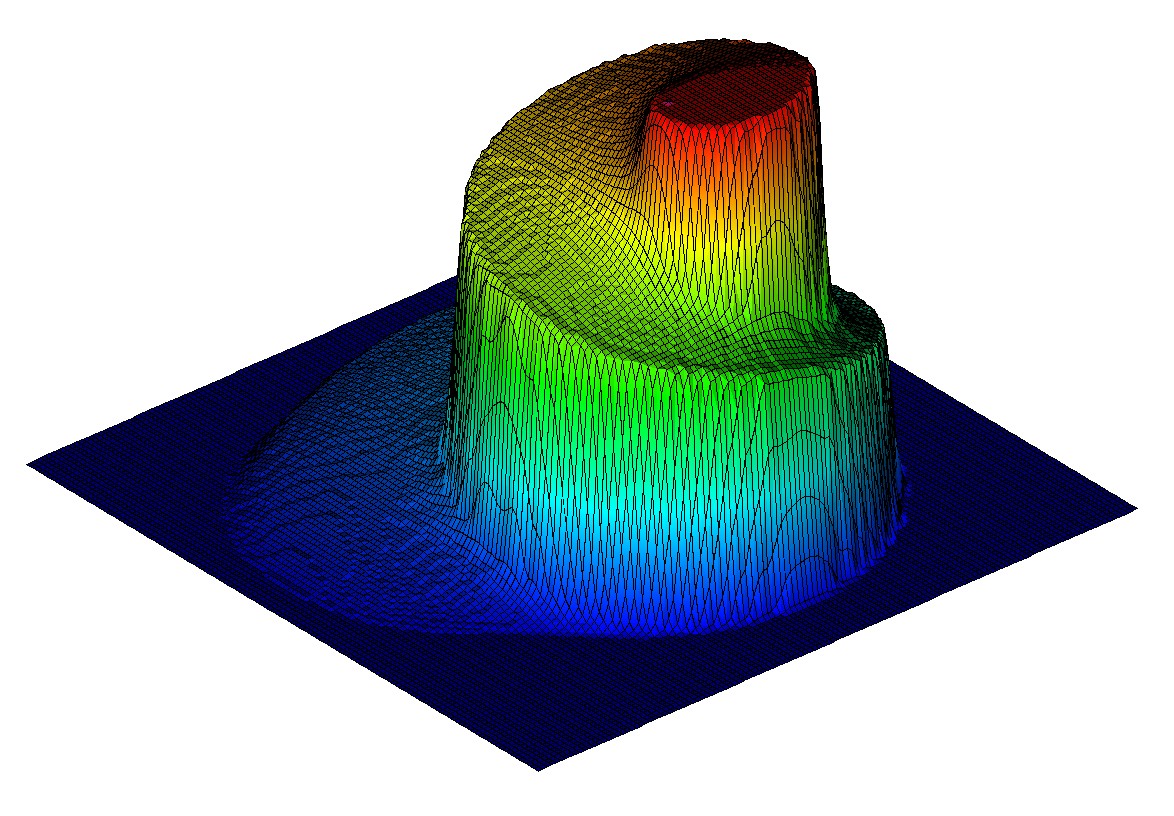}

\end{minipage}%
\begin{minipage}{0.5\textwidth}

\centering (d) WENO, $q=5$,\vskip0.1cm

$u_h\in[0.781,11.006]$ \vskip0.1cm

\includegraphics[width=0.95\textwidth,trim=0 0 0 0,clip]{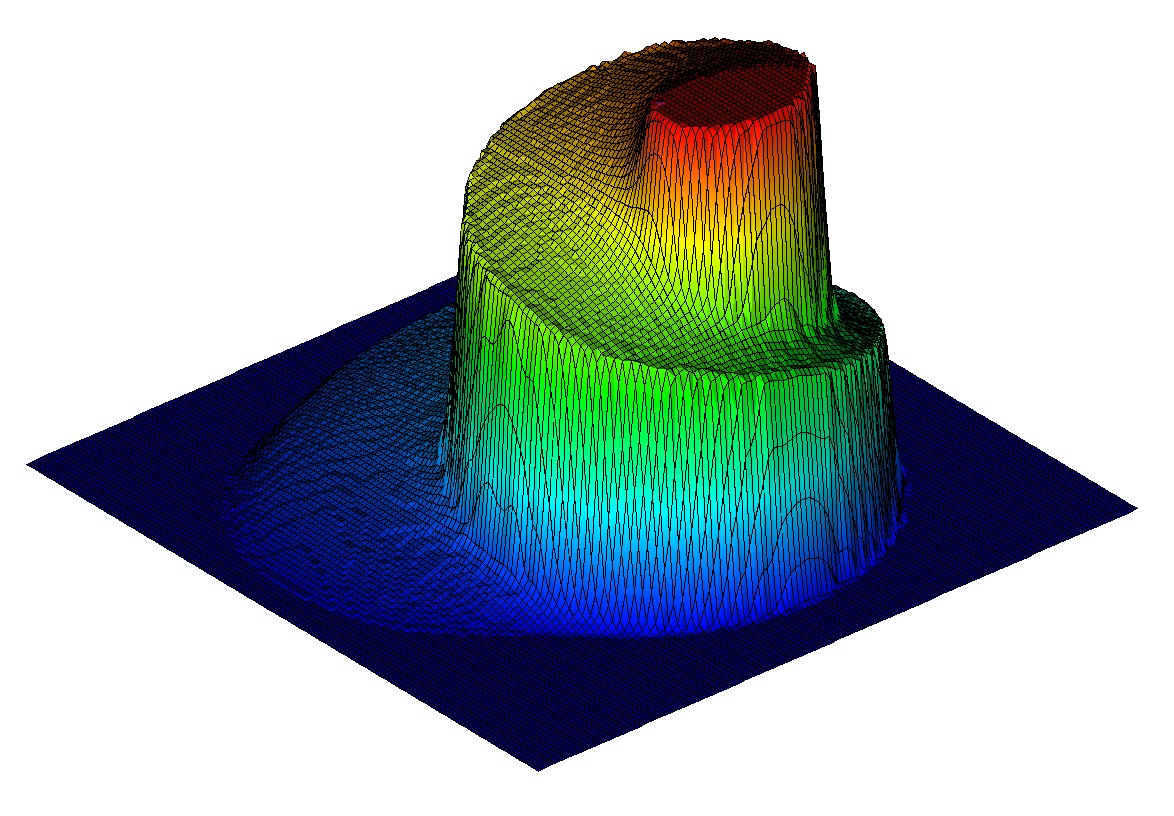}

\end{minipage}

\vskip0.35cm

\begin{minipage}{0.5\textwidth}
\centering (e) WENO, $q=10$, \vskip0.1cm

$u_h\in[0.759,11.050]$ \vskip0.1cm

\includegraphics[width=0.95\textwidth,trim=0 0 0 0,clip]{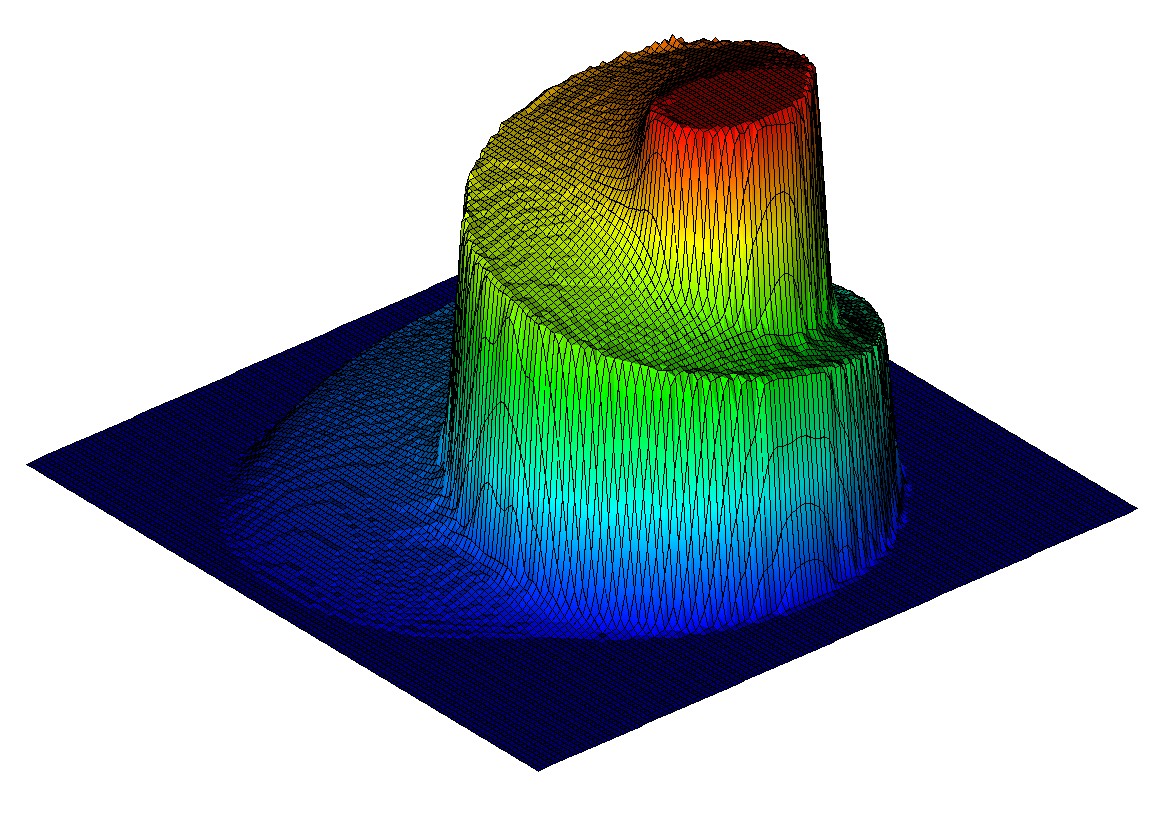}

\end{minipage}%
\begin{minipage}{0.5\textwidth}

\centering (f) WENO-L, $q=10$,\vskip0.1cm

$u_h\in[0.785,10.996]$ \vskip0.1cm

\includegraphics[width=0.95\textwidth,trim=0 0 0 0,clip]{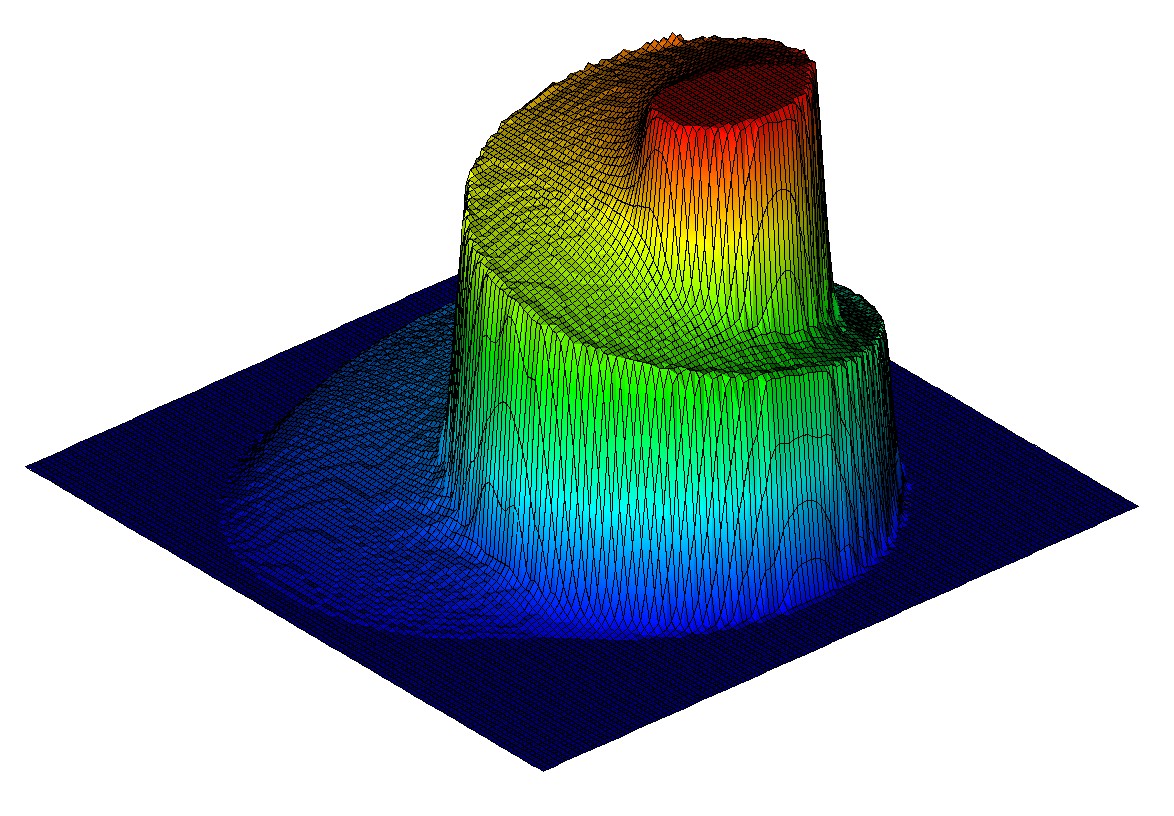}

\end{minipage}

\caption{Numerical solutions to the KPP problem \cite{kurganov2007} at $t=1$ calculated using a uniform mesh, $\mathbb Q_1$ elements, $h=\frac{1}{128}$ and $\Delta t=10^{-3}$. }

\label{kpp-pics}

\end{figure}

\blue{
\subsection{Euler equations in 2D: double Mach reflection}
\label{sec:Euler2Ddmr}

Finally, we assess the performance of our scheme for multidimensional systems by running a 2D test for the compressible Euler equations. We restrict ourselves to computations with the WENO-L version using the full element contributions \eqref{aec}. In contrast to the two-dimensional scalar benchmarks of Sec.~\ref{sec:scalar2D}, we employ the default linear weights in the WENO reconstruction, namely $w_l^{e,\rm lin}=10^{-3}$ for $l=1,\ldots,n_e$ and $w_0^{e,\rm lin}=1-\sum_{l=1}^{n_e}w_l^{e,\rm lin}$. The steepening parameter of the smoothness sensor \eqref{WENO:alpha} is set to $q=1$.

As a benchmark problem, we consider the double Mach reflection test introduced by Woodward and Colella \cite{woodward1984}, which can be regarded as a two-dimensional extension of the blast-wave problem from Sec.~\ref{sec:blast}. It represents a challenging test for multidimensional shock-capturing schemes because excessive diffusion smears the fine-scale features while insufficient stabilization leads to spurious oscillations. The computational domain is $\Omega=(0,4)\times(0,1)$. The boundary consists of a reflecting wall $\Gamma_w=\{(x,0)^\top\in\partial \Omega:1/6\le x \le 4\}$, a supersonic outlet $\Gamma_{\rm out}=\{(4,y)^\top:0<y\le 1\}$, and a supersonic inlet $\Gamma_{\rm in}=\partial \Omega \setminus (\Gamma_w \cup \Gamma_{\rm out})$. The prescribed post-shock and pre-shock states are 
\begin{align*}
	(\varrho_L,v_{x,L},v_{y,L},p_L)&=(8.0,8.25\cos(30^\circ),-8.25\sin(30^\circ),116.5), \\ (\varrho_R,v_{x,R},v_{y,R},p_R)&=(1.4,0.0,0.0,1.0),
\end{align*}
respectively. Initially, the post-shock state is imposed in $\Omega_L=\{(x,y)\;|\;x<\frac{1}{6}+\frac{y}{\sqrt{3}}\}$, and the pre-shock state in $\Omega_R = \Omega \setminus \Omega_L$. To account for the motion of the Mach $10$ shock along the top boundary, time-dependent inflow conditions are prescribed by imposing the post-shock state for $x<\frac{1}{6}+\frac{1+20t}{\sqrt{3}}$ and the pre-shock state elsewhere. The shock-wall interaction produces a Mach $10$ shock impinging on the reflecting wall at an angle of $60^{\circ}$. 

In all simulations, the WENO-L method employs the pressure limiter from \cite{kuzmin2023,kuzmin2020c}, as described in Appendix~A. The pressure fix proposed in \cite{abgrall-arxiv, wissocq2025} produces virtually identical results (not shown here) in this test. No visible differences were observed between the two versions since the pressure correction is activated only in a very small number of cells.

The simulations are run up to the final time $t=0.2$. As DG methods are widely used for challenging hyperbolic problems, we compare our approach with the DG-WENO scheme from \cite{vedral-arxiv} equipped with the classical Zhang--Shu slope limiter \cite{zhang2010b}. For a fair comparison, both schemes use roughly the same number of degrees of freedom per variable, with only minor differences related to the boundary discretization. 

The numerical results obtained with linear and quadratic finite elements are shown in Fig.~\ref{fig:dmr-pics}. Remarkably, the quality of the WENO-L approximation is comparable to that of the DG-WENO solution, and the results remain free of spurious oscillations. This observation demonstrates the competitiveness of continuous finite element approaches for demanding compressible flow problems and may encourage further developments in this direction. 


\begin{figure}[h!]
	\small
	
	\begin{minipage}[t]{\textwidth}
		\centering \blue{(a)  WENO-L, $\mathbb{Q}_1$ elements, $N_h=591745$, $u_h=[1.238,22.415]$}
		\vskip0.2cm
		
		\includegraphics[width=0.95\textwidth,trim=0 0 0 0,clip]{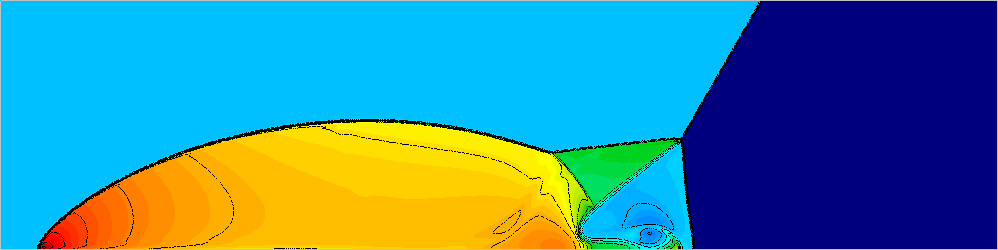}

	\end{minipage}%
	
	\vskip0.35cm
	
	\begin{minipage}{\textwidth}
		\centering \blue{(b) DG-WENO, $\mathbb{Q}_1$ elements, $N_h=589824$, $u_h=[1.400,22.116]$}
		
		\includegraphics[width=0.95\textwidth,trim=0 0 0 0,clip]{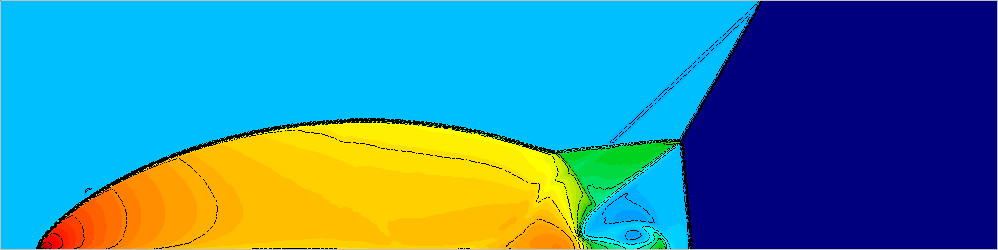}

	\end{minipage}%

	\vskip0.35cm
	
	\begin{minipage}{\textwidth}
		\centering \blue{(c) WENO-L, $\mathbb{Q}_2$ elements, $N_h=591745$, $u_h=[1.338,22.517]$}

		\includegraphics[width=0.95\textwidth,trim=0 0 0 0,clip]{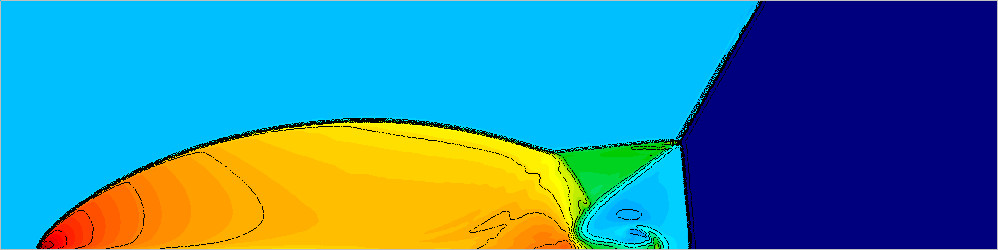}

	\end{minipage}%
	
	\caption{\blue{Density profiles for the double Mach reflection problem \cite{woodward1984} at $t=0.2$ calculated using a uniform mesh.}}
	
	\label{fig:dmr-pics}
	
\end{figure}
}

\section{Conclusions}
\label{sec:conclusions}

In this paper, we have shown that continuous Galerkin methods can be
limited in much the same way as their DG counterparts. Our definition of
the intermediate cell averages leads to a fully multidimensional generalization
of the bar states that represent averaged exact solutions of projected Riemann
problems. The new convex limiting framework enables us to ensure preservation
of invariant domains and high-order accuracy
without using matrix-based graph viscosity and decomposing
antidiffusive element contributions into subcell fluxes. Spurious oscillations
are avoided by using dissipative WENO stabilization instead of subcell flux
limiters with local bounds. The coarse-cell CFL condition for the flux-limited
version of our CG scheme is independent of the basis and weaker
than the subcell CFL condition that guarantees the IDP property for Bernstein
finite elements in the absence of flux limiting. In slope-limited
high-order extensions, the CFL condition that ensures nonlinear stability
of the underlying target scheme should be taken into account.

Even more significant 
efficiency gains can presumably be achieved for CG-LGL finite element
schemes, because the nonuniform distribution of nodes within the
cell makes the subcell CFL condition far more restrictive than that
for the uniform submesh of Bernstein nodes \cite{hajduk2025a}. Furthermore,
the imposition of global bounds on LGL nodal states is less
restrictive than the corresponding inequality constraints for Bernstein
nodes. Hence, it is less likely that convex limiting will degrade the rates of
convergence to smooth solutions. Adaptation of the proposed
algorithms to the LGL setting is straightforward. Regardless
of the basis, the performance of practical implementations
can be greatly improved by exploiting the matrix-free nature
of the presented methods and the data locality of high-order
finite element approximation \cite{andrej2024,bello2020,kronbichler2019}.

\section*{Acknowledgments}

The work of Dmitri Kuzmin was supported by the German Research Foundation (DFG) within the framework of the priority research program SPP 2410 under grant KU 1530/30-1. Hennes Hajduk participated in this project in his capacity as an associate member of  SPP 2410.

\bibliographystyle{abbrv}
\bibliography{references}

\red{
\section*{Appendix A: IDP limiter for the Euler equations}

Let $f=(f_1,\ldots,f_N)$ be an array of $N=2$ antidiffusive fluxes
$(f_{ee'},f_{e'e})$ or $N=|\Ne|$ element contributions $(f_j^e,\ j\in\Ne)$.
In both cases, the components of $f$ possess the zero sum property
$\sum_{i=1}^Nf_i=0$ and modify intermediate cell averages or nodal
states $\bar u_1,\ldots,\bar u_N\in\mathcal G$ as follows:
\[
\bar u_i^*=\bar u_i+\frac{\alpha f_i}{m_i}\in\mathcal G,
\qquad \alpha=\min_{1\le i\le N}\alpha_i.
\]
The scaling factors $m_i>0$ correspond to $|S_{ee'}|$ in
the flux limiter version and to $m_i^e$ in the slope limiter
version. The choice of the upper bounds $\alpha_i\in[0,1]$
for $\alpha\in [0,1]$ should guarantee that
$\bar u_i^*\in\mathcal G$ for $i=1,\ldots,N$. We emphasize that $i$
is not the global index of a node or element but the local
index of an antidiffusive component $f_i$ to be limited using
the correction factor $\alpha$.

For the Euler equations, $\bar u_i$ and $f_i$ are
vectors with $d+2$ components:
\[\bar u_i=(\rho_{i},(\rho \mathbf v)_{i},(\rho E)_{i})^\top,\qquad
f_{i}=( f_{i}^{\rho},\mathbf{f}_{i}^{\rho\mathbf{v}},
 f_{i}^{\rho E})^\top.
 \]
 Positivity preservation for the density $\bar\rho_i^*$
 can easily be enforced using
 \[
\alpha_i^\rho=
\begin{cases}
  -\frac{m_i\bar\rho_i}{f_i^\rho}
  & \mbox{if}\ m_i\bar\rho_i+f_i^\rho<0,\\
  1 & \mbox{otherwise}.
  \end{cases}
\]
The pressure $p(\bar u_i^*)$ and internal energy $e(\bar u_i^*)$
are nonnegative if $$\overline{(\rho E)}_{i}^*\ge
\frac{|\overline{(\rho\mathbf{v})}_{i}^{*}|^2}{2\bar\rho_{i}^*}.$$
The corresponding inequality constraint for $\alpha$
can be written as (cf. \cite{kuzmin2020,kuzmin2020c})
\[
P_{i}(\alpha)\ge Q_{i}=m_i\left(
\frac{|\overline{(\rho\mathbf{v})}_i|^2}{2}-\bar\rho_i
\overline{(\rho E)}_i\right),
\]
where
\begin{align*}
  P_{i}(\alpha)&=\alpha \left(\bar\rho_i f_{i}^{\rho E}
  +\overline{(\rho E)}_{i} f_{i}^{\rho}-
  \overline{(\rho\mathbf{v})}_i\cdot{\mathbf{f}}_{i}^{\rho\mathbf{v}}
\right)
+\frac{\alpha^2}{m_i}\left( f _{i}^{\rho E} f_{i}^{\rho}
-\frac{|{\mathbf{f}}_{i}^{\rho\mathbf{v}}|^2}{2}\right).
\end{align*}
Note that $Q_{i}\le 0$ because
$\bar u_i\in\mathcal G$, where $\mathcal G\subseteq\{u\,:\,
\rho(u)>0,\ e(u)>0\}$.

We can now derive a bound $\alpha_i^p$ for $\alpha$ following
\cite{kuzmin2020,kuzmin2020c} and \cite[Examples 3.39, 4.16, 5.14]{kuzmin2023}.
For any $\alpha\in[0,1]$, we have 
$P_{i}(\alpha)\ge \alpha R_{i}$, where
 \begin{align*}
    R_{i}&=\bar\rho_i f_{i}^{\rho E}
  +\overline{(\rho E)}_{i} f_{i}^{\rho}-
   \overline{(\rho\mathbf{v})}_i\cdot{\mathbf{f}}_{i}^{\rho\mathbf{v}}
 +\frac{1}{m_i}\min\Big\{0,
 f_{i}^{\rho E} f_{i}^{\rho}
-\frac{|{\mathbf{f}}_{i}^{\rho\mathbf{v}}|^2}{2}\Big\}.
 \end{align*}
 To ensure that $\alpha R_{i}\ge Q_{i}$ for 
$i\in\{1,\ldots,N\}$, the pressure limiter must use
\[
\alpha_i^p=
\begin{cases}
  \frac{Q_{i}}{R_{i}}
  & \mbox{if}\ R_{i}<Q_{i},\\
  1 & \mbox{otherwise}.
  \end{cases}
\]
We remark that this formula is very similar to the
one derived in \cite{abgrall-arxiv,wissocq2025}
using the general framework of geometric quasi-linearization
\cite{wu2023}. Both approaches replace the nonlinear pressure
constraint
by linear sufficient conditions.
}

\end{document}